\documentclass{article}
\usepackage{psfrag}
\usepackage[parfill]{parskip}
\usepackage{float, graphicx}
\usepackage[]{epsfig}
\usepackage{amsmath, amsthm, amssymb}
\usepackage{epsfig}
\usepackage{verbatim}
\usepackage{multicol}
\usepackage{url}
\usepackage{latexsym}
\usepackage{mathrsfs}
\usepackage[colorlinks, bookmarks=true]{hyperref}
\usepackage{graphicx}
\usepackage{amsmath}
\usepackage{enumerate}
\usepackage[normalem]{ulem}
\usepackage{bm}
\usepackage{dsfont}
\usepackage{stmaryrd}
\usepackage{tikz-cd}
\usepackage[T1]{fontenc}
\usepackage{authblk}

\usepackage{cite}

\usepackage{color}

\makeatletter
\def\thm@space@setup{%
  \thm@preskip=\parskip \thm@postskip=0pt
}
\makeatother

\numberwithin{equation}{section}

\renewcommand{\cal}{\mathcal}
\newcommand\cA{{\mathcal A}}

\newcommand{\cC}{{\cal C}}
\newcommand{\cD}{\cal D}
\newcommand{\cE}{{\cal E}}

\newcommand\cH{{\mathcal H}}

\newcommand{\cL}{{\cal L}}

\newcommand{\fa}{{\frak a}}
\newcommand{\fb}{{\frak b}}
\newcommand{\fc}{{\frak c}}

\newcommand{\fe}{{\frak e}}
\newcommand{\fr}{{\frak r}}

\newcommand{\rd}{{\rm d}}

\newcommand{\ri}{\mathrm{i}}

\newcommand{\bB}{{\mathbb B}}
\newcommand{\bC}{{\mathbb C}}

\newcommand{\bE}{\mathbb{E}}

\newcommand{\bP}{\mathbb{P}}

\newcommand{\bR}{{\mathbb R}}

\newcommand{\bZ}{\mathbb{Z}}

\newcommand{\la}{\lambda}

\newcommand{\om}{{\omega}}

\newcommand{\veps}{\varepsilon}

\DeclareMathOperator{\supp}{supp}
\DeclareMathOperator{\dist}{dist}

\DeclareMathOperator{\dom}{\mathcal{D}}

\DeclareMathOperator{\OO}{O}
\DeclareMathOperator{\oo}{o}

\renewcommand{\Re}{\mathop{\mathrm{Re}}}
\renewcommand{\Im}{\mathop{\mathrm{Im}}}

\newcommand{\deq}{\mathrel{\mathop:}=} %define :=
\newcommand{\eqd}{=\mathrel{\mathop:}}
\renewcommand{\leq}{\leqslant}
\renewcommand{\le}{\leqslant}
\renewcommand{\geq}{\geqslant}

\newcommand{\cor}{\color{red}}
\newcommand{\cob}{\color{blue}}

\newcommand{\td}{\tilde}
\newcommand{\del}{\partial}

\newcommand{\qq}[1]{[\![{#1}]\!]}

\newcommand{\beq}{\begin{equation}}
\newcommand{\eeq}{\end{equation}}

\theoremstyle{plain} %plain, definition, remark
\newtheorem{theorem}{Theorem}[section]
\newtheorem*{theorem*}{Theorem}
\newtheorem{lemma}[theorem]{Lemma}
\newtheorem*{lemma*}{Lemma}
\newtheorem{corollary}[theorem]{Corollary}
\newtheorem*{corollary*}{Corollary}
\newtheorem{proposition}[theorem]{Proposition}
\newtheorem*{proposition*}{Proposition}
\newtheorem{assumption}[theorem]{Assumption}
\newtheorem*{assumption*}{Assumption}

\newtheorem{definition}[theorem]{Definition}
\newtheorem*{definition*}{Definition}

\newtheorem*{example*}{Example}
\newtheorem{remark}[theorem]{Remark}

\newtheorem*{remark*}{Remark}
\newtheorem*{remarks*}{Remarks}

\def\author#1{\par
    {\centering{\authorfont#1}\par\vspace*{0.05in}}
}

\def\titlefont{\fontsize{13}{15}\bfseries\boldmath\selectfont\centering{}}
\def\authorfont{\fontsize{13}{15}}

\let\affiliationfont\rhfont

\def\address#1{\par
    {\centering{\affiliationfont#1\par}}\par\vspace*{11pt}
}

\def\body{
\setcounter{footnote}{0}
\def\thefootnote{\alph{footnote}}
\def\@makefnmark{{$^{\rm \@thefnmark}$}}
}

\def\title#1{
    \thispagestyle{plain}
    \vspace*{-14pt}
    \vskip 79pt
    {\centering{\titlefont #1\par}}%
    \vskip 1em
}
\newcommand{\bmla}{\bm{\lambda}}

\newcommand{\cov}{{\rm{cov}}}

\newcommand{\tilz}{\tilde z}

\begin{document}

\title{Dyson Brownian Motion for General $\beta$ and Potential at the Edge}

\vspace{1.2cm}

\noindent \begin{minipage}[c]{0.5\textwidth}
 \author{Arka Adhikari}
\address{Harvard University\\
   E-mail: adhikari@math.harvard.edu}
 \end{minipage}
 \begin{minipage}[c]{0.5\textwidth}
 \author{Jiaoyang Huang}
\address{Harvard University\\
   E-mail: jiaoyang@math.harvard.edu}
 \end{minipage}
 %

%\author{Arka Adhikari \ Jiaoyang Huang \ Horng-Tzer Yau}
\begin{abstract}
In this paper, we compare the solutions of Dyson Brownian motion with general $\beta$ and potential $V$ and the associated McKean-Vlasov equation near the edge. Under suitable conditions on the initial data and potential $V$, we obtain the optimal rigidity estimates of particle locations near the edge for short time $t=\oo(1)$. Our argument uses the method of characteristics along with a careful estimate involving an equation of the edge.
With the rigidity estimates as an input,   we prove a central limit theorem for mesoscopic statistics near the edge which, as far as we know, have been done for the first time in this paper. Additionally, combining with \cite{LandonEdge}, our rigidity estimates are used to give a proof of the local ergodicity of Dyson Brownian motion for general $\beta$ and potential at the edge, i.e. the distribution of extreme particles converges to Tracy-Widom $\beta$ distribution in short time. 
\end{abstract}
\section{Introduction}

Random Matrix models were originally suggested by Wigner \cite{MR0083848,MR0077805} to model the nuclei of heavy atoms. 
%as an approximate Hamiltonian for interacting many particle systems. 
The models he originally studied, the Gaussian orthogonal/unitary ensembles were successful in describing the spacing distribution between energy levels. Wigner conjectured that general random matrices will have this same spacing distribution as long as they are in the same symmetry class.

Later, in 1962 \cite{MR0148397}, Dyson interpreted the Gaussian orthogonal/unitary ensembles as dynamical limit of the matrix valued Brownian motion, which is given by
\begin{align}\label{e:MatrixDBM}
   \rd H(t)=\rd B(t)-\frac{1}{2}H(t)\rd t,
\end{align}
where $B(t)$ is the Brownian motion on real symmetric/complex Hermitian matrices. It turns out the eigenvalues of the above matrix valued Brownian motion satisfy a system of stochastic differential equations. These equations have been later generalized to stochastic differential equations, 
called the $\beta$-Dyson Brownian Motion with potential $V$,
\begin{equation}  \label{DBM}
    d \lambda_i(t) = \sqrt{\frac{2}{\beta N}} dB_i(t) + \sum_{i=1}^{N} \frac{dt}{\lambda_i(t) -\lambda_j(t)} - \frac{1}{2} V'(\lambda_i(t)) dt,\quad 1\leq i\leq N,
\end{equation}
where the initial data $\{\la_1(0),\la_2(0),\cdots,\lambda_N(0)\}$ lies in the closer of the Weyl chamber %$(\lambda_1,\cdots\lambda_N)$ 
%lying in the Weyl chamber ,$\triangle_N$
\begin{equation}
   \triangle_N :=\{(x_1,x_2,\cdots,x_N): x_1<x_2\cdots<x_N\}.
\end{equation}
The real symmetric and complex Hermitian matrix valued Brownian motion corresponds to \eqref{DBM} with $\beta=1$ and $\beta=2$ respectively, and quadratic potential $V=x^2/2$.

 Dyson suggested that on times of order $O(1/N)$ one would get equilibrium in the microscopic statistics by evolving a random system stochastically to one of the standard Gaussian matrix models depending on the symmetry class. In fact, one has a dichotomy of three time scales
\begin{enumerate}
    \item For time $t \gg 1$ one should get the global equilibrium, e.g., the global spectral density should approach that for the corresponding Gaussian ensembles. For Dyson Brownian motion with general $\beta$ and potential $V$, this was studied in \cite{GDBM1}.
    \item On scales of order $N^{-1}\ll\eta^* \ll 1 $, one should reach the equilibrium after running Dyson Brownian motion for time $t \gg \eta^*$. Namely, mesoscopic quantities of the form $ \sum_{i=1}^{N} f((\lambda_i -E)/
    \eta^*)$ for appropriate test functions should be universal.
    \item For the microscopic scale, i.e. the scale of order $O(1/N)$, and $\beta=2$, the microscopic eigenvalue distribution should be the same as that of the determinantal point process with the Sine kernel, $K(x,y) = \sin(x-y)/(x-y)$, provided one runs Dyson Brownian motion for $t\gg 1/N$. %\cite{MR0143556,MR0143557,MR0143558}.
\end{enumerate}

The understanding of the local ergodicity of Dyson Brownian motion, i.e. the fact that the local statistics of Dyson Brownian motion reaches an equilibrium in short time, plays an important role in the proof of Wigner's original universality conjecture by Erd{\H o}s, Schlein and Yau \cite{MR2919197}. Their methods to prove universality for matrix models first involve proving a rigidity estimates of the eigenvalues, i.e. the eigenvalues are close to their classical locations, up to an optimal scale. This is the initial data for a Dyson Brownian motion input which interpolates the initial model to the Gaussian orthogonal/unitary ensembles. The second step is to show that Dyson Brownian motion reaches an equilibrium in a short time for local statistics. Since Dyson Brownian motion needs only to be run in short time, then the initial and final models can be compared. The last step compares the original random matrices with ones with small Gaussian component. For a good review about the general framework regarding this type of analysis, one can read the book by Erd{\H o}s and Yau \cite{YauErdosRMT}.

Of the three steps described in the previous section, the step that is the least robust is the proof of the rigidity estimates. This part is very model particular and, depending on the model in question, requires significant effort in trying to prove optimal estimates. Even in the most basic case of Wigner matrices, the concentration of the trace of the resolvent would require very precise cancellation in the form of what is known as the \emph{fluctuating averaging lemma} \cite{MR2871147}.  The proof of this type of cancellation uses very delicate combinatorial expansions involving iterated applications of the resolvent identity. For models even more complicated than the Wigner matrices, such lemmas are an intricate effort. %One can find an exposition of this difficulties in the book of Erdos and Yau \cite{YauErdosRMT}. 
A more general method that does not involve delving into the particulars of a model would be desirable; then we would be able to treat a general class of models uniformly.

A dynamical approach to proving rigidity using Dyson Brownian motion allows us to avoid technical issues relating to the particulars of a matrix model. This would allow us to avoid complicated combinatorial analysis and, in addition, allow us to treat models that  do not occur naturally with an associated matrix structure, such as the $\beta$-ensembles. In an earlier paper by B.Landon and the second author \cite{HL}, they proved the rigidity estimates for the bulk eigenvalues of Dyson Brownian motion. As a result, the optimal rigidity estimates are purely a consequence of the dynamics. The proof of rigidity is based on a comparison between the empirical eigenvalue process of Dyson Brownian motion and the deterministic measure valued process obtained as the solution of the associated McKean-Vlasov equation by using the method of characteristics. The difference in the corresponding Stieltjes transforms can be analyzed by estimates of Gronwall type.

There are substantial difficulties involved in performing a comparison between the solutions of Dyson Brownian motion and the associated McKean-Vlasov equation near the edge. In the bulk, one can derive sufficiently strong estimates by looking at the distance from the characteristics to the real line; this is thanks to the fact that we have strong bounds on the imaginary part of the Stieltjes transform in the bulk. Near the spectral edge, the power of these bounds decay and become too weak to prove optimal rigidity. In our case, we have to establish an equation determining the relative movement of our characteristics to the edge. The estimates of the Stieltjes transform of the empirical particle density near the edge heavily depend on this relative movement. The equation for the edge allows us to explicitly understand how the eigenvalues move from their initial position to the optimal region.

\begin{comment}
We also observe that since the second step of the two part strategy mentioned earlier also involves a running of the DBM once rigidity is established, we see that we can essentially combine the two steps into one and establish the fact that universality is a process that is , in principle, merely a consequence of running the DBM. Near the edge, this establishes Dyson's original vision of the universality of microscopic eigenvalue spacings through  Dyson Brownian Motion. 
\end{comment}

In addition to the rigidity estimates, another main innovation in this paper is the determination  of the correlation kernel for the Stieltjes transform of the empirical particle density of Dyson Brownian motion at mesoscopic scales near the edge. It allows us to prove a mesoscopic central limit theorem near the edge.
The mesoscopic central limit at the bulk for Wigner matrices was proven in \cite{MR1678012,MR1689027, mesoCLT1,mesoCLT3}, for $\beta$-ensemble in \cite{mesoCLT1} and for Dyson Brownian motion in \cite{HL,mesoCLTDBM,fix}.
As far as we know, the mesoscopic central limit theorem near the edge is new even for the Wigner matrices and $\beta$-ensembles. 
The dynamical method provides a unified approach to see how it emerges naturally, and allows us to see the universality of this correlation kernel. 
%Alternative methods for finding the correlation Kernel at mesoscopic scales involve difficult martingale estimates such as in \cite{mesoCLT1} or moment estimates \cite{mesoCLT3}. The dynamical method gives a simple method of computation and ,indeed, allows us to see the universality of this correlation kernel.

Combining with \cite{LandonEdge}, our rigidity estimates are used to give a proof of the local ergodicity of Dyson Brownian motion for general $\beta$ and potential at the edge, i.e. the distribution of extreme particles converges to Tracy-Widom $\beta$ distribution in short time. Our proof uses only the dynamics, and is independent of the matrix models. This is in alignment with Dyson's original vision on the nature of universality of the local eigenvalue statistics. A consequence of our edge universality result is a purely dynamical proof of the edge universality for $\beta$-ensembles with general potential.

\subsection{Related Results in the Literature}

Results for the McKean-Vlasov equation were first established by Chan\cite{MR1176727} and Rogers-Shi\cite{MR1217451}, who showed the existence of a solution for quadratic potentials $V$.
The McKean-Vlasov equation for general potentials $V$ was studied in detail in the works of Li, Li and Xie. In the works \cite{GDBM1} and \cite{GDBM2}, it was shown that under very weak conditions on $V$ the solution of the McKean-Vlasov equation will converge to an equilibrium distribution, that is dependent on the parameters $\beta$ and $V$ at times $t \gg 1$. The authors were able to interpret the time evolution under the McKean-Vlasov equation as a manner of gradient descent on the space of measures. This gives the complete description of Dyson Brownian motion at the macroscopic scale. 

For the microscopic scale, Dyson Brownian motion was studied in detail by Erdos, Yau and various coauthors across a multitude of papers \cite{MR3098073,MR2662426,MR2661171,MR2639734,MR2481753,MR2537522,MR2810797,MR3372074,MR2905803,MR2871147}. Specifically, from these works, it is known that for the classical ensembles $\beta=1,2, 4$ and quadratic potential, with the initial data given by the eigenvalues of a Wigner matrix, it is known that after $t \gg N^{-1}$ the local statistics of the particles are the same as those of the corresponding classical Gaussian ensembles. After this, the two works \cite{kevin3,Landon2016} established gap universality for the classical $\beta=1,2,4$ Dyson Brownian motion with general initial data, by using estimates established in a discrete DiGeorgi-Nash-Moser theorem in \cite{MR3372074}. Fixed Energy Universality required a sophisticated homogenization argument that allowed the comparison between the discrete equation and a continuous version; the results have been established in recent papers \cite{fixedBourgade,fix}. An extension of this interpolation at the edge was shown in \cite{LandonEdge}. These results were a key step in the proof of edge and bulk universality in various models. An alternative approach to Universality was shown, independently, in the works of Tao and Vu \cite{MR2784665}.

In the three-step strategy for proving universality, as developed by Erd{\H o}s, Yau and their collaborators, the first step is to derive a local law of eigenvalue density. This is a very technical and highly model dependent procedure. In the case of Wigner matrices, the proofs have been established in \cite{MR2481753,MR2537522,MR2981427,MR2871147,MR3109424}.
Local laws can be established for other matrix models in the bulk, such as the case of sparse random matrices \cite{MR3098073} and in deformed Wigner matrices \cite{MR3502606}. Establishing local laws near the edge are generally more involved; the case of correlated matrices was shown in \cite{ArkaZiliang,AltEdge1,ErdosEdge1}. Local laws for $\beta$-ensembles near the edge were considered in \cite{MR3253704} with the discrete analogue in \cite{HG}; the Wigner matrices were considered in \cite{MR3161313}.
\begin{comment}
Central Limit Theorems for mesoscopic statistics in the bulk were shown in a series of papers \cite{MR1678012,MR1689027,mesoCLT3,mesoCLT1}. These formulas have also been shown for various models in which one has either an explicit matrix model or a formula. With the Brezin-Hikami formula at $\beta=2$ at quadratic potential, mesoscopic linear statistics were shown in \cite{mesoCLTDBM}, where they showed convergence at scale $\eta$ for times $t\gg \eta$. Mesoscopic statistics were also shown for the classical $\beta=1,2,4$ ensembles in \cite{fix}. These proofs all require very particular strcuture of the family. A dyanamical approach to mesoscopic statistics was used in \cite{HL}; the estimates near the edge are more involved and, as mentioned before, have been performed for the first time here.
\end{comment}

\section{Background}

In this section, we will provide basic definitions and assumptions in our study of the $\beta$-Dyson Brownian motion and the associated McKean Vlasov equation. This section culminates in the analysis of solutions of the McKean-Vlasov equation via the method of characteristics and the proof of various important inequalities on the growth of the solution in time $t$ and the behavior of its characteristics $z_t(u)$. These bounds provide the basis for our later estimates on the edge rigidity of the $\beta$-Dyson Brownian motion near the edge. To make the argument clean, we make the following assumption on the potential $V$. We believe the main results in this paper hold for $V$ in $C^4$ as in \cite{HL}.

\begin{assumption}\label{a:asumpV}
We assume that the potential $V$ is an analytic function.%, and that there exists a constant $\fK\geq 0$ such that $\inf_{x\in \bR} V''(x)\geq -2\fK$.
\end{assumption}

We denote $M_1(\bR)$ as the space of probability measures on $\bR$ and equip this space with the weak topology.  We fix a sufficiently small time $T>0$ and denote by $C([0,T], M_1(\bR))$ the space of continuous processes on $[0,T]$ taking values in $M_1(\bR)$.  It follows from \cite{GDBM1} that for all $\beta\geq 1$ and initial data $\bm\la(0)\in \overline{\Delta_N}$, there exists a strong solution $(\bm \la(t))_{0\leq t\leq T}\in C([0,T],\overline{\Delta_N})$ to the stochastic differential equation \eqref{DBM}. %For any $t>0$, $\bm\la(t)\in \Delta_N$ and $\bm \la(t)$ is a continuous function of $\bm\la(0)$.

We recall the following estimates on the locations of extreme particles of $\beta$-Dyson Brownian motion from \cite[Proposition 2.5]{HL}.
\begin{proposition}\label{normbound}
Suppose $V$ satisfies Assumption \ref{a:asumpV}. Let $\beta\geq 1$, and $\bm \la(0)\in \overline{\Delta_N}$. Let $\fa$ be a constant  such that the initial data $\|\bm \la(0)\|_\infty\leq \fa$.  Then for a sufficiently small time $T>0$, there exists a finite constant $\fb=\fb(\fa, T )$, such that for any $0\leq t\leq T$, the unique strong solution of \eqref{DBM} satisfies:
\beq\label{e:normbound}
\bP(\max\{|\la_1(t)|,|\la_N(t)|\}\geq \fb)\leq e^{-N}.
\eeq
\end{proposition}

%By quasi-analytic extension, and multiplying a cut-off function, we can assume that $V'(x)$ extends to the whole complex plane, and $V'(x+\ri y)$ is quasi-analytic,
%{\cor conditions we need}, and supported on a compact set.
\begin{comment}
{\cor
We denote, 
\beq
\del_z=\frac{1}{2}(\del_x-\ri \del_y),\quad \del_{\bar z}=\frac{1}{2}(\del_x+\ri \del_y).
\eeq
}
\end{comment}
Given a probability measure $\hat \rho_0$, we define the measure-valued process $\{\hat \rho_t\}_{t\geq 0}$ as the solution of the following equation
%\begin{align}\label{eq:dm0}\begin{split}
%\del_t m_t(z)=\left(m_t(z)+\frac{V'(z)}{2}\right)\del_z \left(m_t(z)+\frac{V'(z)}{2}\right)-\frac{V'(z)\del_z V'(z)}{4}+\int_{\bR}g(z,x) \rd \rho_t(x)
%\end{split}
%\end{align}
\begin{align}\label{eq:dm0}\begin{split}
\del_t \hat{m}_t(z)=\del_z \hat{m}_t(z)\left(\hat{m}_t(z)+\frac{V'(z)}{2}\right)+\frac{\hat{m}_t(z) V''(z)}{2}+\int_{\bR}g(z,x) \rd \hat{\rho}_t(x),
\end{split}
\end{align}
where
\beq \label{def:gzx}
g(z,x)\deq\frac{V'(x)-V'(z)-(x-z)V''(z)}{2(x-z)^2},\quad g(x,x)\deq\frac{V'''(x)}{4}.
\eeq
It is easy to see that for any fixed $z$, $g(z,x)$ is analytic in $\bC$ as a function of $x$; for any fixed $x$, $g(z,x)$ is analytic in $\bC$ as a function of $z$.

We analyze \eqref{eq:dm0} by the method of characteristics. Let
\beq \label{def:zt}
\del_t z_t(u)=-\hat m_t(z_t(u))-\frac{V'(z_t(u))}{2},\qquad z_0=u\in \bC_+,
\eeq
If the context is clear, we omit the parameter $u$, i.e., we simply write $z_t$ instead of $z_t(u)$. Plugging \eqref{def:zt} into \eqref{eq:dm0}, and applying the chain rule we obtain
%\beq \label{e:mtzt}
%\del_t \left(m_t(z_t)+\frac{V'(z_t)}{2}\right)=-\frac{V'(z_t)\del_z V'(z_t)}{4}+\frac{1}{\pi}\int_{\bC} \del_{\bar w} \td g(z_t,w) m_t(w)\rd^2 w.
%\eeq
\beq \label{e:mtzt}
\del_t \hat m_t(z_t(u))=\frac{\hat m_t(z_t(u)) V''(z_t(u))}{2}+\int_{\bR} g(z_t(u),x) \rd \rho_t(x).
\eeq
The behaviors of $z_s$ and $\hat{m}_s(z_s)$ are governed by the system of equations \eqref{def:zt} and \eqref{e:mtzt}.

{ As a consequence of Proposition \ref{normbound}, if the probability measure $\hat\rho_0$ is supported on $[-\fa, \fa]$, then there exists a finite constant $\fb=\fb(\fa, T)$, such that $\hat\rho_t$ are supported on $[-\fb, \fb]$ for $0\leq t\leq T$. We fix a large constant $\fr$. If $z_t(u)\in \bB_{2\fr}(0)$ and $\dist(z_t(u),[-\fb,\fb])\geq 1$, then 
\begin{align}\label{def:fr}
    |\del_t z_t|\leq 1+\frac{1}{2}\sup_{z\in \bB_{2\fr}(0)}|V'(z)|.
\end{align}
Therefore, for any $u\in \bB_\fr(0)$, we have $z_t(u)\in \bB_{2\fr}(0)$ for any $0\leq t\leq T$, provided $T$ is small enough.
}

We also frequently use the following estimates studying the imaginary part of characteristics. They were proven in \cite[Proposition 2.7]{HL}.
\begin{proposition} \label{prop:ImEst}
Suppose $V$ satisfies assumption \ref{a:asumpV}. Let $\beta\geq 1$, and $\bm \la(0)\in \overline{\Delta_N}$. Fix large constant $\fr>0$ Then for a sufficiently small time $T>0$, there exist constants depending on potential $V$ and $\fa$, such that the following holds. Fix any $0\leq s\leq t\leq T$ with $u\in \bB_{\fr}(0)$ and $\Im[z_t(u)]>0$, 
\begin{align}\begin{split}
%    & \Im[m_t(z)]\Im[z] \leq 1, |\partial_z m_t(z)|\leq \frac{\Im[m_t[z]]}{\Im[z]}\\
    & e^{-C(t-s)} \Im[z_t] \leq \Im[z_s]\\
    & e^{-(t-s)C}\Im[\hat m_t(z_t)] \leq \Im[\hat m_s(z_s)] \leq e^{(t-s)C} \Im[\hat m_s(z_s)] \\
    & e^{-C(t-s)}\left(\Im[z_t]+(t-s)\Im[\hat m_t(z_t)]\right) \leq \Im[z_s] \leq e^{C(t-s)}\left(\Im[z_t]+(t-s)\Im[\hat m_t(z_t)]\right)\\
    & e^{-C(t-s)}(\Im[z_t]\Im[\hat m_t(z_t)]+(t-s)\Im[\hat m_t(z_t)]^2) \leq \Im[z_0]\Im[\hat m_0(z_0)] 
\end{split}\end{align}
\end{proposition}

%The stieltjes transform $\hat m_t(z)$ is well defined on $\bC\setminus \supp(\hat \rho_t)$.
%The right edge $E_t$ is characterized by
%\begin{align}
%\del_z \hat m_t(E_t)=\infty,
%\end{align}
%and from Proposition \eqref{prop:invflow}, there exists a characteristic flow terminates at $E_t$.

\section{Square Root Behavior Measures}

In the earlier work \cite{HL},  the bulk rigidity of $\beta$-Dyson Brownian motion was proved via a comparison of the empirical density $\rho_t$ with $\tilde \rho_t$, the solution of the associated McKean-Vlasov equation with $\rho_0$ as initial data. This is not a good choice for studying the spectral edge. In most applications, we take $\rho_0$ to be the empirical eigenvalue density of a random matrix, which itself is random. As a consequence, the solution $\tilde\rho_t$ of the associated McKean-Vlasov equation with $\rho_0$, is again a random measure. Even if we have a good control on the difference between $\rho_t$ and $\tilde \rho_t$, it does not tell us the locations of the extreme eigenvalues, unless we have a very precise control of $\tilde \rho_t$. Unfortunately the edge universality asks exactly the locations of the extreme eigenvalues.

In order to circumvent this problem, we comparison the empirical density $\rho_t$ with $\hat \rho_t$, the solution of the associated McKean-Vlasov equation with initial data $\hat\rho_0$ close to $\rho_0$. In most applications, we take $\hat\rho_0$ to be either the semi-circle distribution,
\begin{align}\label{e:semicircle}
    \rho_{\rm sc}(x)=\frac{\sqrt{[4-x^2]_+}}{2\pi},
\end{align}
or the Kesten-McKay distribution,
\begin{align}\label{e:km}
    \rho_{d}=\left(1+\frac{1}{d-1}-\frac{x^2}{d}\right)^{-1}\frac{\sqrt{[4-x^2]_+}}{2\pi}.
\end{align}
As one can see from the expressions of semi-circle distribution \ref{e:semicircle}, and Kesten-McKay distribution \ref{e:km}, they both have square root behavior at the spectral edge. It is believed that square root behavior is necessary for edge universality. For the remainder of the paper, we assume that the initial measure $\hat\rho_0$ has square root behavior in the following sense. 

%we use a new solution of the McKean-Vlasov equation $\hat{m}_t$ that, although it does not equal $m_0$ at time 0, can be compared with $m_t$. The measure associated to $\hat{m}_t$ has properties such as square root behavior near the edge; this makes it more amenable to calculations of $|m_t - \hat{m}_t|$. This is our motivation for what we will call stable probability measures.
%{\cor
% define the set $W_t$ for $[0,T]$, such that if $z_t(z)\in W_t$, then $z_{t+s}(z)\in W_s$\cdots.
%}
\begin{definition}\label{d:stable}
%The  density $\hat \rho_0$ is stable if there exists some $T\geq 0$, such that the solution $(\hat \rho_t)$ for ${t\in [0,T]}$ of the  McKean-Vlasov equation \eqref{eq:dm0} with initial data $\hat\rho_0$ has square root behavior at the right edge. In our context, we will describe square root behavior in the following way
We say a probability measure $\hat\rho_0$ has square root behavior at $E_0$ if the measure is supported in $(-\infty, E_0]$ and, in addition, there is some neighborhood $\mathcal{N}$ around $E_0$ such that its Stieltjes transform satisfies
%\begin{enumerate}
%\item The potential $V$ is analytic in a small neighborhood of $E_0$.
\begin{align}\label{e:holom0}
\hat m_0(z)=A_0(z)+\sqrt{B_0(z)},
\end{align}
with  $A_0(z)$ and $B_0(z)$ analytic in $\mathcal{N}$ and with $z=E_0$ a simple root of $B_0(z)$.
%\end{enumerate}

\end{definition}
\begin{remark}\label{r:Imm0}
If $\hat \rho_0$ has square root behavior at right edge $E_0$, for any $z=E_0+\kappa+\ri\eta$, with $\eta>0$, it is easy to check that 
\begin{align}
    \Im[\hat m_0(z)]\asymp \Im\left[\sqrt{z-E_0}\right]\asymp \left\{
    \begin{array}{cc}
           \sqrt{|\kappa|+\eta}, & \kappa\leq 0\\
          \eta/\sqrt{|\kappa|+\eta}, &\kappa\geq 0.
    \end{array}
    \right.
\end{align}
\end{remark}

The Stieltjes transforms of semi-circle distribution and Kesten-McKay distribution are given by 
\begin{align}\begin{split}
&m_{\rm sc}(z)=\int_{\bR}\frac{\rho_{\rm sc}(x)\rd x}{x-z}=-\frac{z}{2}+\frac{\sqrt{z^2-4}}{2}   \\
&m_{d}(z)=\int_{\bR}\frac{\rho_{d}(x)\rd x}{x-z}=\left(1+\frac{1}{d-1}-\frac{z^2}{d}\right)^{-1}\left(-\frac{(d-2)z}{2d}+\frac{\sqrt{z^2-4}}{2}\right). 
\end{split}\end{align}
They both have square root behavior in the sense of Definition \ref{d:stable}. More generally, we have the following proposition.
\begin{proposition}
If $\hat \rho_0$ has an analytic density $\hat\rho_0$ in a small neighborhood of $E_0$, given by
\begin{align}
    \hat\rho_0(x)=S(x)\sqrt{[E_0-x]_+}, \quad E_0-\varepsilon \leq x\leq E_0+\varepsilon,
\end{align}
where $S(x)>0$ is analytic on $[E_0-\varepsilon,E_0+\varepsilon]$, then $\hat\rho_0$ has square root behavior in the sense of Definition \ref{d:stable}.
\end{proposition}
%\begin{proof}
%\cor To add.
%\end{proof}
%There are multiple  natural examples of stable measures, such as the Semicircle Law of Wigner matrices, the Kesten-McKay law of sparse random graphs, and the Marcenko-Pasture law of sample covariance matrices.

One important consequence of our definition of square root behavior measure is the following proposition which shows us how the square root behavior is a property that propagates in time when solving the McKean-Vlasov equation. We postpone its proof to the Appendix \ref{a:TimStab}.

\begin{proposition}\label{prop:TimStab}
Let $\hat\rho_0$ be a probability measure which has square root behavior at the right edge $E_0$ in the sense of Definition \ref{d:stable}. Fix a sufficiently small time $T>0$, and let $(\hat \rho_t)_{t\in [0,T]}$ the solution of the  McKean-Vlasov equation \eqref{eq:dm0} with initial data $\hat\rho_0$. Then the measures $\hat{\rho}_t$ have square root behavior at the right edge $E_t$, for any $0\leq t\leq T$. The edge $E_t$ satisfies,
\begin{align}\label{e:edgeEqn}
\del_t E_t = -\hat m_t(E_t)-\frac{V'(E_t)}{2},
\end{align}
and it is Lipschitz in time, $|E_t - E_s| = O(|t-s|)$ for $0\leq s\leq t\leq T$.
As a consequence, $\hat \rho_t$ has a density in the neighborhood of $E_t$, given by
\begin{align}\label{e:ConstDef}
    \hat\rho_t(x)=(1+\oo(1))C_t\sqrt{[E_t-x]_+},\quad E_t-\varepsilon\leq x\leq E_t+\varepsilon.
\end{align}
%there exist constants $C_t$,
%\begin{equation} 
%    \lim_{x \rightarrow E_t-}\frac{\hat\rho_t(E_t-x)}{\sqrt{E_t-x}} := C_t,
%\end{equation}
The constants $C_t$ are Lipschitz in time, $|C_t - C_s| = O(|t-s|)$, for $0\leq s\leq t\leq T$.
\end{proposition}

%{
%\begin{remark}\label{r:density}
%Let $\hat\rho_0$ be a stable measure as in Definition \ref{d:stable} and $\hat \rho_t$, with associated Green's function $\hat m_t$, be the solution of the McKean-Vlasov equation with initial data $\hat \rho_0$. By the inverse Stieltjes transform, \eqref{e:holom0} implies
%\begin{align}
%\frac{\rd \hat\rho_t(x)}{\rd x}=\lim_{\eta\rightarrow 0^+}\frac{1}{\pi}\Im[\hat m_t(x+\ri\eta)]\asymp \sqrt{B_t(x)} \asymp \sqrt{E_t -x},
%\end{align}
%for any $E_t-\fe\leq x\leq E_t$.
%\end{remark}
%}

The following proposition studies the growth of the distance of the real part of the characteristics $z_t(u)$ to the edge $E_t$. This is the main proposition we use to give strong bounds on $|m_t - \hat{m}_t|$ close to the edge and it serves as one of our fundamental inequalities in next section. The square root behavior of the measures $\rho_t$ was used essentially to describe an equation for the growth of $E_t$ and to provide estimates for the Stieltjes transform.

\begin{proposition}\label{p:gap}
Let $\hat \rho_0$ be a probability measure having square root behavior in the sense of Definition \ref{d:stable}. Fix small $\varepsilon>0$ and a sufficiently small time $T>0$, and let $(\hat \rho_t)_{t\in [0,T]}$ the solution of the  McKean-Vlasov equation \eqref{eq:dm0} with initial data $\hat\rho_0$. 
If at some $t\ll 1$, the characteristics $z_t(u)=E_t+\kappa_t+\ri\eta_t$, with $0<\eta_t,\kappa_t\leq \varepsilon$, then there exists an universal constant $C$ such that for any $0\leq s\leq t$,
\begin{align}\label{e:gap}
\sqrt{\kappa_s}\geq \sqrt{\kappa_t}+C(t-s).
\end{align}
\end{proposition}

\begin{proof}
We denote $z_s(u)=E_s+\kappa_s+\ri\eta_s$, for $0\leq s\leq t$. Thanks to \eqref{def:zt} and Proposition \eqref{prop:TimStab}, if $z_s(u)\in \bB_{2\varepsilon}(E_s)$, then there exists some universal constant $C$ such that $|\del_s z_s(u)|\leq C$. If we take $T$ sufficiently small, we will have that $z_s(u)\in \bB_{2\varepsilon}(E_s)$ for any $0\leq s \leq t$. In the following we prove that if $\kappa_s\geq 0$, then $\del_s\kappa_s\leq -C\sqrt{\kappa_s}$ for some universal constant $C$. Then the claim \eqref{e:gap} follows by integrating from $s$ to $t$, and we have $\kappa_s\geq 0$ for all $0\leq s\leq t$.

%Since $\hat \rho_0$ is a stable measure, we have
%\begin{align}
%\Im[m_s(z_s(u))]\asymp %\frac{\eta_s}{\sqrt{\kappa_s+\eta_s}},
%\end{align}
%for any $0\leq s\leq t$. From \eqref{e:imptbound}, it %follows that
%\begin{align}\label{e:mszsestimate}
%\Im[m_s(z_s(u))]\asymp %\frac{\eta_t}{\sqrt{\kappa_t+\eta_t}},\quad 0\leq %s\leq t.
%\end{align}
%We take imaginary parts on both sides of %\eqref{def:zt}, for $0\leq s\leq t$,
%\begin{align}
%\del_s \Im[z_s(u)]=-\Im[\hat %m_s(z_s(u))]-\frac{\Im[V'(z_s(u))]}{2}
%\geq -C\left(\frac{\eta_t}{\sqrt{\kappa_t+\eta_t}}+\I%m[z_s(u)]\right).
%\end{align}
%Thus it follows that there exists some constant $C$, %such that
%\begin{align}\label{e:etasbound}
%\eta_s=\Im[z_s(u)]\leq %C\left(\eta_t+(t-s)\frac{\eta_t}{\sqrt{\kappa_t+\eta_%t}}\right).
%\end{align}
%By our assumption that $t, \kappa_t,\eta_t\ll1$, %\eqref{e:etasbound} implies $\eta_s\ll1$. %and %\eqref{e:mszsestimate} implies %$\eta_s/\sqrt{\kappa_s+\eta_s}\ll1$, and thus %$\kappa_s\ll1$.
% Combining \eqref{e:etasbound} with %\eqref{e:mszsestimate}  gives us that $\kappa_s \ll %1$. Indeed,\eqref{e:mszsestimate} gives us $\kappa_s %+ \eta_s \asymp (\frac{\eta_s}{\eta_t})^2 (\eta_t + %\kappa_t)$. Applying \eqref{e:etasbound} allows us to %compare the right hand side to $(1 + %\frac{t-s}{\sqrt{\kappa_t}})^2(\eta_t + \kappa_t)$. %In the domain $\kappa_t \gg \eta_t$ this will be %bounded by $\max\{\kappa_t,t^2\}$ as both $\kappa_t %\ll 1$ and $t \ll 1$, we clearly see that $\kappa_s %\ll 1$.

We recall the differential equation \eqref{e:edgeEqn} for the edge $E_s$
\begin{align}\label{e:edgeEqncopy}
\del_s E_s=-\hat m_s(E_s)-\frac{V'(E_s)}{2}.
\end{align}
We take real part of \eqref{def:zt}, and take difference with \eqref{e:edgeEqncopy}
\begin{align}\label{e:diff}
\del_s \kappa_s
=-\left(\Re[\hat m_s(z_s(u))]-\hat m_s(E_s)\right)
-\frac{\Re[V'(z_s(u))]-V'(E_s)}{2}.
\end{align}
For the first term in \eqref{e:diff}
\begin{align}\begin{split}\label{e:firstterm}
&\phantom{{}={}}\Re[\hat m_s(z_s(u))]-\hat m_s(E_s)
=(\Re[\hat m_s(z_s(u))]-\hat m_s(\Re[z_s(u)]))+(\hat m_s(\Re[z_s(u)])-\hat m_s(E_s))\\
&=\eta_s^2\int\frac{\rd \hat \rho_s(x)}{(\Re[z_s(u)]-x)((\Re[z_s(u)]-x)^2+\eta_s^2)}
+\kappa_s\int\frac{\rd \hat \rho_s(x)}{(\Re[z_s(u)]-x)(E_s-x)}.
\end{split}\end{align}
The purpose of the above decomposition is to write out the expressions for the Stieltjes transform in a way that we can easily compare the corresponding integral expressions. From the integral expression, we can compute the leading order behavior in terms of $\kappa_s$ and $\eta_s$ in order to get an equation. Thanks to Proposition \ref{prop:TimStab}, $\hat \rho_s$ has square root behavior. From Remark \ref{r:Imm0}, we have $\rd \hat\rho_s(x)/\rd x\asymp \sqrt{E_s-x}$ on a neighborhood of $E_s$, and we can estimate \eqref{e:firstterm}
\begin{align}
\Re[\hat m_s(z_s(u))]-\hat m_s(E_s)
\geq C\left(\frac{\eta_s^2}{(\kappa_s+\eta_s)^{3/2}}+{\sqrt{\kappa_s}}\right).
\end{align}
where $C>0$ is some universal constant. For the second term in \eqref{e:diff}
\begin{align}\begin{split}\label{e:secondterm}
\Re[V'(z_s(u))]-V'(E_s)
&=(\Re[V'(z_s(u))]-V'(\Re[z_s(u)]))+(V'(\Re[z_s(u)])-V'(E_s))\\
&\geq -C(\eta_s^2+\kappa_s).
\end{split}\end{align}
Uniformly for $0\leq s\leq t$, we have $\kappa_s,\eta_s\leq 2\varepsilon$. By taking $\varepsilon$ sufficiently small, it follows by combining \eqref{e:firstterm} and \eqref{e:secondterm}, there exists some constant $C>0$ such that
\begin{align}
\del_s \kappa_s\leq -C\sqrt{\kappa_s},
\end{align}
and the claim \eqref{e:gap} follows.

\end{proof}

%We will now mention some consequences of the equation $\partial_t E_t = - m_t(E_t) - \frac{V'(E_t)}{2}$
%\begin{lemma}\label{lem:EdgeMov}
% We have the following two relations. Recall the definition of $C_t$ from \eqref{e:ConstDef}.
% \begin{align} \label{e:EdgeChange}
%     &|E_t - E_s| = O(|t-s|)\\
%     &|C_t - C_s| = O(|t-s|)
% \end{align}
%\end{lemma}

\section{Rigidity Estimates}
We prove our edge rigidity estimates in this section. Roughly speaking if the initial data is regular on the scale $\eta^*$, then the optimal rigidity holds for time $t\geq C\sqrt{\eta^*}$, provided $C$ is large enough. We fix a smaller number $r>0$ and the control parameter { $M=(\log N)^{12}$}.

\begin{assumption}\label{a:initial}
Let $\bmla(0)=(\la_1(0),\la_2(0),\cdots,\la_N(0))\in \overline{\Delta_N}$, and $\|\bm \la(0)\|_\infty\leq \fa$ for some constant $\fa$.
We assume that the initial empirical density
\begin{align}
\rho_0=\frac{1}{N}\sum_{i=1}^{N}\delta_{\la_i(0)}
\end{align}
satisfies
\begin{enumerate}
\item $\la_1(0)\leq E_0+\eta^*$.
\item There exists a  measure $\hat\rho_0$, with square root behavior as defined in Definition \ref{d:stable} such that we have the estimate 
\begin{align}
    |m_0(z) - \hat{m}_0(z)|\leq  \frac{M}{N \eta}, \quad z\in \cD_0^{\rm in},
\end{align} 
and 
\begin{align}
    |m_0(z) - \hat{m}_0(z)|\leq  \frac{1}{M}\frac{1}{N \eta}, \quad z\in \cD_0^{\rm out},
\end{align} 
and 
\begin{align}
    |m_0(z) - \hat{m}_0(z)|\leq  \frac{M}{N }, \quad z\in \cD_0^{\rm far},
\end{align} 
where $m_0(z)$ and $\hat m_0(z)$ are the Stieltjes transform of $\rho_0$ and $\hat\rho_0$ respectively, and the domains $\cD_0^{\rm in}$, $\cD_0^{\rm out}$ and $\cD_0^{\rm far}$ are given by
\begin{align}
\begin{split}
 &\cD_0^{\rm in}\deq\left\{z\in \bC^+\cap \bB_{E_0}(r): \Im[z]\Im[\hat m_0(z)]\geq (\eta^*)^{3/2} \right\},\\
&\cD_0^{\rm out}\deq\{z\in \bC^+\cap \bB_{E_0}(r): \Re[z]\geq E_0+\eta^* \},\\
& \cD_0^{\rm far}\deq\{z\in \bC^+: \fr-1\leq \dist(z,\supp\hat \rho_0)\leq \fr+1\},
\end{split}\end{align}
where $\fr$ is a large constant as defined in \eqref{def:fr}, and $\bB_{E_0}(r)$ is the radius $r$ disk centered at $E_0$.
\end{enumerate}
\end{assumption}

\begin{remark}
We remark here that it is essential to control the difference of $m_0$ and $\hat{m}_0$ far away from the support of $\hat\rho_0$, i.e. on $\cD_0^{\rm far}$. The effect of the potential $V$ is to cause a long range interaction that will cause two solutions to diverge if we have no control in this region. To see this effect, one should notice that if we were to compare the linear statistics of two measures, the difference will change by no more than a constant factor.% This is a simple consequence of a Gronwall lemma. 
\end{remark}

We define the following function
\begin{align}\label{e:deff}
f(t)=\left(\max\left\{\sqrt{\eta^*}-\fc t, M N^{-1/3}\right\} \right)^2,
\end{align}
where small constant $\fc>0$ will be chosen later.
It holds that $f(0)=\eta^*$, and it has similar behavior as the real part of characteristics as in \eqref{e:gap}, i.e it satisfies $\sqrt{f(s)}\leq \sqrt{f(t)}+\fc(t-s)$ for any $0\leq s\leq t$. We use this function in interpolating from weak eigenvalue rigidity at the edge at time $0$ to better eigenvalue rigidity at time $t$.

\begin{theorem} \label{Thm:EdgeRidgity}
Suppose $V$ satisfies Assumption \ref{a:asumpV} and the initial data $\bmla(0)$ satisfies Assumption \ref{a:initial}. For time $T=(\log N)^{-3}$, with high probability under the Dyson Brownian motion \eqref{DBM}, we have $\lambda_1(t) \leq E_t + f(t)$ for $t \in [0,T]$.

\end{theorem}

We define the spectral domains $\cD_t$. Roughly speaking the information of Stieltjes transform $m_t(z)$ on $\cD_t$ reflects the regularity of the empirical particle density $\rho_t$ on the scale $f(t)$. 
\begin{definition}
For any $t\geq 0$, we define the region $\mathcal{D}_t= \cD_t^{\rm in}\cup \cD_t^{\rm out}\cup \cD_t^{\rm far}$, where
\begin{align}\begin{split}\label{e:defcDt}
    &\mathcal{D}_t^{\rm in} \deq
    \{z\in \bC^+\cap \bB_{E_t}(r-t/\fc):  \Im[z]\Im[\hat m_t(z)]\geq f(t)^{3/2} \},\\
    &\mathcal{D}_t^{\rm out} \deq\left\{z\in \bC^+\cap \bB_{E_t}(r-t/\fc): \Re[z]\geq E_t+f(t)\right\}, 
      %\{z: \Im[z] > \max\{(\sqrt{\eta^*} - \fc s/2)^2, N^{-2/3}\} \}\bigcup\{z: \dist(\Re[z],\supp{\hat{\rho}}_s) > \eta^*\}
    \\
    &\mathcal{D}_t^{\rm far}\deq \{z\in \bC^+: \fr-1+t/\fc\leq \dist(z,\supp\hat \rho_t)\leq \fr+1-t/\fc\},\\
\end{split}\end{align}

%We also define the region $\mathcal{D}^*$ as
%\begin{equation}
%    \mathcal{D}^* :=\{z: \Im[z] > \eta^*/2\}
%    \bigcup\{\dist(\Re[z],\supp{\hat{\rho}}_0) > \eta^*\}
%\end{equation}
\end{definition}

%\begin{remark}
%We remark that under Assumption \ref{a:initial}, we in fact have that for any $\eta^*/2\leq\Im[z]\leq \fb$ and $E_0+\eta^*\leq \Re[z]\leq \fb$, $|m_0(z) - \hat{m}_0(z)| \lesssim M/(N(\Re[z]-E_0))$.
%\end{remark}

%\section{Rigidity Estimates}

%\begin{proposition}\label{p:ksbound}
%If at some $t\ll 1$, the characteristic flow $z_t(u)=E_t+f(t)+\ri\eta_t$, with $0<\eta_t\ll1$, then for $0\leq s\leq t$,
%\begin{align}\label{e:ksbound}
%\kappa_s\geq f(s)+\fc\sqrt{\kappa_s}(t-s),
%\end{align}
%where $\kappa_s=\Re [z_s(u)]-E_s$.
%\end{proposition}
%\begin{proof}
%By the definition of $f(t)$ in \eqref{e:deff}, we have
%\begin{align}
%\sqrt{f(s)}\leq \fc(t-s)+\sqrt{f(t)}= \fc(t-s)+\sqrt{\kappa_t}\leq \sqrt{\kappa_s}-\fc(s-t).
%\end{align}
%The claim \eqref{e:ksbound} follows.
%\end{proof}
%Moreover, by our assumption that $\hat \rho_0$ is a stable measure (as in Definition \ref{d:stable}). There exists a constant $\fc$ such that 
%\begin{align}
%    \Im[\hat m_t(E_t+\kappa+\ri \eta)]\geq \fc \sqrt{|\kappa|+\eta},
%\end{align}
%for any $0\leq t\ll1$ and $\kappa\leq 0$.

For any $0\leq s\leq t$, the spectral domain $\cD_t$ is a subsect of the domain $\cD_s$ under the characteristic flow.

\begin{proposition}\label{p:domainic}
Suppose $V$ satisfies Assumption \ref{a:asumpV} and the initial data $\bmla(0)$ satisfies Assumption \ref{a:initial}. For any $0\leq s\leq t\ll r$, we have
\begin{align}
    z_s\circ z_t^{-1}(\cD_t)\subset \cD_s,
\end{align}
provided $N$ is large enough.
\end{proposition}
\begin{proof}

By integrating \eqref{def:fr}, we get that $|z_t-z_0|=\OO(t)$ for any $z_t\in \cD_t$. It follows that $z_t^{-1}(\cD_t^{\rm far})\subset \cD_0^{\rm far}$, and $z_t^{-1}(\bC^+\cap\bB_{E_t}(r-t/\fc))\subset \bC^+\cap\bB_{E_0}(r)$, provided that $\fc$ is small enough.

For any $z_t\in  \{z\in \bC^+\cap \bB_{E_t}(r-t/\fc):  \Re[z]\geq E_t+f(t) \}$, let $z_s=E_s+\kappa_s+\ri\eta_s$ for $0\leq s\leq t$. By the definition of $f(t)$ in \eqref{e:deff}, we have
\begin{align}\label{e:ksbound0}
\sqrt{f(s)}\leq \fc(t-s)+\sqrt{f(t)}= \fc(t-s)+\sqrt{\kappa_t}\leq \sqrt{\kappa_s}-\fc(t-s),
\end{align}
provided that $\fc\leq C/2$, where $C$ is the constant in \eqref{e:gap}. We can rearrange \eqref{e:ksbound0} to get
\begin{align}\label{e:ksbound}
\kappa_s\geq f(s)+\fc\sqrt{\kappa_s}(t-s).
\end{align}
As a consequence, we have $\Re[ z_0]\geq E_0+f(0)=E_0+\eta^*$ and  $z_t^{-1}(\cD_t^{\rm out})\subset \cD_0^{\rm out}$.

 Thanks to Proposition \ref{prop:ImEst}, if $z_t\in \cD_t^{\rm in}$ and $\Re[z_t]\leq E_t+f(t)$, we have
\begin{align}\begin{split}
    \Im[z_0]\Im[\hat m_0(z_0)]
    &\geq e^{-tC}(\Im[z_t]\Im[\hat m_t(z_t)]+t\Im[\hat m_t(z_t)]^2)\\
    &\geq e^{-tC} (f(t)^{3/2}+t\Im[\hat m_t(z_t)]^2)\geq (\eta^*)^{3/2},
\end{split}\end{align}
provided $t\leq \sqrt{\eta^*}/(3\fc)$ or $\Re[z_t]\leq E_t-\eta^*$. For $t\geq \sqrt{\eta^*}/(3\fc)$, in fact we have $z_t^{-1}(\bC^+\cap\bB_{E_t}(\eta^*))\subset \cD_*^{\rm out}$. We prove it by contradiction. Say if there exists some $z_t\in \bC^+\cap\bB_{E_t}(\eta^*)$, such that $\Re[z_0]\leq E_t+\eta^*$.
By our assumption that $\hat\rho_t$ has square root behavior, we have $\Im[\hat m_t(z_t)]\leq C\sqrt{\eta^*}$.
Thanks to Proposition \ref{prop:ImEst}, we have $\Im[z_0]\geq e^{-tC}(\Im[z_t]+t\Im[\hat m_t(z_t)])\geq e^{-tC}t\Im[\hat m_t(z_t)]$, and thus
\begin{align}\begin{split}
     \Im[\hat m_t(z_t)]\geq e^{-Ct}\Im[\hat m_0(z_0)]\geq 
    \frac{e^{-Ct}}{C}\frac{\Im[z_0]}{\sqrt{\eta^*+\Im[z_0]}}\geq\frac{e^{-2Ct}}{C} \frac{t\Im[\hat m_t(z_t)]}{\sqrt{\eta^*+e^{-tC}t\Im[\hat m_t(z_t)]}},
\end{split}
\end{align}
which is impossible if $t\geq \sqrt{\eta^*}/(3\fc)$, and $\fc$ is sufficiently small. This finishes the proof of Proposition \ref{p:domainic}

\end{proof}

The following proposition gives optimal bulk estimate of $m_t$, i.e. on the spectral domain $\cD_t^{\rm in}\cup \cD_t^{\rm  far}$. 
\begin{proposition}
\label{p:rigidity}
Suppose $V$ satisfies the Assumption \ref{a:asumpV}. Fix time $T=(\log N)^{-3}$. For any initial data $\bmla(0)$ satisfies Assumption \ref{a:initial}, uniformly for any $0\leq t\leq T$, and $w\in \dom_t^{\rm in}\cup \dom_t^{\rm far}$ 
%in the spectral domain 
%\beq\label{def:dom}
%\dom_t=\left\{w\in \bC_+: \Im [m_t(w)]\geq \frac{e^{Kt}M(\log N)^{\delta}}{N\Im[w]},\quad \Im[w]\leq 3{\fb}-t,\quad  |\Re[w]|\leq 3\fb-t\right\},
%\eeq
%where $M=(\log N)^{2+\delta}$, and $K=K(V)$ and $\fb=\fb(\fa, V)$ are constants, 
there exists a set $\Omega$ that occurs with overwhelming probability on which the following estimate holds: if $w\in \cD_t^{\rm in}$
\beq \label{e:diffmm}
|m_t(w)- \hat m_t(w)|\leq \frac{M}{N\Im[w]}, 
\eeq
if $w\in \cD_t^{\rm far}$
\beq \label{e:diffmm2}
|m_t(w)- \hat m_t(w)|\leq \frac{M}{N}.
\eeq
\end{proposition}

The proof of proposition \ref{p:rigidity} follows the same argument as \cite[Theorem 3.1]{HL}, with two modifications. Firstly, when we use Gronwall inequality, we need to take care of the error from the initial data, i.e. $m_0(z)-\hat m_0(z)\neq 0$. This is where our Assumption \ref{a:initial} comes into play. Secondly, we estimate the error term involving the potential $V$ using a contour integral.
\begin{proof}[Proof of Proposition \ref{p:rigidity}]

By Ito's formula, $ m_s(z)$ satisfies the stochastic differential equation
\begin{align}\begin{split}\label{eq:dm}
\rd  m_s(z)= -&\sqrt{\frac{2}{\beta N^3}}\sum_{i=1}^N \frac{{\rm d} B_i(s)}{(\la_i(s)-z)^2}+ m_s(z)\del_z  m_s(z)\rd s \\
+&\frac{1}{2N}\sum_{i=1}^{N}\frac{V'(\la_i(s))}{(\la_i(s)-z)^2}\rd s+\frac{2-\beta}{\beta N^2}\sum_{i=1}^{N}\frac{\rd s}{(\lambda_i(s)-z)^3}.
\end{split}\end{align}
%We rearrange \eqref{eq:dm} as
%\begin{align}\label{eq:dm2}\begin{split}
%d \td m_s=&\left(\td m_s+\frac{V'(z)}{2}\right)\del_z \left(\td m_s(z)+\frac{V'(z)}{2}\right)+\frac{1}{N}\sum_{i=1}^N g(z,\la_i(t))-\frac{V'(z)\del_z V'(z)}{4} \\
%+&\frac{2-\beta}{\beta N^2}\sum_{i=1}^{N}\frac{\rd t}{(\la_i(t)-z)^3}-\sqrt{\frac{2}{\beta N^3}}\sum_{i=1}^N \frac{{\rm d} B_i(t)}{(\la_i(t)-z)^2},
%\end{split}
%\end{align}
%where 
%\beq 
%g(z,\la)=\frac{V'(\la)-V'(z)-(\la-z)V''(z)}{2(\la-z)^2},\quad g(\la,\la)=\frac{V'''(\la)}{4}.
%\eeq
%By our definition \eqref{def:V'}, $V'$ is quasi-analytic along the real axis. One can directly check the following properties of $g(z,\la)$:{\rn{1}}) $\lim_{z\rightarrow E\in \bR}g(z,\lambda)=g(E,\lambda)$, which is real; {\cor Need to expand: \rn{2}) for any fixed $z$, $g(z,\la)$ as a function of $\la\in \bR$, $\|g(z,\la)\|_\infty=O(1)$, $\|\del_\la g(z,\la)\|_\infty=O(\Im[z]/|z-\lambda|^2), \|\del^2_\la g(z,\la)\|_\infty=O(\Im[z]^2/|z-\lambda|^3)$.}
%
%For any fixed $z$, $g(z,\la)$ as a function of $\la\in \bR$, we define the quasi-analytic extension of $g(z,\cdot)$ of order two,
%\beq
%\td g(z,x+\ri y)=(g(z,x)+\ri y \del_x g(z,x))\chi(y),
%\eeq 
%By the Helffer-Sj{\"o}strand formula,
%\beq
%\frac{1}{N} \sum_{i=1}^{N} g(z,\la_i(t))=\frac{1}{\pi}\int_{\bC} \del_{\bar w} \td g(z,w) \td m_s(w)\rd^2 w.
%\eeq
We can rewrite \eqref{eq:dm} as 
\begin{align}\label{eq:dm3}\begin{split}
\rd  m_s(z)=-&\sqrt{\frac{2}{\beta N^3}}\sum_{i=1}^N \frac{{\rm d} B_i(s)}{(\la_i(s)-z)^2}+\del_z  m_s(z)\left( m_s(z)+\frac{V'(z)}{2}\right)\rd s+\frac{ m_s(z)\del_z V'(z)}{2}\rd s\\+& \int_{\bR} g(z,x) \rd \rho_s(x)\rd s
%\oint_{\mathcal{C}_s} g(z,w) m_s(w) \rd w \rd s
+\frac{2-\beta}{\beta N^2}\sum_{i=1}^{N}\frac{\rd s}{(\la_i(s)-z)^3},
\end{split}
\end{align}
%\begin{align}\label{eq:dm3}\begin{split}
%\rd \td m_s=-&\sqrt{\frac{2}{\beta N^3}}\sum_{i=1}^N \frac{{\rm d} B_i}{(\la_i(t)-z)^2}+\left(\td m_s+\frac{V'(z)}{2}\right)\del_z \left(\td m_s(z)+\frac{V'(z)}{2}\right)-\frac{V'(z)\del_z V'(z)}{4}\\+& \frac{1}{\pi}\int_{\bC} \del_{\bar w} \td g(z,w) \td m_s(w)\rd^2 w
%+\frac{2-\beta}{\beta N^2}\sum_{i=1}^{N}\frac{\rd t}{(\la_i(t)-z)^3},
%\end{split}
%\end{align}
 $ g(z,w)$ is defined in \eqref{def:gzx}. 
Plugging \eqref{def:zt} into \eqref{eq:dm3}, and by the chain rule,  we have 
%\begin{align}\label{e:tdmzt}\begin{split}
%\rd \left(\td m_s(z_s)+\frac{V'(z_s)}{2}\right)=-&\sqrt{\frac{2}{\beta N^3}}\sum_{i=1}^N \frac{{\rm d} B_i(t)}{(\la_i(t)-z_s)^2}+\left(\td m_s(z_s)-m_s(z_s)\right)\del_z \left(\td m_s(z_s)+\frac{V'(z_s)}{2}\right)\rd t\\-&\frac{V'(z_s)\del_z V'(z_s)\rd t}{4}
%+ \frac{1}{\pi}\int_{\bC} \del_{\bar w} \td g(z_s,w) \td m_s(w)\rd^2 w \rd t
%+\frac{2-\beta}{\beta N^2}\sum_{i=1}^{N}\frac{\rd t}{(\la_i(t)-z_s)^3},
%\end{split}
%\end{align}
\begin{align}\label{e:tdmzt}\begin{split}
\rd  m_s(z_s)=-&\sqrt{\frac{2}{\beta N^3}}\sum_{i=1}^N \frac{{\rm d} B_i(s)}{(\la_i(s)-z_s)^2}+\del_z m_s(z_s)\left(m_s(z_s)-\hat{m}_s(z_s)\right)\rd s+\frac{ m_s(z_s)V''(z_s)}{2}\rd s\\
+ & %\oint_{\mathcal{C}_s}  g(z_s,w)  m_s(w)\rd w \rd s
\int_{\bR}  g(z_s,x)  \rd \rho_s(x)\rd s
+\frac{2-\beta}{\beta N^2}\sum_{i=1}^{N}\frac{\rd s}{(\la_i(s)-z_s)^3}.
\end{split}
\end{align}
It follows by taking the difference of \eqref{e:mtzt} and \eqref{e:tdmzt} that, 
\begin{align}\label{e:diffm}\begin{split}
\rd (m_s(z_s)-\hat{m}_s(z_s))=-&\sqrt{\frac{2}{\beta N^3}}\sum_{i=1}^N \frac{{\rm d} B_i(s)}{(\la_i(s)-z_s)^2}+\left( m_s(z_s)-\hat{m}_s(z_s)\right)\del_z \left( m_s(z_s)+\frac{V'(z_s)}{2}\right)\rd s\\
+& 
\int_\bR  g(z_s,x) (\rd \rho_s(x)-\rd \hat\rho_s(x))\rd s
%\oint_{\mathcal{C}_s}  g(z_s,w) (m_s(w)-\hat{m}_s(w))\rd  w\rd s
+\frac{2-\beta}{\beta N^2}\sum_{i=1}^{N}\frac{\rd s}{(\la_i(s)-z_s)^3}.
\end{split}
\end{align}
 We can integrate both sides of \eqref{e:diffm} from $0$ to $t$ and obtain
\beq \label{eq:mzt}
m_t(z_t)- \hat{m}_t(z_t)=\int_0^t \left( \cE_1(s)\rd s+\rd \cE_2(s) \right) +(m_0(z_0) - \hat{m}_0(z_0)),
\eeq
where the error terms are
\begin{align}
\label{defcE1}\cE_1(s)=&\left(m_s(z_s)-\hat{m}_s(z_s)\right)\del_z \left( m_s(z_s)+\frac{V'(z_s)}{2}\right)+
\int_\bR  g(z_s,x) (\rd \rho_s(x)-\rd \hat\rho_s(x)),
%\oint_{C_s} g(z_s,w) (m_s(w)- \hat{m}_s(w))\rd w , 
\\
\label{defcE2}\rd \cE_2(t)=&\frac{2-\beta}{\beta N^2}\frac{\rd s}{(\lambda_i(s)-z_s)^3}-\sqrt{\frac{2}{\beta N^3}}\sum_{i=1}^N \frac{{\rd} B_i(s)}{(\la_i(s)-z_s)^2}.
\end{align}
We remark that $\cal E_1$ and $\cal E_2$ implicitly depend on $u$, the initial value of the flow $z_s(u)$.  The local law will eventually follow from an application of Gronwall's inequality to \eqref{eq:mzt}.

We define the following lattice on the upper half plane $\bC_+$, 
\beq\label{def:L}
\cal L=\left\{E+\ri \eta\in \dom_0^{\rm in}\cup \dom_0^{\rm out}\cup \dom_0^{\rm far}: E\in \bZ/ N^{3}, \eta\in \bZ/N^{3}\right\}.
\eeq
It follows from Propositions \ref{p:domainic}, $z_t^{-1}(\cD_t)\subset\dom_*^{\rm in}\cup \dom_*^{\rm out}\cup \dom_*^{\rm far}$, and for any and $w\in \dom_t$, there exists some lattice point $u\in \cal L\cap z_t^{-1}(\dom_t)$, such that
$
|z_t(u)-w|=\OO(N^{-3}).
$
%Therefore, thanks to the Lipschitz property of Stieltjes transform, the statement \eqref{}

%\begin{assumption}\label{a:initial}
%Let $\bmla(0)=(\la_1(0),\la_2(0),\cdots,\la_N(0))$, with $\la_1(0)\geq \la_2(0)\geq \cdots\geq \la_N(0)$. We assume that %the initial empirical density
%\begin{align}
%\rho_0=\frac{1}{N}\sum_{i=1}^{N}\delta_{\la_i(0)}
%\end{align}
%satisfies
%\begin{enumerate}
%\item $\la_1(0)\leq E_0+f(0)$.
%\item There exists a stable measure $\hat\rho_0$ such that we have the estimate $|m(z) - \hat{m}(z)| \le \frac{1}{N %\dist(z, \supp{\hat{\rho}_0})}$ for $z \in \mathcal{D}^*$
%\end{enumerate}
%\end{assumption}

We define the stopping time 
\begin{align}\begin{split}\label{stoptime}
\sigma\deq 
T
&\bigwedge 
\inf_{s\geq0}\left\{\|\bmla(s)\|_{\infty}\geq \fb\right\}\\
&\bigwedge
\inf_{s\geq0}\left\{\exists w\in \cD_s^{\rm in }: \left|m_s(w)-\hat m_s(w)\right|\geq \frac{M}{N\Im[w]}\right\}\\
&\bigwedge\inf_{s\geq0}\left\{\exists w\in \cD_s^{\rm far}: \left|m_s(w)-\hat m_s(w)\right|\geq \frac{M}{N}\right\}.
\end{split}\end{align}
By the same argument as in \cite[Proposition 3.8]{HL}, using Burkholder-Davis-Gundy inequality, there exists a set $\Omega$ of Brownian paths $\{B_1(s), B_2(s), \cdots, B_N(s)\}_{0\leq s\leq t}$, such that for any $0\leq s\leq t$, and $u\in \cL\cap z_t^{-1}(\cD_t^{\rm in})$,
\begin{align}\begin{split}\label{e:continuityarg}
\left|\int_0^{s\wedge \sigma} \rd \cE_2(s)\right|
\leq  \frac{(\log N)^2}{N\Im[z_{s\wedge\sigma}(u)]},
\end{split}
\end{align}
and $u\in \cL\cap z_t^{-1}(\cD_t^{\rm far})$,
\begin{align}\begin{split}\label{e:continuityarg}
\left|\int_0^{s\wedge \sigma} \rd \cE_2(s)\right|
\leq  \frac{(\log N)^2}{N}.
\end{split}
\end{align}

%and $\mathcal{C}_s$ is a contour of distance $\fr$ away from the support of $\hat{\rho}_s$. Thanks to our definition of $\cD_s$ in \ref{e:defcDt}, we have $\cC_s\subset \cD_s$.

For the last term in \eqref{defcE1}, we rewrite it as a contour integral and bound it simply by its absolute value.
\begin{proposition} \label{prop:HFbound}
Under the assumptions of Theorem \ref{Thm:EdgeRidgity} for any $u \in z_t^{-1}(\cD_t)$ and $s\in [0,t]$ we have
\begin{equation}\label{e:HFbound}
%\left|\oint_{\mathcal{C}_s} (g(z_{s \wedge \sigma}(u) ,w) ({m}_{s \wedge \sigma}(w) - \hat m_{s\wedge \sigma}(w)) \rd w\right|
    \left|\int_{\bR} g(z_{s \wedge \sigma}(u) ,x) (\rd \rho_{s\wedge \sigma}(x)-\rd \hat\rho_{s\wedge \sigma}(x)) \right| =\OO\left( \frac{M}{N}\right).
\end{equation}
\end{proposition}

\begin{proof}
From our choice of the stopping time \eqref{stoptime}, we have both $\rho_{s\wedge \sigma}$ and $\hat \rho_{s\wedge \sigma}$ are supported on $[-\fb, \fb]$. Moreover, $g(z_{s\wedge \sigma}(u), x)$ is analytic in $x$, we can rewrite the integral in \eqref{e:HFbound} as a contour integral
\begin{align}
   \int_{\bR} g(z_{s \wedge \sigma}(u) ,x) (\rd \rho_{s\wedge \sigma}(x)-\rd \hat\rho_{s\wedge \sigma}(x)) =
    -\frac{1}{2\pi \ri}\oint_{\mathcal{C}_s} (g(z_{s \wedge \sigma}(u) ,w) ({m}_{s \wedge \sigma}(w) - \hat m_{s\wedge \sigma}(w)) \rd w
\end{align}
where $\mathcal{C}_s$ is a contour of distance $\fr$ away from the support of $\hat{\rho}_s$. Thanks to our definition of $\cD_s$ in \ref{e:defcDt}, we have $\cC_s\subset \cD_s$.
The above contour integral can be bounded as
\begin{align}\begin{split}
    &\phantom{{}={}}\left|\oint_{\mathcal{C}_s} (g(z_{s \wedge \sigma} ,w) ({m}_{s \wedge \sigma})(w) -\hat m_{s\wedge \sigma}(w)) \rd w\right| \\
    & \leq {\rm length}(\cC_s)\sup_{w\in \mathcal{C}_s}|g(z_{s \wedge \sigma} ,w)| |(\hat{m}_{s \wedge \sigma})(w) - m_{s\wedge \sigma}(w))| =\OO\left( \frac{M}{N}\right),
\end{split}\end{align}
where we use the fact that $g$ is bounded on the contour $\mathcal{C}_s$, the length of $\mathcal{C}_s$ is bounded, and we have rigidity along the contour $\mathcal{C}_s$.
\end{proof}
%Note that $D$ is some constant derived from the size support of $\hat{\rho}$ and a supremum bound on the analytic function $g(z,x)$ on a compact interval.

We plug \eqref{e:HFbound} and \eqref{e:continuityarg} into \eqref{eq:mzt}, on the event $\Omega$, for $u\in \cL \cap z_t^{-1}(\cD_t^{\rm in})$ we have
\begin{align}\begin{split}
    m_{t\wedge\sigma}(z_{t\wedge\sigma})-\hat m_{t\wedge \sigma}(z_{t\wedge \sigma})
    &=(m_0(z_0)-\hat m_s(z_0))+\OO\left(\frac{({t\wedge\sigma})M}{N}
+\frac{(\log N)^2}{N\eta_{t\wedge \sigma}}\right)\\
&+\int_0^{t\wedge\sigma}\left|\hat{m}_{s}(z_s)-m_s(z_s)\right|\left|\partial_z \left( m_s(z_s)+\frac{V'(z_s)}{2}\right)\right|\rd s.
\end{split}\end{align}
It follows by the Gronwall inequality, and same argument as in \cite{HL}, we get
\begin{align}
     |m_{t\wedge\sigma}(z_{t\wedge\sigma})-\hat m_{t\wedge \sigma}(z_{t\wedge \sigma})|\leq \frac{\oo(M)}{N\eta_{t\wedge\sigma}},
\end{align}
provided that $t\leq T=(\log N)^{-3}$. And similarly for $u\in \cL\cap z_t^{-1}(\cD_t^{\rm far})$, we have 
\begin{align}
     |m_{t\wedge\sigma}(z_{t\wedge\sigma})-\hat m_{t\wedge \sigma}(z_{t\wedge \sigma})|\leq \frac{\oo(M)}{N},
\end{align}
provided that $t\leq T=(\log N)^{-3}$.
Thus with high probability we have $\sigma=T$, and Proposition \ref{p:rigidity} follows.
\end{proof}

\begin{comment}
\begin{proposition}\label{Thm:EdgeRigid}
Fix time $T=(\log N)^{-2}$ and $0\leq t\leq T$. For any initial data $\bmla(0)$ satisfies Assumption \ref{a:initial},  with high probability under the Dyson Brownian motion \eqref{DBM}, there is no particle in the interval $[E_t+f(t)-N^{-1}, E_t+f(t)+N^{-1}]$ for any time $t-N^{-1}\leq s\leq t$, i.e.
\begin{align}
    |\{i\in \qq{N}: \la_i(s)\in [E_t+f(t)-N^{-1}, E_t+f(t)+N^{-1}]\}|=0,\quad t-N^{-1}\leq s\leq t.
\end{align}
\end{proposition}
\end{comment}
\begin{comment}
{\cor
\begin{lemma}
\begin{align}\begin{split}\label{e:someclaims}
\Im[m_s(z_s)]=-\partial_s \Im[z_s] +O(\Im[z_s])\\
\Im[m_s(z_s)]\asymp \Im[\hat m_t(z_t)]\\
\del_t E_t=\OO(1).
\end{split}\end{align}
\end{lemma}
it seems that we repeatedly used them.
}
\end{comment}

\begin{proof}[Proof of Theorem \ref{Thm:EdgeRidgity}]
Theorem \ref{Thm:EdgeRidgity} follows from a very precise estimate of the Stieltjes transform. More precisely, it follows from the following estimate 
\begin{align}\label{e:mtbound}
    |m_t(E_t+\kappa+\ri \eta)-\hat m_t(E_t+\kappa+\ri \eta)|\ll \frac{1}{N\eta},
\end{align}
where $\kappa\geq M^2 N^{-2/3}$ and $\eta=M^{-1/3}\kappa^{1/4}N^{-1/2}\geq M^{1/6}N^{-2/3}$, that there is no particle on the interval $[E_t+\kappa-\eta, E_t+\kappa+\eta]$. Thanks to our assumption \ref{a:initial} that $\hat\rho_0$ and $\hat \rho_t$ have square root behavior, and $\Im[\hat m_t(E_t+\kappa+\ri \eta)]\asymp \eta/\sqrt{\kappa+\eta}\ll 1/N\eta$. 
Then it follows that 
\begin{align}\label{e:sumerror}
    \Im[m_t(E_t+\kappa+\ri \eta)]=\frac{1}{N}\sum_{i=1}^N \frac{\eta}{(\la_i(t)-\kappa-f(t))^2+\eta^2}\le \frac{1}{N\eta}.
\end{align}
If there exists some $\la_i(t)$ such that $|\la_i(t)-E_t-\kappa|\leq \eta$, then the righthand side of \eqref{e:sumerror} is at least $1/(2N\eta)$. This leads to a contradiction.

In the following, we will use a stopping time argument to show estimates like \eqref{e:mtbound}. We let $t_i=i/N$, for $i\leq \lceil TN\rceil$ and $\{z_s(u_i)\}_{1\leq s\leq t_i}$ denote the characteristic flow starting at $u_i$ such that at time $t_i$, $z_{t_i}(u_i)=E_{t_i} + f(t_i)+ \ri M^{-1/3}f(t_i)^{1/4}N^{-1/2}$. 
Thanks to \eqref{e:gap}, for $0\leq t\leq t_i$, we have
\begin{align}\label{e:real}
    z_t(u_i)-E_t\geq \left(\sqrt{f(t_i)}+\fc(t_i-t)\right)^2.
\end{align}
Moreover, using Proposition \ref{prop:ImEst}
\begin{align}\label{e:imag}
    \Im[z_t(u_i)]\asymp \Im[z_{t_i}(u_i)]+(t_i-t)\Im[\hat m_{t_i}(z_{t_i}(u_i))]\asymp M^{-1/3}f(t_i)^{1/4}N^{-1/2}(1+(t_i-t)/\sqrt{f(t_i)}).
    \end{align}
where we used that $\hat \rho_{t_i}$ has square root behavior. It follows from comparing \eqref{e:real} and \eqref{e:imag} we get that
\begin{align}\label{e:realim}
    z_t(u_i)-E_t\geq M^{3/2}\Im[z_t(u_i)],
\end{align}
for any $0\leq t\leq t_i$. 

We now define the stopping time $\sigma$
\begin{align}\begin{split}\label{e:defsigma}
\sigma := T&\bigwedge\inf\{s:\lambda_1(s) - E_s\geq  f(s)\} \\
& \bigwedge\inf\left\{s:\exists i\leq \lceil TN\rceil, {\bf 1}_{s\leq t_i}|m_s(z_s(u_i)) - \hat{m}_s(z_s(u_i))|\geq \frac{1}{M^{1/4}} \frac{1}{N\Im[z_s(u_i)]}\right\}\\
 & \bigwedge \inf \left\{s:\exists w\in \cD_s^{\rm far}, |m_s(w)- \hat{m}_s(w)| \geq \frac{M}{N }\right \}.
\end{split}\end{align}
%We will construct a set $\Omega$ such that on $\Omega$ we have $\sigma=T$ and $\Omega$ holds with high probability.
To get a more precise estimate of the Stieltjes transform $m_s(z_s(u_i))$, we need to upgrade the estimate \eqref{e:continuityarg}.  
The following proposition analyzes the short range deterministic term in \eqref{defcE2} for the edge terms
\begin{proposition}\label{l:thirdbound}
For any $0\leq t\leq t_i$,
 \begin{equation}
\frac{2- \beta}{\beta} \int_{0}^{t \wedge \sigma}  \frac{1}{N^2} \sum_{k=1}^{N}\frac{1}{|\lambda_{k}(s) - z_{s}(u_i)|^3}\rd s\leq  \frac{C \log N}{N \kappa_{t \wedge \sigma}},
 \end{equation}
 where $\kappa_t=\Re[z_t(u_i)]-E_t$ for $0\leq t\leq t_i$.
\end{proposition}
\begin{proof}
For simplicity of notations, we write $z_t(u_i)$ as $z_t$, and denote $\eta_t=\Im[z_t(u_i)]$ for $0\leq t\leq t_i$.
By the definition of $\sigma$, for $s\leq \sigma$, it holds $\la_1(s)\leq E_s+f(s)$.
 \begin{align}\begin{split}
 &\phantom{{}={}}\int_{0}^{t \wedge \sigma} \frac{1}{N^2} \sum_{k=1}^{N}\frac{1}{|\lambda_{k}(s) - z_{s}|^3} \rd s
 \leq  \frac{1}{N}\int_{0}^{t \wedge \sigma} \frac{ \text{Im}[m_s(z_s)]}{|\Re[z_s]-\lambda_1(s)+\ri\eta_s| \text{Im}[z_s]} \rd s \\
 &\leq  \frac{2}{N}\int_{0}^{t \wedge \sigma} \frac{\text{Im} [\hat{m}_s(z_s)]}{|\Re[z_s]-\lambda_1(s)+\ri\eta_s|\text{Im}[z_{s}]} \rd s
 \leq  \frac{C}{N} \int_{0}^{t \wedge \sigma} \frac{\eta_{s}/\sqrt{\kappa_{s}}}{(\kappa_s -f(s)+\eta_s) \eta_{s}}\rd s\\
&\leq \frac{C}{N}\int_{0}^{t \wedge \sigma} \frac{\rd s}{(\kappa_{s})^{1/2}((\kappa_{s})^{1/2}(t_i-s)+\eta_s)}\leq \frac{C}{N\kappa_{t\wedge \sigma}}\int_{0}^{t \wedge \sigma} \frac{\rd s}{(t_i-s)+\eta_s/\sqrt{\kappa_s}}\leq \frac{C\log N}{N\kappa_{t\wedge \sigma}},
\end{split}\end{align}
where in the last line we use \eqref{e:ksbound} and the increasing gap \eqref{e:gap} inequality, for any $0\leq s\leq t\wedge \sigma$,
\begin{align}
\kappa_s^{1/2}\geq \kappa_{t\wedge \sigma}^{1/2}+C(t\wedge \sigma-s)\geq \kappa_{t\wedge \sigma}^{1/2}.
\end{align}
\end{proof}

A similar analysis can be done to analyze the short range stochastic term in \eqref{defcE2} for the edge terms.
\begin{proposition}\label{l:StochasticBnd}
There exists a set $\Omega$, which holds with high probability, such that on $\Omega$ the following inequality holds for any $i\leq \lceil TN\rceil$
\begin{equation}\label{StochasticBnd}
\int_{0}^{t \wedge \sigma}\sqrt{\frac{2}{\beta N^3}} \sum_{k=1}^{N} \frac{\rd B_k(s)}{|z_s(u_i) - \lambda_k(s)|^2} ds \leq  \frac{C(\log N)^2}{N \sqrt{\kappa_{t \wedge \sigma} \eta_{t \wedge \sigma}}},
\end{equation}
 where $\kappa_t=\Re[z_t(u_i)]-E_t$ and $\eta_t=\Im[z_t(u_i)]$ for $0\leq t\leq t_i$.
\end{proposition}
\begin{proof}
For simplicity of notations, we write $z_t(u_i)$ as $z_t$,
For a given $t_i$, we define a series of partial stopping times $0=t_i^0< t_i^1<t_i^2<\cdots$, as follows:
\begin{equation}\label{e:choosetik}
t^k_i = t_i \wedge \inf\{t> t^{k-1}_i : \kappa_{t} \eta_{t} < \kappa_{t^{k-1}_i} \eta_{t^{k-1}_i}/2\},\quad k=1,2,3,\cdots.
\end{equation}
Notice that since $\kappa_t$ and $\eta_t$ are finite and cannot be smaller than $N^{-2/3}$, for a given $t_i$, we have $t_i^k=t_i$ for $k\gtrsim \log N$.

We now apply the Burkholder-Davis Gundy inequality to our stochastic integral. The quadratic variation can be found as follows:
\begin{align}\begin{split}
&\phantom{{}={}}\int_{0}^{t^k_{i} \wedge \sigma} \frac{2}{\beta N^3} \sum_{k=1}^{N} \frac{\rd s}{|z_s - \lambda_k(s)|^4}
 \leq \frac{C}{N^2} \int_{0}^{t^k_i \wedge \sigma} \frac{\text{Im}[m_s(z_s)]}{\text{Im}[z_s] ((\kappa_s - f(s))^2 + (\eta_s)^2)} \rd s\\
& \leq \frac{C}{N^2} \int_{0}^{t^k_i \wedge \sigma} \frac{\text{Im}[\hat m_s(z_s)]}{\text{Im}[z_s] ((\kappa_s - f(s))^2 + (\eta_s)^2)} \rd s
\leq \frac{C}{N^{2}} \int_{0}^{t^k_i \wedge \sigma} \frac{ \eta_{s}/{\sqrt{\kappa_{s}}}}{\eta_{s}((\kappa_s)^{1/2}(t_i-s)+ \eta_{s})^2}\rd s \\
&\leq \frac{C}{N^{2}} \int_{0}^{t^k_i \wedge \sigma} \frac{\rd s}{(\kappa_{s})^{3/2}(t_i-s+\eta_s/(\kappa_s)^{1/2})^2}
\leq \frac{C}{N^{2}} \int_{0}^{t^k_i \wedge \sigma} \frac{\rd s}{(\kappa_{t^k_i \wedge \sigma})^{3/2}(t_i-s+\eta_{t^k_i \wedge \sigma}/(\kappa_{t^k_i \wedge \sigma})^{1/2})^2}\\
&\leq  \frac{C} {N^{2}} \frac{1}{\kappa_{t^k_i \wedge \sigma} \eta_{t^k_i \wedge \sigma}}
\end{split}\end{align}
where in the third inequality we use \eqref{e:ksbound}, and
\begin{align}
\frac{\eta_s}{\sqrt{\kappa_s}}\asymp
\Im[\hat m_s(z_s)]\asymp \Im[\hat m_{t_i^k\wedge \sigma}(z_{t_i^k\wedge \sigma})]\asymp \frac{\eta_{t_i^k\wedge \sigma}}{\sqrt{\kappa_{t_i^k\wedge \sigma}}}.
\end{align}
The Burkholder-Davis Gundy inequality implies that with high probability we must have
\begin{equation}\label{localineq}
\sup_{0\leq t \leq t^k_i } \left|\int_{0}^{t\wedge \sigma}\sqrt{\frac{2}{\beta N^3}} \sum_{k=1}^{N} \frac{\rd B_i(s)}{|z_s - \lambda_k(s)|^2} ds \right|\leq \frac{C(\log N)^2}{N \sqrt{\kappa_{t^k_i \wedge \sigma} \eta_{t^k_i \wedge \sigma}}}
\end{equation}
We define $\Omega$ to be the set of Brownian paths $\{B_1(s), \cdots, B_N(s)\}_{0\leq s\leq T}$ on which, \eqref{localineq} holds for all $k$. It follows from the discussion above,  $\Omega$ holds with high probability.
Therefore, for any $t\in[t_{i}^{k-1},t_i^k]$, the bounds \eqref{localineq} and our choice of $t_i^k$ \eqref{e:choosetik} yield that on $\Omega$,
\begin{align}\begin{split}
\left|\int_{0}^{t\wedge \sigma}\sqrt{\frac{2}{\beta N^3}} \sum_{k=1}^{N} \frac{\rd B_i(s)}{|z_s - \lambda_k(s)|^2} ds \right|\leq \frac{C(\log N)^2}{N \sqrt{\kappa_{t \wedge \sigma} \eta_{t \wedge \sigma}}}.
\end{split}
\end{align}
This finishes the proof of proposition \ref{l:StochasticBnd}.
\end{proof}

The last term in \eqref{defcE1} can be estimated by Proposition \ref{prop:HFbound}. Now let us return to the equation \eqref{eq:mzt} for the difference between $\hat{m}_t(z)$ and $m_t(z)$. Fix some $i\leq \lceil TN\rceil$, we denote $z_t=E_t+\kappa_t+\ri\eta_t =z_t(u_i)$ for any time $0\leq t\leq t_i$, thanks to Proposition \ref{l:thirdbound}, \ref{l:StochasticBnd} and \ref{prop:HFbound} we have

\begin{align}\begin{split}\label{e:globalest}
\left|\hat{m}_{ t \wedge \sigma}(z_{ t \wedge \sigma})-m_{ t \wedge \sigma}(z_{ t \wedge \sigma})\right|
\leq& \int_0^{ t\wedge\sigma}\left|\hat{m}_{s}(z_s)-m_s(z_s)\right|\left|\partial_z \left( \hat m_s(z_s)+\frac{V'(z_s)}{2}\right)\right|\rd s\\+&
\frac{C({ t\wedge\sigma})M}{N}
+\frac{C(\log N)^2}{N\sqrt{\eta_{ t \wedge \sigma} \kappa_{ t \wedge \sigma}}} +\left|\hat{m}_0(z_0) - m_0(z_0) \right|.
\end{split}\end{align}
Notice that for $s\leq  t\wedge \sigma$,
\begin{align}\label{e:aterm}
\left|\partial_z \left(\hat m_s(z_s)+\frac{V'(z_s)}{2}\right)\right|
\leq \frac{\text{Im}[ \hat m_s(z_s)]}{\text{Im}[z_s]}+C,
\end{align}
From the definition \eqref{e:defsigma} of $\sigma$, we have that $|m_s(z_s) - \hat{m}_s(z_s)|\leq \text{Im}[\hat{m}_s(z_s)]/\log N$, and it follows
\begin{align}\label{e:aterm2}
\left|\partial_z \left(\hat m_s(z_s)+\frac{V'(z_s)}{2}\right)\right|
\leq \frac{\text{Im}[ \hat m_s(z_s)]}{\text{Im}[z_s]}+C\leq \left(1+\frac{1}{\log N}\right)\frac{\text{Im}[ \hat{m}_s(z_s)]}{\text{Im}[z_s]}+C.
\end{align}
We denote the quantity,
\begin{align} \label{eqn:betabd}
\beta(s):= \left(1+\frac{1}{\log N}\right)\frac{\text{Im}[ \hat{m}_s(z_s)]}{\text{Im}[z_s]}+C= \text{O}\left(\frac{\text{Im}[ \hat{m}_s(z_s)]}{\text{Im}[z_s]}\right),
\end{align}
and rewrite \eqref{e:globalest} as
\begin{align}\begin{split}
\left|\hat{m}_{ t\wedge\sigma}(z_{ t\wedge\sigma})-m_{ t\wedge\sigma}(z_{ t\wedge\sigma})\right|
%\leq C\int_0^{ t\wedge\sigma}\frac{\Im[ m_{s}(z_s)]}{\Im[z_s]}\left|\td m_s(z_s)-m_s(z_s)\right|\rd s+
%\frac{({ t\wedge\sigma})M\log N}{N}\\
%+&\frac{2-\beta}{\beta}\int_0^{ t\wedge\sigma}\frac{\Im[ m(z_s,s)]\rd s}{N\Im[z_s]^2}+(\log N)^{1+\delta}\sqrt{\int_0^{ t\wedge\sigma}\frac{\Im[ m(z_s,s)]\rd s}{N^2\Im[z_s]^3}},\\
\leq& \int_0^{ t\wedge\sigma}\beta(s)\left|\hat{m}_s(z_s)-m_s(z_s)\right|ds\\+&
\frac{C({ t\wedge\sigma})M}{N}
+\frac{C(\log N)^2}{N\sqrt{\eta_{ t \wedge \sigma}\kappa_{ t \wedge \sigma}}}+\left|\hat{m}_0(z_0) - m_0(z_0) \right|.
\end{split}
\end{align}
By Gronwall's inequality, this implies the following estimate for any $ 0\leq t\leq t_i$
\begin{align}\begin{split}\label{e:midgronwall}
&\left|\hat{m}_{t\wedge\sigma}(z_{t\wedge\sigma})-m_{t\wedge\sigma}(z_{t\wedge\sigma})\right|
\leq \frac{C({t\wedge\sigma})M}{N}
+\frac{C(\log N)^2}{N\sqrt{\eta_{t\wedge \sigma} \kappa_{t \wedge \sigma}}}+\left|\hat{m}_0(z_0) - m_0(z_0) \right|\\
+&\int_0^{t\wedge\sigma}\beta(s)\left(\frac{sM(\log N)^{2}}{N}
+\frac{C(\log N)^2}{N\sqrt{\eta_{s}\kappa_{s}}}+\left|\hat{m}_0(z_0) - m_0(z_0) \right|\right)e^{\int_s^{t\wedge\sigma} \beta(\tau)d\tau} \rd s.%\\
%\leq& \frac{C({t\wedge\sigma})M(\log N)^{1+\delta}}{N}
%+\frac{CM(\log N)^{1+\delta}}{N\Im[z_{t\wedge \sigma}(u)]}\int_0^{t\wedge\sigma}s\Im[m_s(z_s(u))]\rd s+\frac{C(\log N)^{2+%\delta}}{N\Im[z_{t\wedge \sigma(u)}]}\\
%\leq& C\left(\frac{({t\wedge\sigma})M (\log N)^{1+\delta} +(\log N)^{2+\delta}}{N\Im[z_{t\wedge\sigma}(u)]}\right).
%\leq& C\left(\frac{({t\wedge\sigma})M(\log N)^{1+\delta}}{N}+\frac{({t\wedge\sigma})M (\log N)^{1+\delta} +(\log N)^{2+\delta}}{N\Im[z_{t\wedge\sigma}(u)]}\right).
\end{split}
\end{align}
For the function $\beta$ we have the following estimates
\begin{align}\begin{split}
\int_s^{t\wedge\sigma} \beta(\tau)\rd \tau
&\leq C(t-s)+\left(1+\frac{1}{\log N}\right)\int_s^{t\wedge\sigma} \frac{\Im[\hat m_s(z_s)]}{\Im[z_s]}\rd \tau\\
&\leq C(t-s)+ \left(1+\frac{1}{\log N}\right)\log \left(\frac{\text{Im}[z_{s}]}{\text{Im}[z_{t\wedge\sigma}]}\right),
\end{split}\end{align}
and thus
\begin{align}\label{e:beta}
e^{\int_s^{t\wedge\sigma} \beta(\tau)d \tau}
\leq e^{C(t-s)} e^{\left(1+\frac{1}{\log N}\right)\log \left(\frac{\text{Im}[z_{s}]}{\text{Im}[z_{t\wedge\sigma}]}\right)}
\leq C  \frac{\text{Im}[z_s]}{\text{Im}[z_{t\wedge\sigma}]},
\end{align}
where in the last equality, we used the estimate $\text{Im}[z_{s}]/\text{Im}[z_{t\wedge\sigma}] \leq C N$.
Combining the above inequality \eqref{e:beta} with \eqref{eqn:betabd} we can bound the last term in \eqref{e:midgronwall} by
\begin{align}\begin{split}\label{e:term2}
&\phantom{{}={}}C\int_0^{t\wedge\sigma}\frac{\text{Im}[ \hat m_{s}(z_s)]}{\text{Im}[z_{t\wedge \sigma}]}\left(\frac{sM}{N}
+\frac{C(\log N)^2}{N\sqrt{\eta_{s}\kappa_{s}}}+\left|\hat{m}_0(z_0) - m_0(z_0) \right|\right) \rd s\\
&\leq \frac{CM}{N\text{Im}[z_{t\wedge \sigma}]}\int_0^{t\wedge\sigma}s\text{Im}[\hat m_s(z_s)] ds+ \int_0^{t\wedge \sigma} \frac{\text{Im}[ \hat m_{s}(z_s)]}{\text{Im}[z_{t\wedge \sigma}]} \frac{C(\log N)^2}{N\sqrt{\eta_{s}\kappa_{s}}} \rd s\\
&+ C(t\wedge \sigma) \left|\hat{m}_0(z_0) - m_0(z_0) \right| \frac{\text{Im}[\hat m_{t\wedge \sigma}(z_{t \wedge \sigma})]}{\text{Im}[z_{t\wedge \sigma}]}.
\end{split}\end{align}
Since $|V' (z)| \leq C$, it follows that $\text{Im}[\hat m_s(z_s)]=-\partial_s \text{Im}[z_s] +O(1)$. Therefore we can bound  the first term in the righthand side of \eqref{e:term2} as
\begin{align}\begin{split}\label{e:term3}
\int_0^{t\wedge\sigma}s\text{Im}[\hat m_s(z_s)]ds = &\int_0^{t\wedge\sigma}(-\partial_s \text{Im}[z_s] )s \rd s + O ( (t \wedge \sigma )^2 ) = O ( t \wedge \sigma ) .
%=&(t\wedge \sigma)(\Im[z_0(u)]-\Im[z_{t \wedge \sigma}(u)])+C(t\wedge \sigma)^2=\OO(t\wedge\sigma).
\end{split}\end{align}
We notice that
\begin{align}
\Im[\hat m_s(z_s)]\asymp \eta_s/(\kappa_s)^{1/2}\asymp\eta_{t\wedge \sigma}/(\kappa_{t\wedge \sigma})^{1/2}.
\end{align}
For the second term in the righthand side of \eqref{e:term2}, we have
\begin{align}\begin{split}
\int_{0}^{t \wedge \sigma} \frac{ \eta_{t\wedge \sigma}/(\kappa_{t\wedge\sigma})^{1/2}}{ \eta_{t\wedge \sigma} N \sqrt{\kappa_s \eta_s}} \rd s
&
= O\left(\int_{0}^{t\wedge \sigma} \frac{1}{N(\eta_{t\wedge \sigma})^{1/2} (\kappa_{t\wedge \sigma})^{1/2}(\kappa_{s})^{1/2}  }\rd s\right) \\
&= O\left(\int_{0}^{t\wedge \sigma} \frac{1}{N(\eta_{t\wedge \sigma})^{1/2} (\kappa_{t\wedge \sigma}) (\sqrt{\kappa_{t\wedge\sigma}}+(t\wedge\sigma-s))  }\rd s\right) \\
&= O\left(\frac{\log N}{N \sqrt{\eta_{t\wedge \sigma} \kappa_{t\wedge \sigma}}}\right).
\end{split}\end{align}
For the last term in the righthand side of \eqref{e:term2}, we have
\begin{align}\begin{split}
(t\wedge \sigma) \left|\hat{m}_0(z_0) - m_0(z_0) \right| \frac{\text{Im}[\hat m_{t\wedge \sigma}(z_{t \wedge \sigma})]}{\text{Im}[z_{t\wedge \sigma}]}
&= O\left(  \frac{(t \wedge \sigma)}{M N \eta_0\sqrt{\kappa_{t\wedge \sigma}}}\right) \\
&= O\left(  \frac{(t \wedge \sigma)}{M N (\eta_{t\wedge \sigma}+(t\wedge \sigma)\eta_{t\wedge \sigma}/\sqrt{\kappa_{t\wedge\sigma}})\sqrt{\kappa_{t\wedge \sigma}}}\right) \\
%&{\cor \leq O((t\wedge \sigma) \frac{1}{N \sqrt{(\eta_{t\wedge \sigma} +(t\wedge \sigma))(\kappa_{t\wedge \sigma} + \sqrt{\kappa_{t\wedge \sigma}})}}\frac{1}{\sqrt{\kappa_{t\wedge \sigma}}}) }\\
&= O\left(\frac{1}{M N \eta_{t\wedge \sigma}}\right).
\end{split}\end{align}

%{\cor It seems we should use
%\begin{align}
%\left|\hat{m}_0(z_0) - m_0(z_0) \right|\leq \frac{1}{N\kappa_0}.
%\end{align}
%}

%Since $|V' (z)| \leq C$, it follows that $\text{Im}[m_s(z_s)]=-\partial_s \text{Im}[z_s] +O(1)$. Therefore we can bound the integral term in \eqref{e:term2} by, as in the previous paper,
%\begin{align}\begin{split}\label{e:term3}
%\int_0^{t\wedge\sigma}s\text{Im}[m_s(z_s)]ds = &\int_0^{t\wedge\sigma}(-\partial_s \text{Im}[z_s] )s ds + O ( (t \wedge \sigma )^2 ) = O ( t \wedge \sigma ) .
%\end{split}\end{align}
%=&(t\wedge \sigma)(\Im[z_0(u)]-\Im[z_{t \wedge \sigma}(u)])+C(t\wedge \sigma)^2=\OO(t\wedge\sigma).
Combining all the above estimates, we have that on the event $\Omega$, for any $i\leq TN$ and $0\leq t\leq t_i$,
\begin{equation}\label{e:outside}
 \left|\hat{m}_{t\wedge \sigma}(z_{t\wedge \sigma}(u_i)) - m_{t\wedge \sigma}(z_{t\wedge \sigma}(u_i)) \right| = O\left(\frac{(\log N)^3}{N \sqrt{\kappa_{t\wedge \sigma} \eta_{t\wedge \sigma}}}+\frac{1}{M N \eta_{t\wedge \sigma}}\right)  \ll \frac{1}{M^{1/4}N \eta_{t\wedge \sigma}},
\end{equation}
where $z_{t\wedge \sigma}(u_i)=E_{t\wedge \sigma}+\kappa_{t\wedge\sigma}+\ri \eta_{t\wedge \sigma}$, and we used \eqref{e:realim}.

In the following we show that on the event $\Omega$, $\sigma=T$, otherwise if there exists a sample in $\Omega$ such that $\sigma<T$. Thanks to \eqref{e:outside}, we must have $\la_\sigma=E_\sigma+f(\sigma)$. We prove this is impossible by contradiction. If $\sigma<T$, then there exists some $i\leq \lceil TN \rceil $, $t_{i-1}< \sigma\leq t_i$. 
We recall that by \eqref{def:fr} and Proposition \ref{prop:TimStab},  $z_\sigma=z_{t_i}+\OO(1/N)$, { $E_\sigma=E_{t_i}+\OO(1/N)$} and $f(\sigma)=f(t_i)+\OO(1/N)$. Therefore, we have $z_\sigma(u_i)=z_{t_i}(u_i)+\OO(1/N)=E_{t_i}+f(t_i)+ \ri M^{-1/3}f(t_i)^{1/4}N^{-1/2}=E_\sigma+f(\sigma)+\ri M^{-1/3}f(t_i)^{1/4}N^{-1/2}+\OO(1/N)$. It follows from the argument as given at the beginning of proof, i.e by taking $E_{\sigma}+\kappa+\eta=z_\sigma$, we get that there is no eigenvalue in a neighborhood of $E_{\sigma}+f(\sigma)$ at time $\sigma$. This leads to a contradiction! This finishes the proof of \eqref{Thm:EdgeRidgity}.

% we only need to consider the effect of the addition of a new initial value term. Inside the Bulk of the spectrum the contribution of the initial value term will be
%\begin{align}
%& C'(t\wedge \sigma) \left|\hat{m}_0(z_0(u)) - m_0(z_0(u)) \right| \frac{\text{Im}[m_{t\wedge \sigma}(z_{t \wedge \sigma}(u))]}{\text{Im}z_{t\wedge \sigma}(u)} \\
%&\le O( (t \wedge \sigma) \frac{Q}{N \eta_0(u)} \frac{\sqrt{\kappa_{t\wedge \sigma}(u)+ \eta_{t\wedge \sigma}(u)}}{\eta_{t\wedge \sigma}}) \\ &\le O((t\wedge \sigma) \frac{Q}{N (\eta_{t\wedge \sigma}(u) +(t\wedge \sigma))(\sqrt{\kappa_{t\wedge \sigma}(u)+ \eta_{t\wedge \sigma}(u)})})\frac{\sqrt{\kappa_{t\wedge \sigma}(u)+ \eta_{t\wedge \sigma}(u)}}{\eta_{t\wedge \sigma}} \\
%&\le O(\frac{Q}{N \eta_{t \wedge \sigma}(u)})
%\end{align}
%
%The above computation was only performed for the set $W_t$. The analysis for $D_t \setminus W_t$ is exactly the same.

\end{proof}

As a consequence of Theorem \ref{Thm:EdgeRidgity} and Proposition \ref{p:rigidity}, by the same argument as \cite[Corollary 3.2]{HL}, we have the following corollary on the locations of extreme eigenvalues.
\begin{corollary} \label{Col:GrenEdg}
Under the assumptions of  \ref{Thm:EdgeRidgity}, we have the following:
 there exists a constant $\fe>0$ such that with high probability under the Dyson Brownian motion \eqref{DBM}, for $ \sqrt{\eta^*}/\fc\leq t\leq T$, and uniformly for indices $1\leq i\leq \fe N$, we have 
\begin{align}
    |\la_i(t)-\gamma_i(t)|\leq \frac{M^2}{N^{2/3}i^{1/3}},
\end{align}
where $\gamma_i(t)$ are the classical particle locations of the density $\hat \rho_t$, i.e.
\begin{align}
    \frac{i-1}{N}=\int_{\gamma_i(t)}^{E_t}\rd \hat \rho_t(x).
\end{align}
\end{corollary}

\section{Mesoscopic Central Limit Theorem}
We will try to prove here a version of the CLT for mesoscopic linear statistics for empirical particle density $\mu_t$ with $\mu_0$ satisfying the assumptions of  \ref{a:initial}; if not explicitly stated, $m_0$ will be assumed to satisfy said hypothesis. 

{
\begin{definition}
We fix a control parameter $\epsilon$.
We define $\cH_t$ to be the region $\{z: - (\Im[z])^{4/5 +  \epsilon} \leq \Re[z]-E_t \leq N^{-\epsilon}, N^{-2/3+\epsilon} \leq \Im[z]\leq N^{-\epsilon}, |z -E_t|\le N^{-\epsilon} \}$. From this point onwards, we will use the notation that if $z_t(u)$ is a characteristic then $\kappa_t(u)= \Re[z_t(u)] - E_t$ and $\eta_t(u)= \Im[z_t(u)]$. We will not always give reference to the parameter u when the context is obvious.
\end{definition}
\begin{remark}
From the equation determining the movement of characteristics, we know that each characteristic moves at $\OO(1)$ in time. Thus, if we choose a time $s<t$ such that $|s-t| \ll (\log N)^{-2}$, then we are assured that if $z_t(u) \in \cH_t$ necessarily we have $|z_s(u) - E_s| \ll (\log N)^{-2}$. In addition, one can check that if $z_t(u)$ is in $\cH_t$ for $t \le N^{-\epsilon}$, then we must necessarily have that $z_s(u)$ is in $\cH_s$ for $s<t$ as the edge moves faster to the right than any characteristic.

\end{remark}
}

Our goal is to prove the following theorem
\begin{theorem} \label{t:mesoCLT}
Let $m(z)$ satisfy \ref{a:initial}. First fix a scale $\eta$ satisfying $N^{-2/3 + \epsilon} \ll \eta^* \ll  \eta \ll N^{-\epsilon} $. Consider complex numbers $w_1, w_2, \cdots,w_n$ and a time $t$ satisfying $\sqrt{ \eta} N^{\epsilon} \leq t \leq (\log N)^{-4} $. Then the rescaled quantities $\Gamma_{t}[E_t +w_i \eta] = N \eta \left[m_{t}(E_{t} + w_i  \eta) - \hat{m}_{t}(E_t + w_i \eta)\right] - \frac{2- \beta}{4 \beta w_i}$ asymptotically form a Gaussian Field with limiting Covariance Kernel
\begin{align*}
    K_{\rm edge}(w_i,w_j):=\lim_{N\rightarrow \infty}N^2 \cov\langle \Gamma_t(E_t+w_i\eta), \Gamma_t(E_t+w_j\eta) \rangle=\frac{1}{2\beta \sqrt{w_i}\sqrt{w_j}(\sqrt{w_i}+\sqrt{w_j})^2},
\end{align*}

\begin{comment}
the following characteristic function up to a factor of $(1+ \OO(N^{-\epsilon}))$

\begin{align}\begin{split} \label{e:variance}
& -\sum_{1\leq j,\ell\leq k}\Re\left[\frac{(a_j-\ri b_j)(a_\ell+\ri b_\ell)\Im[z_t(u_j)]\Im[z_t(u_\ell)]}{4\beta\sqrt{\kappa_t(u_j) - i \eta_t(u_j)}\sqrt{\kappa_t(u_\ell)+i \eta_t(u_\ell)} (\sqrt{\kappa_t(u_j) - i \eta_t(u_j)}+ \sqrt{\kappa_t(u_\ell) + i \eta_t(u_\ell)})^2}\right]\\
&-\sum_{1\leq j,\ell\leq k}\Re\left[\frac{(a_j+\ri b_j)(a_\ell+\ri b_\ell)\Im[z_t(u_j)]\Im[z_t(u_\ell)]}{4\beta \sqrt{\kappa_t(u_j) + i \eta_t(u_j)}\sqrt{\kappa_t(u_\ell)+i \eta_t(u_\ell)}(\sqrt{\kappa_t(u_j) + i \eta_t(u_j)}+ \sqrt{\kappa_t(u_\ell) + i \eta_t(u_\ell)})^2}\right]\\
&-\sum_{1\leq j,\ell\leq k}\Re\left[\frac{(a_j-\ri b_j)(a_\ell-\ri b_\ell)\Im[z_t(u_j)]\Im[z_t(u_\ell)]}{4\beta\sqrt{\kappa_t(u_j) - i \eta_t(u_j)}\sqrt{\kappa_t(u_\ell)-i \eta_t(u_\ell)}(\sqrt{\kappa_t(u_j) - i \eta_t(u_j)}+ \sqrt{\kappa_t(u_\ell) - i \eta_t(u_\ell)})^2}\right] 
\end{split}\end{align}

\end{comment}

\end{theorem}
\begin{remark}
We will remark here that this result only gives the leading order when $\kappa_t = \OO(\eta_t)$. Regardless, the variance bound in the larger region is important later.
\end{remark}
\begin{proof}[Proof of Theorem \ref{t:mesoCLT}]
Let the event $\Omega$ be as in Theorem \ref{p:rigidity}.  Thanks to the estimates \eqref{e:diffmm} and Lemma \ref{dpbound} which hold on $\Omega$, we can bound the second term on the RHS of \eqref{e:diffm}  by first splitting it into an error part and main term part. The following is an estimate of the error term.
\begin{align}\begin{split}
\left|\left(m_t(z_t)- \hat{m}_t(z_t)\right)\del_z \left( (m_t(z_t) - \hat{m}_t(z_t))+\frac{V'(z_t)}{2}\right)\right|
\leq& \frac{M}{N\Im[z_t]}\left(\frac{M}{N\Im[z_t]^{2}}+ 1\right)\\
=& \OO\left(\frac{M^2}{(N^2\Im[z_t]^{3})}+\frac{M}{N \Im[z_t]}\right),
\end{split}\end{align}
The main term coming from this contribution is
\begin{equation}
    \left(m_t(z_t)- \hat{m}_t(z_t)\right)\del_z \left( \hat{m}_t(z_t) \right)
\end{equation}
%
%  $\Im[m_s(z_s)]\geq (2\fd)^{-1}$ from \eqref{e:Immsbound}, and $\Im[u]/\Im[z_t(u)]\leq CN$ from \eqref{e:z0ztbound},
%\begin{align}\label{e:integralbound}
%\int_0^t\frac{\rd s}{\Im[z_s]^p}\leq 2\fd \int_0^t\frac{\Im[m_s(z_s)]}{\Im[z_s]^p}\rd s \leq C\log N,
%\quad p=1;\qquad \leq \frac{C}{\Im[z_t]^{p-1}},\quad p=2,3,4.
%\end{align}
%Since both $\td m_s$ and $m_s$ are analytic on the upper half plane, combining with \eqref{e:dzmsbound}, we have
%\beq\label{e:contourest}
%\del_z \td m_s(z_s)
%=\del_z \left(\td m_s(z_s)-m_s(z_s)\right)+\del_z m_s(z_s)\\
%=\frac{1}{2\pi \ri}\oint_{\cC} \frac{\td m_s(w)-m_s(w)}{(w-z_s)^2}\rd w+O\left(\frac{1}{t}\right)
%\eeq
%where $\cC$ is a small contour in the upper half plane centering at $z_s$ with radius $\Im[z_s]/2$. On the event $\Omega$, by \eqref{e:diffmm} in Theorem \eqref{t:rigidity}, the contour integral can be bounded,
%\beq
%\left|\frac{1}{2\pi \ri}\oint_{\cC} \frac{\td m_s(w)-m_s(w)}{(w-z_s)^2}\rd w\right|
%\leq \frac{1}{2\pi }\oint_{\cC} \frac{|\td m_s(w)-m_s(w)|}{|w-z_s|^2}\rd w
%=\OO\left(\frac{M}{N\Im[z_s]^2}\right).
%\eeq
%Therefore, on the event $\Omega$, the first term on the righthand side of \eqref{e:intdiffm} can be bounded by
%\begin{align*}
%\left|\int_0^t\left(\td m_s(z_s)-m_s(z_s)\right)\del_z \left(\td m_s(z_s)+\frac{V'(z_s)}{2}\right)\rd t\right|
%\leq& C\int_0^t\frac{M}{N\Im[z_s]}\left(\frac{M}{N\Im[z_s]^2}+\frac{1}{t}\right)\rd t\\
%\leq& C\left(\frac{M^2}{(N\Im[z_t])^2}+\frac{M\log N}{Nt}\right),
%\end{align*}
%where we used \eqref{e:propzt} and \eqref{e:Immsbound}
%
The contour integral on the righthand side of \eqref{e:diffm} is an error term. Using  Proposition \ref{prop:HFbound}, we have on the event $\Omega$% (as defined in Proposition \ref{p:error}), we have
\beq
 \left|\oint_{\mathcal{C}_t} g(z_t,w) (m_t(w)- \hat{m}_t(w)) \rd w \right|\leq \frac{CM}{N}.
\eeq
We can rewrite the final term on the righthand side of \eqref{e:diffm} as
\beq
\frac{2-\beta}{\beta N^2}\sum_{i=1}^{N}\frac{1}{(\la_i(t)-z_t)^3}
=\frac{2-\beta}{2\beta N} \del_z^2 (m_t(z_t) - \hat{m}_t(z_t)) +\frac{2-\beta}{2\beta N}  \partial_z^2 \hat{m}_t(z_t)  
\eeq
Thanks to Lemma \ref{dpbound}, we have
\begin{align}\begin{split}
\left|\del_z^2 (m_t(z_t) - \hat{m}_t(z_t))\right|
=\OO\left(\frac{1}{N (\Im[z_t])^{3}}\right).
\end{split}
\end{align}
\begin{comment}
It follows that
\beq
\left|\frac{2-\beta}{\beta N^2}\int_0^t\sum_{i=1}^{N}\frac{\rd s}{(\la_i(s)-z_s)^3}\right|
=\OO\left(\frac{1}{(N^2\Im[z_t]^{3/2})}+\frac{\log N}{Nt}\right).
\eeq
\end{comment}
%
We have the following differential equation in order to study the fluctuations of $m_t(z_t) - \hat{m}_t(z_t)$
\begin{align}\begin{split}
    \partial_t(m_t(z_t) - \hat{m}_t(z_t)) 
    &= (m_t(z_t) - \hat{m}_t(z_t)) \partial_z \hat{m}_t(z_t) + \frac{2-\beta}{2 \beta N} \partial_z^2 \hat{m}_t(z_t) \\
    &- \sqrt{\frac{2}{\beta N^3}} \sum_{i=1}^N \frac{ {\rm d} B_i(t)}{(\la_i(s) - z_s)^2} + \OO\left(\frac{1}{N^2 \Im[z_t]^3}\right)
\end{split}\end{align}
We can explicitly solve the above equation by using $\mathcal{I}_t:=\exp{\int_{0}^t \partial_z \hat{m}_s(z_s) \rd s}$ as an integrating factor.
The solution can be explicitly written up as
\begin{align} \label{e:StochasDiff}
    \begin{split}
        m_t(z_t) - \hat{m}_t(z_t) &=\mathcal{I}_t (\mathcal{I}_s)^{-1} (m_s(z_s) - \hat{m}(z_s))  + \mathcal{I}_t \int_s^{t}\mathcal{I}_q^{-1} \left( \frac{2-\beta}{2\beta N} \partial_z^2 \hat{m}_q(z_q) +\right. \\&
        \left.- \sqrt{\frac{2}{\beta N^3}} \sum_{i=1}^{N} \frac{{\rm d}B_i(q)}{(\la_i(q)-z_q)^2} + \OO\left( \frac{1}{N^2 \Im[z_q]^3}\right) \right) \rd q
    \end{split}
\end{align}
The deterministic integral in the above line is an offset term for the mean value. The stochastic integral is the cause of the gaussian fluctuation.
From Lemma \ref{lem:IntFac}, which will be shown later, we can proceed further and evaluate the quantities that appear in the integral of \ref{e:StochasDiff}.
Choosing $s$ so that $(\sqrt{|\kappa_t + i \eta_t|})^{1- \epsilon} \ll t-s \ll (|\kappa_t + i \eta_t|)^{1/4 + \epsilon}$, we can evaluate
\begin{align}
    \begin{split}
        \frac{2-\beta}{2\beta N} \frac{1}{\sqrt{\kappa_t + i \eta_t}} \int_s^{t} \partial_z^2 \hat m_q(z_q) \sqrt{\kappa_q + i \eta_q} \rd q & =  \frac{2-\beta}{2 \beta N} \frac{1}{\sqrt{\kappa_t + i \eta_t}} [\int_s^{t} \frac{C_q \pi}{4(\kappa_q + i \eta_q)} \rd q+ \OO((s-t)(\log N)^5 )]\\
        &= \frac{2-\beta}{4 \beta N} \frac{1}{(\kappa_t + i \eta_t)}[1+ \OO(N^{-\epsilon})]
    \end{split}
\end{align} 
The bound we have on $[m_s(z_s) - \hat{m}_s(z_s)] \mathcal{I}_t(\mathcal{I}_s)^{-1}$ is $\OO \left(\frac{1}{N\sqrt{\kappa_s+ i \eta_s} \sqrt{\kappa_t + i \eta_t}} \right)$ where we applied rigidity at time $t$. This is clearly of much smaller order than $\frac{1}{N (\kappa_t + i \eta_t)}$
%
%
%\begin{align}\begin{split}
%\int_0^t\del_z^2 \td m_s(z_s)\rd s
%=\int_0^t\frac{1}{\pi \ri}\oint_{\cC} \frac{\td m_s(w)-m_s(w)}{(w-z_s)^3}\rd w\rd s
%+\int_0^t\frac{1}{2\pi \ri}\oint_{\cC} \frac{\del_z m_s(w)}{(w-z_s)^2}\rd w\rd s
%\end{split}\end{align}
%where $\cC$ is a small contour in the upper half plane centering at $z_s$ with radius $\Im[z_s]/2$.
%
%On the event $\Omega$, by \eqref{e:diffmm} in Theorem \eqref{t:rigidity}, and \rn{2} in Corollary \ref{c:uniformest}, we have
%\begin{align}\begin{split}
%\left|\int_0^t\del_z^2 \td m_s(z_s)\rd s\right|
%=&\int_0^t\frac{1}{\pi}\oint_{\cC} \frac{|\td m_s(w)-m_s(w)|}{|w-z_s|^3}\rd w\rd t
%+\int_0^t\frac{1}{2\pi }\oint_{\cC} \frac{|\del_z m_s(w)|}{|w-z_s|^2}\rd w\rd t\\
%\leq& C\left(\int_0^t\frac{\rd s}{N\Im[z_s]^3}
%+\int_0^t \frac{\rd s}{t\Im[z_s]}\right)\leq C\left(\frac{1}{(N\Im[z_t])^2}+\frac{\log N}{t}\right).
%\end{split}
%\end{align}
%It follows that
%\beq
%\left|\frac{2-\beta}{\beta N^2}\int_0^t\sum_{i=1}^{N}\frac{\rd s}{(\la_i(s)-z_s)^3}\right|
%=C\left(\frac{1}{(N^2\Im[z_t])^2}+\frac{\log N}{Nt}\right).
%\eeq
By combining the above estimates we see that on the event $\Omega$, we have
\beq\label{e:tmmtdiff}
m_t(z_t)- \hat{m}_t(z_t) = \frac{2 - \beta}{4 \beta N} \frac{1}{\kappa_t + i \eta_t} [1+ O(N^{-\epsilon})] +\sqrt{\frac{2}{\beta N^3}}\int_s^t \mathcal{I}_t (\mathcal{I}_s)^
{-1} \sum_{i=1}^N \frac{{\rm d} B_i(q)}{(\la_i(q)-z_q)^2} + \OO(\frac{1}{N^2 \Im[z_t]^{5/2}}).
\eeq
We remark at this point that since we have that $\Im[z_t] \gg N^{-2/3}$, the final term in $\eqref{e:tmmtdiff}$ will be less than the previous two terms and can essentially be absorbed into the $\OO(N^{-\epsilon})$ factor appearing above.

In the following we show that the Brownian integrals are asymptotically jointly Gaussian. We fix $z_1,z_2\cdots, z_k  \in H_t$ such that there exists a constant $B>1$ such that for all $ i,j$, $B^{-1}\leq \Im[z_j](\Im[z_i])^{-1}\leq B $ and let $u_1,u_2,\cdots, u_k$ be points such that $z_t(u_i) = z_i$ for $i=1,2,\cdots,k$ respectively.
For $1\leq j\leq k$, let
\beq
X_j(t)= \Im[z_t(u_j)]\sqrt{\frac{2}{\beta N}}\int_s^t\sum_{i=1}^N  \mathcal{I}_t (\mathcal{I}_s)^{-1}\frac{{\rm d} B_i(t)}{(\la_i(s)-z_s(u_j))^2},\quad j=1,2,\cdots, k.
\eeq
where $s$ is a time such that $(\log N)^{-2} \max(\sqrt{|\kappa_t(u_1) + i \eta_t(u_i)|})^{1/2+ \epsilon}\gg (t-s) \gg (\sqrt{\max_{i=1,2,\cdots,n}(\eta_t(u_i)})^{1-\epsilon}$. Such a time exists based on how we chose our points $z_1, z_2,\cdots, z_n$.

We compute their joint characteristic function,
\beq\label{e:cfunc}
\bE\left[\exp\left\{\ri\sum_{j=1}^k a_j\Re[X_j(t)]+b_j\Im[X_j(t)]\right\}\right]
\eeq
Since $\sum_{j=1}^k a_j\Re[X_j(t)]+b_j\Im[X_j(t)]$ is  a martingale, the following is also a martingale
\beq
\exp\left\{\ri \sum_{j=1}^k \{a_j\Re[X_j(t)]+b_j\Im[X_j(t)]\}+\frac{1}{2}\left\langle \sum_{j=1}^k a_j\Re[X_j(t)]+b_j\Im[X_j(t)]\right\rangle\right\}
\eeq
In particular, its expectation is one. By computations performed later in Proposition \ref{p:var} and Lemma \ref{Lem:VarEval}, on the event $\Omega$ , the quadratic variation is given by
\begin{align}\begin{split}
&\frac{1}{2}\left\langle \sum_{j=1}^k a_j\Re[X_j(t)]+b_j\Im[X_j(t)]\right\rangle [1+\OO(N^{-\epsilon})]\\
=& -\sum_{1\leq j,\ell\leq k}\Re\left[\frac{(a_j-\ri b_j)(a_\ell+\ri b_\ell)\Im[z_t(u_j)]\Im[z_t(u_\ell)]}{16\beta\sqrt{\kappa_t(u_j) - i \eta_t(u_j)}\sqrt{\kappa_t(u_\ell)+i \eta_t(u_\ell)} (\sqrt{\kappa_t(u_j) - i \eta_t(u_j)}+ \sqrt{\kappa_t(u_\ell) + i \eta_t(u_\ell)})^2}\right]\\
&-\sum_{1\leq j,\ell\leq k}\Re\left[\frac{(a_j+\ri b_j)(a_\ell+\ri b_\ell)\Im[z_t(u_j)]\Im[z_t(u_\ell)]}{16\beta \sqrt{\kappa_t(u_j) + i \eta_t(u_j)}\sqrt{\kappa_t(u_\ell)+i \eta_t(u_\ell)}(\sqrt{\kappa_t(u_j) + i \eta_t(u_j)}+ \sqrt{\kappa_t(u_\ell) + i \eta_t(u_\ell)})^2}\right]\\
&-\sum_{1\leq j,\ell\leq k}\Re\left[\frac{(a_j-\ri b_j)(a_\ell-\ri b_\ell)\Im[z_t(u_j)]\Im[z_t(u_\ell)]}{16\beta\sqrt{\kappa_t(u_j) - i \eta_t(u_j)}\sqrt{\kappa_t(u_\ell)-i \eta_t(u_\ell)}(\sqrt{\kappa_t(u_j) - i \eta_t(u_j)}+ \sqrt{\kappa_t(u_\ell) - i \eta_t(u_\ell)})^2}\right] 
\end{split}\end{align}
The value of \eqref{e:cfunc} is clearly the exponential of the upper quantity.
\begin{comment}
Therefore,
\begin{align}\begin{split}
\eqref{e:cfunc}
=&\exp\left\{\sum_{1\leq j,\ell\leq k}\Re\left[\frac{(a_j-\ri b_j)(a_\ell+\ri b_\ell)\Im[z_t(u_j)]\Im[z_t(u_\ell)]}{2\beta(z_t(u_j)-\overline{z_t(u_\ell)})^2}\right]\right\} \\+&\OO\left( \frac{M}{N\min_j\{\Im[z_t(u_j)]\}}+\frac{\max_j\{\Im[z_t(u_j)]\}}{t}\right).
\end{split}\end{align}
\end{comment}
Since by \eqref{e:tmmtdiff},
\begin{align*}
\Gamma_t(z_t(u_j))=X_j(t)+\OO(N^{-\epsilon}),
\end{align*}
and so we have a gaussian field. Choosing $z_t(u_j) = E_t + w_j \eta$, we get Theorem \ref{t:mesoCLT}.
\end{proof}

From the above computation, we get the kernel, defined as 
\begin{align*} \label{e:CharFunc}
    K_{\rm edge}(w,w')=\frac{1}{2\beta \sqrt{w}\sqrt{w'}(\sqrt{w}+\sqrt{w'})^2}
\end{align*}
where $w,w'\in \bC\setminus \bR_{-}$. We recall the kernel in the bulk is given by
\begin{align}
    K_{\rm bulk}(w,w')=\frac{2}{\beta (w-w')^2},
\end{align}
provided $w\in \bC_+, w'\in \bC_-$, or $w\in \bC_-,w'\in \bC_+$, otherwise $K_{\rm bulk}(w,w')=0$. We can in fact recover the kernel in the bulk from the kernel in the edge. We take $w=\kappa+\ri \eta$, and $w=\kappa'+\ri \eta'$, and let $\kappa,\kappa'$ tend to $-\infty$,
\begin{align}
    K_{\rm edge}(w,w')
    \rightarrow
     \left\{
    \begin{array}{cc}
          0 & \text{if $\eta\eta'>0$,}\\
          K_{\rm bulk}(w,w') & \text{if $\eta\eta'<0$.}\\
    \end{array}
    \right.
\end{align}

\begin{corollary}\label{c:mesoCLT}
Under the assumptions of Theorem \ref{t:mesoCLT}, the following holds for any compactly supported test function $\psi$ in the Sobolev space $H^{s}$ with $s>1$. Let { $N^{-2/3+ \epsilon}\ll \eta^{*} \ll \eta\ll t \ll N^{-1/2 - \epsilon}$} and define
\begin{align}
\psi_{\eta}(x)=\psi\left(\frac{x-E_t}{\eta}\right).
\end{align}
The normalized linear statistics converges to a Gaussian
\beq\label{e:toGaussian}
\hat{\cal L}(\psi_{\eta})\deq \sum_{i=1}^N \psi_{\eta}(\lambda_i(t))-N\int_{\bR} \psi_{\eta,E}(x) \rd \rho_t(x)\rightarrow N(0, \sigma_\psi^2) - \frac{2-\beta}{ 4  \beta} \psi(0),
\eeq
in distribution as $N\rightarrow \infty$, where 
%is asymptotically a Gaussian variable, 
%\beq
%\label{characteristic}\bE\left[e^{\ri \theta \cal L(\psi_{\eta,E})}\right]=e^{-\theta^2\sigma_\psi^2/{2}}+\OO\left(\frac{M^2}{N\eta}+\frac{M\eta\log N}{t}\right),
%\eeq
%where 
\beq
\sigma_\psi^2\deq \frac{1}{4\pi^2 \beta}\int_{\bR^2} \left(\frac{\psi(x^2)-\psi(y^2)}{x-y}\right)^2\rd x\rd y.
\eeq
\end{corollary}

\subsection{Preliminary Estimates}

We will start by proving some of the simple estimates that have appeared in the derivation of the covariance Kernel. These quantities will also reoccur frequently during later computations.

\begin{lemma} \label{IntegralBounds}
We have for those points such that $z_t(u) \in H_t$
\begin{equation}
\int_{0}^{t} \frac{1}{(\Im[z_s(u)])^{p}} ds=  \OO\left( \frac{M}{\eta_t^{p-1/2}}\right)
\end{equation}
where the constant appears above is independent of $N$ and $u$ in $H_t$
%and we also have
%\begin{equation}
%\int_{0}^{t} \frac{1}{(\Im[z_s(u)])^{p}\sqrt{\Im[z_s(u)] \Re[z_s(u)]} } ds =\OO \left( \frac{1}{(\eta_t)^{p+\frac{1}{2}} }\right)
%\end{equation}
\end{lemma}
\begin{proof}
We can perform explicit calculation due to our assumption of square root behavior.
\begin{equation}
\int_{0}^{t} \frac{1}{(\Im[z_s(u)])^p} ds =\OO\left( \int_{0}^{t} \frac{M}{(\eta_t + (t-s) \sqrt{\eta_t})^p} ds \right)
%= \OO \left( \frac{1}{(\eta_t^{p/2})} \frac{1}{(\frac{(s-t)}{\sqrt{C+1}}+ \sqrt{\eta_t})^{p-1}}|_{0}^{t} \right)
= \OO \left(\frac{M}{ \eta_t^{p-1/2}}\right)
\end{equation}
%\end{proof}
%
%For the other computation, we have
%\begin{align}
%\int_{0}^{t} \frac{1}{(\Im[z_s(u)])^{p}\sqrt{\Im[z_s(u)] \Re[z_s(u)]} } ds = &\OO\left( \int_{0}^{t} \frac{1}{(\eta_t + (t-s) \frac{\eta_t}{\sqrt{\kappa_t}})^{p+\frac{1}{2}} (\sqrt{\kappa_t} +(t-s))} ds \right)\\
%=& \OO \left( \frac{1}{(\frac{\eta_t}{\sqrt{\kappa_t}})^{p+\frac{1}{2}}} \frac{1}{((s-t)+ \sqrt{\kappa_t})^{p +1/2}}|_{0}^{t} \right) = \OO \left( \frac{1}{(\eta_t)^{p+\frac{1}{2}} }\right)
%\end{align}
\end{proof}
%
%\subsection{Proof of Mesoscopic Central Limit Theorem}
%We can get regularity of derivatives of $m_t$ from regularity of the derivatives of $\hat{m}_t$. Recall from the previous section, that we want to consider mesoscopic statistics in the region $z \in H_t$ where we would have the following rigidity of the Green's function: $|m_t(z) - \hat{m}_t(z)| \le \frac{1}{N \sqrt{(\Re[z]- E_t) \Im[z]}}$ provided we have sufficiently good rigidity in $z_t^{-1}(H_t)$
%
%
%
%

The following lemma estimates various quantities that will reoccur when one tries to compare the measure $m_t$ to the stable measure $\hat{m}_t$
\begin{lemma}\label{dpbound}
Suppose that the assumptions of Proposition \ref{p:rigidity} hold.   Fix $u\in H_t$. If u is of the form $z_t(x)$ for some $x$ in $z_t^{-1}(H_t)$, then on the event $\Omega$  as given by \ref{p:rigidity}, we have the following estimate uniformly for $0\leq s\leq t$, with the constants appearing below independent of $x$ and $N$.
\begin{equation} \label{dptdms}
\del_z^p (m_s(z_s(x)) - \hat{m}_s(z_s(x)))=\OO\left(\frac{M}{N (\Im[z_s(x)])^{p+1}}\right)
\end{equation}
\end{lemma}

%And especially,
%\beq \label{dptdms}
%\del_z^p \td m_s(z_s(u))=\OO\left(\frac{M}{N\Im[z_s(u)]^{p+1}}+\frac{1}{t\Im[z_s(u)]^{p-1}}\right)
%\eeq

\begin{proof}

Since both $m_s$ and $\hat m_s$ are analytic on the upper half plane, by Cauchy's integral formula
\beq
\del_z^p \left(m_s(z_s(x))-\hat m_s(z_s(x))\right)
=\frac{p!}{2\pi \ri}\oint_{\cC} \frac{m_s(w)-\hat m_s(w)}{(w-z_s(x))^{p+1}}\rd w ,
\eeq
where $\cC$ is a small contour in the upper half plane centering at $z_s(x)$ with radius $\Im[z_s(x)]/2$. On the event $\Omega$, we use  \eqref{e:diffmm} in Proposition \ref{p:rigidity} to bound the integral by %, the contour integral can be bounded,
\begin{align}\begin{split}
\left|\frac{p!}{2\pi \text{i}}\oint_{\cC} \frac{m_s(w)-\hat m_s(w)}{(w-z_s(x))^{p+1}}\rd w\right|
\leq & \frac{p!}{2\pi }\oint_{\cC} \frac{|m_s(w)-\hat m_s(w)|}{|w-z_s(x)|^{p+1}}\rd w
\\
= & \OO\left( \frac{M}{N  (\Im[z_s(x)])^{p+1}} \right).
\end{split}\end{align}
\end{proof}

%\begin{subsection}[Proof of Theorem \ref{t:mesoCLT}]

%It is clear from how we defined the regions $R_0$ and $R_t$ that any point u in $R_t\cup\[z: \text{Im[z]}\]$

%By \rn{1}) in Corollary \ref{c:uniformest}, $\{E+\ri \eta: E\in[E_0-r/2, E_0+r/2], M^2/N\ll \eta\ll t\}\subseteq z_t(\{E+\ri\eta: E\in[E_0-r, E_0+r], \eta\in [\eta^*, 1]\}\cap \Omega_t)$. In the following, we fix some $u\in\{E+\ri\eta: E\in[E_0-r, E_0+r], \eta\in [\eta^*, 1]\}$, such that $z_t(u)\in\{E+\ri \eta: E\in[E_0-r/2, E_0+r/2], M^2/N\ll \eta\ll t\}$.

 %By our assumption $\fd^{-1}\leq \Im[m_0(u)]\leq \fd$, combining with \eqref{e:mtbound}, uniformly for any $0\leq s\leq t$, we have
%\beq\label{e:Immsbound}
%(2\fd)^{-1}\leq \Im[m_s(z_s(u))]\leq 2\fd,
%\eeq
%since $t\ll1$. Therefore, {\cor $z_t(u)\in \dom_t$} (as in \eqref{def:dom}),
%
%By Lemma \ref{l:zsinDs}, $z_t\in \dom_t$, and the local law of  Theorem \ref{t:rigidity} holds.

 In the following proposition we attempt to calculate the Quadratic Variation of the stochastic integral that appears in the proof of \ref{t:mesoCLT}. We need the following lemma on the behavior of the integrating factor $\mathcal{I}_t(\mathcal{I}_s)^{-1}$ whose proof will be found near the end of the section after more technical details have been established.
 
 \begin{lemma} \label{lem:IntFac}
 Let $\hat{m}_t$ be a Green's function associated to a stable measure. Consider a characteristic $z_t = E_t + \kappa_t + i \eta_t$ and ,along this characteristic, consider times $s$ and $t$ such that $(s-t)^2= O(\sqrt{\kappa_t + i \eta_t}^{1+ \tilde \epsilon})$.  Then we have
 \begin{equation}
     (\mathcal{I}_t)(\mathcal{I}_s)^{-1}= \exp{\int_s^t \partial_z \hat{m}_s(z_s)} = \frac{\sqrt{\kappa_s+ i \eta_s}}{\sqrt{\kappa_t + i \eta_t} }(1+ \OO(N^{- \tilde  \epsilon}))
 \end{equation}
 
 \end{lemma}
 
 Replacing this integrating factor with the above term, we have the following Quadratic Variance integrals.
 
 \begin{comment}
{\cob
We integrate both sides of \eqref{e:diffm}, and get the following integral expression for $ m_t(z_t)$,
\begin{align}\label{e:intdiffm}\begin{split}
&m_t(z_t)- \hat{m}_t(z_t)=(m_0(z_0) - \hat{m}_0(z_0))+\int_0^t\left(m_s(z_s)-\hat{m}_s(z_s)\right)\del_z \left( m_s(z_s)+\frac{V'(z_s)}{2}\right)\rd s\\
+& \frac{1}{\pi}\int_0^t\int_{\bC} \del_{\bar w} \td g(z_s,w) ( m_s(w)- \hat{m}_s(w))\rd^2 w\rd s
+\frac{2-\beta}{\beta N^2}\int_0^t\sum_{i=1}^{N}\frac{\rd s}{(\la_i(s)-z_s)^3}\\
-&\sqrt{\frac{2}{\beta N^3}}\int_0^t\sum_{i=1}^N \frac{{\rm d} B_i(s)}{(\la_i(s)-z_s)^2}.
\end{split}
\end{align}
For the  proof of the mesoscopic central limit theorem, we will show that the first four terms on the righthand side of \eqref{e:intdiffm} are negligible, and the Gaussian fluctuation is from the last term, i.e. the integral with respect to Brownian motion. In the following Proposition, we calculate the quadratic variance of the Brownian integrals.
}
\end{comment}
\begin{proposition}\label{p:var}
Suppose that the  assumptions of Corollary \ref{Col:GrenEdg} hold. Consider two points $u,u'$ in $z_t^{-1}(H_t)$. We will use the notation $z_t(u) = E_t + \kappa_t + i \eta_t$ and $z_t(u') = E_t + \kappa_t' + i \eta_t'$ and, without loss of generality, assume $\eta_t< \eta_t'$. Along these characteristics, consider times $(\log N)^{-3} \gg t,s \gg \sqrt{\eta^*}$ with $(t-s)^2 = \OO(\sqrt{\eta_t}^{1+\tilde \epsilon})$ . Then we have the following expressions for the quadratic variance.
\begin{align}
\label{e:var1}&\frac{1}{N^3}\int_s^t\sum_{i=1}^N [(\mathcal{I}_t)(\mathcal{I}_q)^{-1}]^2\frac{{\rm d}q}{(\la_i(q)-z_q)^4}  = \frac{1}{6N^2}\int_s^t \frac{\kappa_q + i \eta_q}{\kappa_t + i \eta_t} \del^3_z\hat{m}_q(z_q)\rd q [1+ O(N^{- \tilde \epsilon})] +\OO\left( \frac{M}{N^3\eta_t^{7/2}} \right) , \\
\label{e:var2}&\frac{1}{N^3}\int_s^t\sum_{i=1}^N [(\mathcal{I}_t)(\mathcal{I}_q)^{-1}][(\mathcal{I}_t')(\mathcal{I}_q')^{-1}]\frac{{\rm d}q}{(\la_i(q)-z_q)^2(\la_i(q)-z_q')^2}=\\
&\frac{1}{2\pi\ri N^2}\int_{s}^{t} \frac{\sqrt{\kappa_q + i \eta_q}\sqrt{\kappa_q' + i \eta_q'}}{\sqrt{\kappa_t + i \eta_t}\sqrt{\kappa_t' + i \eta_t'}} \oint_{\cal C}  \frac{\hat{m}_q(w) }{(w-z_q)^2(w-z_q')^2}\rd w \rd q[1+ \OO(N^{-\tilde \epsilon})]  + \OO\left(\frac{M}{N^3 \eta_t^{7/2}}+\frac{M}{N^3 \eta_t'^{7/2}}\right)  , \\
\label{e:var3}&\frac{1}{ N^3}\int_s^t[\overline{(\mathcal{I}_t)(\mathcal{I}_q)^{-1}}][(\mathcal{I}_t')(\mathcal{I}_q')^{-1}] \sum_{i=1}^N \frac{{\rm d} q}{(\la_i(q)-\bar{z}_q)^2(\la_i(q)-z_q')^2}=\\
& \int_{s}^{t}\frac{\sqrt{\kappa_q - i \eta_q}\sqrt{\kappa_q' + i \eta_q'}}{\sqrt{\kappa_t - i \eta_t}\sqrt{\kappa_t' + i \eta_t'}}\left[-
\frac{2(\overline{- \hat{m}_q(z_q)}+ \hat{m}_q(z_q'))}{(\bar z_q-z_q')^3}+\frac{\overline{\del_{z}  \hat{m}_q( z_q)}+\del_z \hat{m}_q(z_q')}{(\bar z_q-z_q')^2} \right] \rd q [1+ \OO(N^{-\tilde \epsilon})]\\
&+ \OO\left(\frac{M}{N^3 \eta_t^{7/2}}+\frac{M}{N^3 \eta_t'^{7/2}}\right).
\end{align}
\end{proposition}
\begin{proof}
 
%\begin{align}\begin{split}
    %\int_{s}^t \partial_z \hat{m}_s(z_s) \rd s &= \int_s^t\int_{0}^{\infty} \frac{C_q \sqrt{x}}{(x+ \kappa_q + i \eta_q)^2} \rd x+\OO(\log N^{-1})  \rd t \\
    %&= \int_{s}^{t} \frac{C_q \pi}{2 \sqrt{\kappa_q + i \eta_q}} \rd q +\OO((t-s) \log N^{1/2})\\
    %&= \ln\left(\frac{\sqrt{\kappa_s + i \eta_s}}{\sqrt{\kappa_t + i \eta_t}} \right)(1 +\OO(N^{-\epsilon}))
%\end{split}\end{align}

%asymp $\Im[z_t']+(t-s)$. Since $\Im[z_t]\geq \Im[z_t']$, there exists a constant $c$ depending on $V$ and $\fd$, such that uniformly for $0\leq s\leq t$, $\Im[z_s]\geq c\Im[z_s']$.

When simplifying the above expression, it is more important to evaluate the Green's function in constant time. We will deal with the time integral later.

We define $g_q(w)\deq (m_q(w) - \hat{m}_q(w))$.
For \eqref{e:var1}, the left hand side can be written as the derivative of the Stieltjes transform $\td m_q$ at $z_q$, and so
\begin{align}\begin{split}
\frac{\del^3_z m_q(z_q)}{6N^2}  = \frac{\del^3_z g_q(z_q)}{6N^2}+ \frac{1}{6N^2} \del^3_z \hat{m}_q(z_q)
%=&\left|\frac{1}{6N^2}\int_0^t\left(\del^3_z\left(\td m_s(z_s)-m_s(z_s)\right)+\del_z^3m_s(z_s)\right)\rd s\right|\\
=  \OO\left(\frac{M}{N^3 \eta_q^4} \right)+\frac{\del^3_z \hat{m}_q(z_q)}{6N^2} 
\end{split}\end{align}
where we used Lemma \ref{dpbound} and Lemma \ref{IntegralBounds}.

We write the LHS of \eqref{e:var2}, as a contour integral of $\hat m_s$:
\begin{align}\begin{split}\label{e:contourintm}
\frac{1}{N^3}\sum_{i=1}^N \frac{1}{(\la_i(q)-z_q)^2(\la_i(q)-z_q')^2}
%=&\frac{-2(\td m_s(z_s)-\td m_s(z_s'))}{(z_s-z_s')^3}+\frac{\del_z \td m_s(z_s)+\del_z \td m_s(z_s')}{(z_s-z_s')^2}\\
=&\frac{1}{2\pi\ri N^2}\oint_{\cal C}\frac{\hat{m}_q(w) + g_q(w)}{(w-z_q)^2(w-z_q')^2}\rd w,
\end{split}\end{align}

 In the case that $\max\{\Im[z_q]/3, \Im[z_q']/3\} \geq |z_q-z'_q|$, (without loss of generality, we assume that the maximum is $\Im[z_q]/3$) we set $\cal C$ to be a contour centered at $z_q$ with radius $\Im[z_q]/2$. In this case we have $\dist(\cal C, \{z_q, z_q'\})\geq \Im[z_q]/6$.

  In the case that $|z_q-z'_q|\geq \max\{\Im[z_q]/3, \Im[z_q']/3\}$, %and $|z_s-z'_s|\geq |z_s-z'_s|/4+\Im[z_s]/4\geq \Im[z_s']/4$.
we let $\cal C=\cal C_1\cup \cal C_2$ consist of two contours, where $\cal C_1$ is centered at $z_q$ with radius $\Im[z_q]/6$, and $\cal C_2$ is centered at $z_q'$ with radius $\Im[z_q']/6$. Then in this case we have
$\dist(\cal C_1, z_q')\geq \Im[z_q]/6$ and $\dist(\cal C_2, z_q)\geq \Im[z_q']/6$.

We analyze the contour integral of $g_q(w)$ via Taylor expansion. The computation of the associated contour integral can be done more explicitly with the square root behavior assumption.
In the first case, thanks to Lemma \ref{dpbound} and the fact that $\Im[z_q] \gg N^{-2/3}$ would imply that $\Im[w] \gg N^{-2/3}$ and we have the rigidity estimates \eqref{e:diffmm}. As a result, we have the following Taylor expansion
\beq\label{e:tdmsexpand}
 g_q(w)=g_q(z_q)+(w-z_q)\del_z  g_q(z_q)+(w-z_q)^2\OO\left(\frac{M}{N \eta_q^{3}}\right).
\eeq
Plugging \eqref{e:tdmsexpand} into \eqref{e:contourintm}, we see that the first two terms vanish and
\begin{align}
|\eqref{e:contourintm}|\leq \frac{C}{N^2}\int_{\cal C}\left(\frac{M}{N  \eta_q^{5}}\right)\rd w
=\OO\left(\frac{M}{N^3  \eta_q^{4}}\right),\label{e:contourintm1}
\end{align}
where we used that $|\cal C|\asymp \Im[z_q]$. Clearly, if instead $\Im[z_q'] >\Im[z_q]$ then the above inequality would hold except with the $\kappa_q$ and $\eta_q$ replaced with $\kappa_q'$ and $\eta_q'$. Clearly, we could use the sum of the above quantity along with its analogue with $\eta_q'$ to provide a bound in the first case.

 In the second case, \eqref{e:tdmsexpand} holds on $\cal C_1$. Similarly, for  $w\in \cal C_2$ we have
\beq\label{e:tdmsexpand2}
g_q(w)= g_q(z_q')+(w-z_q')\del_z  g_q(z_q')+(w-z_q')^2\OO\left(\frac{M}{N\eta_q'^{3}}\right).
\eeq
It follows by plugging \eqref{e:tdmsexpand} and \eqref{e:tdmsexpand2} into \eqref{e:contourintm}, that we can bound \eqref{e:contourintm} by
\begin{align}\begin{split}\label{e:contourintm2}
& \frac{C}{N^2}\left(\int_{\cal C_1}\left(\frac{M}{N (\eta_q)^{5} }\right)\rd w+\int_{\cal C_2}\left(\frac{M}{N (\eta_q')^{5}}\right)\rd w\right)\\
&= \OO\left(\frac{M}{N^3  \eta_q^{4}}\right) +\OO\left(\frac{M}{N^3  \eta_q'^{4}}\right)
\end{split}\end{align}
where we used $|\cal C_1| = \OO(\Im[z_q])$ and $ |\cal C_2|=\OO(\Im[z_q'])$.

%
%{\cor
%For \eqref{e:var2}, similarly, the righthand side can be written in terms of the Stieltjes transform $\td m_s$ at $z_s, z_s'$:
%\beq
%\frac{1}{N}\sum_{i=1}^N \frac{1}{(\la_i(s)-z_s)^2(\la_i(s)-z_s')^2}
%=\frac{-2(\td m_s(z_s)-\td m_s(z_s'))}{(z_s-z_s')^3}+\frac{\del_z \td m_s(z_s)+\del_z \td m_s(z_s')}{(z_s-z_s')^2}.
%\eeq
%Since $\td m_s$ is analytic, by \eqref{dptdms}
%\begin{align*}
%\frac{\td m_s(z_s)-\td m_s(z_s')-(z_s-z_s')\del_z \td m_s(z_s')}{(z_s-z_s')^2}=\OO\left(\frac{M}{N\Im[z_s]^{3}}+\frac{1}{t\Im[z_s]}\right)
%\end{align*}
%the same estimate holds for $\td m_s(z_s)-\td m_s(z_s')-(z_s-z_s')\del_z \td m_s(z_s)$. Notice that $|z_s-z_s'|=\Omega(\Im[z_t])$, we have
%\beq
%|\eqref{e:var2}|\leq \frac{C}{N^2\Im[z_t]}\int_0^t \left(\frac{M}{N\Im[z_s]^{3}}+\frac{1}{t\Im[z_s]}\right)\rd s
%=\OO\left(\frac{M}{N^3\Im[z_t]^{3}}+\frac{\log N}{N^2t\Im[z_t]}\right),
%\eeq
%where we used \eqref{e:propzt} in Proposition \ref{p:propzt}.
%}

Finally, for \eqref{e:var3},
\beq\label{e:mainterm}
\frac{1}{N}\sum_{i=1}^N \frac{1}{(\la_i(q)-\bar z_q)^2(\la_i(q)-z_q')^2}
=\frac{2(\overline{- m_q(z_q)}+ m_q(z_q'))}{(\bar z_q-z_q')^3}+\frac{\overline{\del_{z}  m_q( z_q)}+\del_z  m_q(z_q')}{(\bar z_q-z_q')^2}.
\eeq
Note that $|\bar z_q-z_q'|\geq \Im[z_q]+\Im[z_q']$.  As before we will separate $m_q$ into $g_q$ and $\hat{m}_q$ and analyze the corresponding term with $g_q$ for the second term in \eqref{e:mainterm}, we have by \eqref{dptdms},
\begin{align}\begin{split}
\left|\frac{1}{N^2}\frac{\overline{\del_{z} g_q( z_q)}+\del_z g_q(z_q')}{(\bar z_q-z_q')^2}\right|
\leq& \frac{C}{N^2} \left(\frac{M}{N  (\eta_q)^{4}} + \frac{M}{N (\eta_q')^{4}}\right)\\
%= &\OO\left(\frac{M}{N^3\Im[z_t]^2}\int_0^t \frac{\rd s}{\Im[z_s']^2}+\frac{1}{N^2t\Im[z_t]}\right)\\
%=&\OO\left(\frac{1}{N^3 \eta_s^{7/2}}+\frac{1}{N^3 \eta_s'^{7/2}}\right).
\end{split}\end{align}
A similar analysis can be performed for $g_q$ for the first term in \eqref{e:mainterm}.

%For the first term in \eqref{e:mainterm}, we recall the definition of the vector flow $z_s(u)$ as in \eqref{def:zt}.  Since $\|V'(z)\|_{C^1}=\OO(1)$, we have
%\beq
%\overline{-\td m_s(z_s)}+\td m_s(z_s')=\del_s (\bar z_s-z_s') +\OO(|\bar z_s-z_s'|).
%\eeq
%Therefore,
%\begin{align}
%\begin{split}
%\frac{2}{N^2}\int_{0}^{t}\frac{(\overline{-\td m_s(z_s)}+\td m_s(z_s'))}{(\bar z_s-z_s')^3}\rd s
%=&
%\frac{2}{N^2}\int_{0}^{t}\frac{\del_s (\bar z_s-z_s') }{(\bar z_s-z_s')^3}\rd s
%+\OO\left( \frac{1}{N^2}\int_0^t\frac{\rd s}{\Im[z_s]^2}\right)\\
%=&-\frac{1}{N^2(\bar z_t-z_t')^2}+\frac{1}{N^2(\bar u-u')^2}+\OO\left(\frac{1}{N^2\Im[z_t]}\right)\\
%=&-\frac{1}{N^2(\bar z_t-z_t')^2}+\OO\left(\frac{1}{N^2t^2}+\frac{1}{N^2\Im[z_t]}\right),
%\end{split}
%\end{align}
%where we used $|\bar u-u'|\geq \Im[u]+\Im[u']\geq ct$. This finishes the proof of Proposition  \ref{p:var}.
One can perform the integral of the error terms in time via Lemma \ref{IntegralBounds}

\end{proof}

\subsection{Computation of Quadratic Variance quantities associated with $\hat{m}$}

In this section, we will find the highest order expansion of various quantities that are associated with the quadratic variance terms appearing in the previous lemma.

One can easily check that 
\begin{equation} 
\oint_{\cal C}\frac{\hat{m}_q(w) }{(w-z_q)^2(w-z_q')^2}\rd w  = \int_{-\infty}^{\infty} \frac{\hat{\rho}_q(x)}{(x -z_q)^2(x-z_q')^2} \rd x
\end{equation}
In fact, if one lets $z_q = z_q'$ on the right hand size of the above equation, one obtains $6\partial_z^3 \hat{m}_q(z_q)$ while instead one replaces $z_q'$ by its complex conjugate $\bar{z_q'}$ then one would obtain the first term in the right hand side of line \eqref{e:var3}.

%In all of the lemmas that follow, we assume we are dealing with a t-parameterized family of measures that satisfy the condition of \ref{Asmp:Reg}.
\begin{lemma}

Recall \eqref{e:ConstDef}, we have the following integral evaluation.
\begin{equation}\label{e:term}
\int_{-\infty}^{\infty} \frac{\hat{\rho}_q(x)}{(x -z_q)^2(x-z_q')^2} \rd x = \pi/2 \frac{C_q}{\sqrt{\kappa_q + i \eta_q}\sqrt{\kappa_q' + i \eta_q'}(\sqrt{\kappa_q + i \eta_q}+ \sqrt{\kappa_q' + i \eta_q'})^3} (1+ O(N^{-\tilde \epsilon}))
\end{equation}
with $z_q-E_q=\kappa_q + i \eta_q$ and $z_q'-E_q= \kappa_q' + i \eta_q'$ and both $z_q,z_q' \in H_q$.
\end{lemma}

%Namely, we will try to calculate
%\begin{equation}\label{e:term}
%\int_{-\infty}^{0} \frac{\rho(x)}{(x -z_s)^2(x-z_s')^2} \rd x
%\end{equation}
%where we have $z_s= \kappa_s + i \eta_s$ and $z_s'=\kappa_s' + i \eta_s'$ both in $H^{C}_t$%with $\eta_s\gg \kappa_s$ and $\eta_s' \gg \kappa_s'$.

%To highest order, we claim that the computation can be reduced to understanding
%\begin{equation} \label{e:maincont}
%\int_{-\infty}^{0} \frac{\rho(x)}{(x- i \eta_s)^2(x - i \eta_s')^2} \rd x
%\end{equation}
%
%By scaling, one can see that the order of \eqref{e:maincont} is of the order of $\eta_s^{-7/2}$.
%
%Now one can compute the difference  between \eqref{e:maincont} and \eqref{e:term} explicitly as follows:
%\begin{equation}
%\int_{-\infty}^{0} \frac{\rho(x) (x-z_s)^2(x-z_s')^2 -(x- i\eta_s)^2(x- i \eta_s')^2}{(x-z_s)^2(x-z_s')^2(x-i \eta_s)^2(x-i \eta_s')^2} \rd x
%\end{equation}
%
%If one expands out the above quantity, the only terms that remain are those that involve an appearance of $\kappa_s$ or $\kappa_s'$
%Taking absolute value , where we can bound $|x-z_s|$ with $|x-i \eta_s|$ in the denominator, and using scaling properties, we see that the size of the above quantity is of order $O(\frac{\kappa_s}{\eta_s^{\frac{9}{2}}}$ since we assumed that $\kappa_s \ll \eta_s$, these are clearly lower order contributions and we can safely claim that $\eqref{e:maincont}$ is indeed the main contribution to the integral terms.
\begin{proof}
From this point on,when computing integrals at fixed time,  we will automatically translate $x$ so that the edge of the measure $\rho_q$ is located at 0 and the support of the measure after translation is $(-\infty,0]$.
Now, we try to explicitly identify the main term found in \eqref{e:term}, we have
\begin{equation}
\int_{-(\log N)^{-1}}^{0} \frac{\hat{\rho}_q(x)}{(x- z_q)^2(x - z_q')^2} \rd x + \int_{-\infty}^{-(\log N)^{-1}} \frac{\hat{\rho}_q(x)}{(x-  z_q)^2(x -  z_q')^2} \rd x  
\end{equation}
The second term of the above equation can clearly be bounded by $\OO((\log N)^4)$ if one takes into account the bounded support of $\hat{\rho}_q(x)$ and the fact that $z_q$ and $z_q'$ are far away from $(\log N)^{-1}$. This can easily be seen to be an $\OO(N^{-\tilde \epsilon})$ factor of the main term in $\eqref{e:term}$.
{
To estimate  the other quantity, we use a Taylor expansion of $\hat{\rho}_q(x)$ around the point 0 as $|\hat{\rho}_q(x) - C_q\sqrt{x}| = \OO(x^{\frac{3}{2}})$ . Without loss of generality we will assume that $|\kappa_q| \ge |\kappa_q'|$. To illustrate the computation, we will only consider the case that  $-\kappa_q \ge \eta_q$, and $-\kappa_q' \ge \eta_q'$. Similar techniques can be used in all other cases and are generally simpler. The above computation can be divided into two cases; the first case is where $|\kappa_q - \kappa_q'| \ge |\kappa_q|/2$. We will perform a decomposition of the integral as follows:
\begin{align}
     \int_{-\infty}^0 \frac{|x|^{3/2}}{|x-z_q|^{2}|x-z_q'|^2} \rd x &= \int_{-\infty}^{2 \kappa_q} \frac{|x|^{3/2}}{|x-z_q|^{2}|x-z_q'|^2} \rd x  + \int_{2 \kappa_q}^{\max(2 \kappa_q', (\kappa_q + \kappa_q')/2)} \frac{|x|^{3/2}}{|x-z_q|^{2}|x-z_q'|^2} \rd x  \\
     & + \int_{\max(2 \kappa_q', (\kappa_q + \kappa_q')/2)}^{0} \frac{|x|^{3/2}}{|x-z_q|^{2}|x-z_q'|^2} \rd x \\ 
     & \le 1/4 \int_{-\infty}^{2 \kappa_q} \frac{|x|^{3/2}}{|x|^4} \rd x + \frac{4|2 \kappa_q|^{5/2}}{\eta_q^2 |\kappa_q - \kappa_q'|^2} + \frac{4|2\kappa_q'|^{5/2}}{\eta_q'^2 |\kappa_q - \kappa_q'|^2}\\
     & \le \OO(|\kappa_q|^{-3/2}) + \OO(\frac{|\kappa_q|^{1/2}}{\eta_q^2}) + \OO(\frac{|\kappa_q'|^{5/2}}{\eta_q'^2 |\kappa_q|^2})
\end{align}
where the $\OO$ represents a constant factor that does not depend on $N$.

We need to show that the above quantity will be less than $N^{-\tilde\epsilon} |\kappa_q|^{-2} |\kappa_q'|^{-1/2}$. This is equivalent to the set of inequalities $|\kappa_q| \le N^{-\tilde \epsilon}$, $|\kappa_q|^3 \le \eta_q^2 N^{-\tilde \epsilon}$, $|\kappa_q'|^{3} \le \eta_q'^2 N^{-\tilde \epsilon}$. Since in $\mathcal{H}_t$, we have that $|\kappa_q| \le  \eta^{4/5 + \epsilon}$ and similar with $|\kappa_q'|$, along with the fact $\eta_q, \eta_q' \gg N^{-2/3}$, we have more than enough room to have the desired inequalities.

Now consider the case in which $|\kappa_q - \kappa_q'| \le |\kappa_q|/2$. We can divide the integral as
\begin{align}
 \int_{-\infty}^0 \frac{|x|^{3/2}}{|x-z_q|^{2}|x-z_q'|^2}  &= \int_{-\infty}^{2 \kappa_q} \frac{|x|^{3/2}}{|x-z_q|^{2}|x-z_q'|^2} \rd x  + \int_{2 \kappa_q}^{0} \frac{|x|^{3/2}}{|x-z_q|^{2}|x-z_q'|^2} \rd x\\
 \le \OO(|\kappa_q|^{-3/2}) +  \OO(\frac{|\kappa_q|^{3/2}}{\eta_q^2 \eta_q'^2})
\end{align}
To deal with the last factor, recall that in $\cH_q$ we have that $(\eta_q')^{4/5 + \epsilon} \ge (-\kappa_q') \ge (- \kappa_q)/2$. Up to a constant factor that does not depend on $N$, this would imply the second term in the previous equation would be
\begin{equation*}
    \frac{\kappa_q^{3/2 -2 \frac{1}{4/5 + \epsilon}}}{\eta_q^2}
\end{equation*}
To show that this is less than $\frac{1}{|\kappa_q|^{5/2}}N^{-\tilde \epsilon}$ is now merely a consequence of the fact that $|\kappa_q|^{4/5 + \epsilon} \le \eta_q$
}

\begin{comment}
We have,
\begin{align}
    \int_{-\infty}^0 \frac{|x|^{3/2}}{|x-z_q|^{2}|x-z_q'|^2} \rd x &\le \int_{0}^{-\infty} \frac{|x|^{3/2}}{|x-z_q|^4} \rd x + \int_{0}^{-\infty}\frac{|x|^{3/2}}{|x-z_q|^4} \rd x\\
    &\le \int_{-\infty}^{-2 \max{\kappa_q,\eta_q}} \frac{|x|^{3/2}}{|x-z_q|^4} \rd x + \int_{-2 \max(\kappa_q,\eta_q)}\frac{|x|^{3/2}}{|x-z_q|^4} \rd x + \cdots \\
    &\le \OO\left(\frac{1}{|\max(\kappa_q,\eta_q)|^{3/2}}\right) +\OO\left(\frac{|\max{\kappa_q, \eta_q}|^{5/2}}{|\eta_q|^4}\right)+\cdots
\end{align}
{\cob where the $\cdots$ represent the corresponding terms involving $\kappa_q'$ and $\eta_q'$. Notice that if we have $\eta_q \ge \kappa_q$, this this error term is clearly $\OO(N^{-\tilde \epsilon})$ times the main term in \eqref{e:term}. In the case that $\kappa_q \ge \eta_q$, we have to check the inequality $\frac{\kappa_q^{5/2}}{\eta_q^4} \le \frac{N^{-\tilde \epsilon}}{\kappa_q^{5/2}}$, which is equivalent to $\kappa_q \le \eta^{4/5} N^{-4/5 \epsilon}$. Noting that in the definition of $\cH_t$, we required there be some $\epsilon$ such that both $\kappa_q \le \eta^{4/5 + \epsilon} $ and $\eta \le N^{-\epsilon}$. Then, if we choose $\tilde \epsilon$ to be $\epsilon^2$, then we can see this error term can be incorportated as a small factor of the main term.}

\end{comment}

We now only need to compute the following integral
\begin{equation}\label{e:ttm}
\int_{-(\log N)^{-1}}^{0} \frac{C_q \sqrt{x}}{(x- z_q)^2 (x - z_q')^2} \rd x = \int_{-\infty}^{0} \frac{C_q \sqrt{|x|}}{(x - z_q)^2 (x- z_q')^2} \rd x + O((\log N)^{5/2})
\end{equation}
where we can see that the term $(\log N)^{5/2}$ can be bounded by an $N^{-\tilde \epsilon}$ factor of the main term. We can evaluate the integral in \eqref{e:ttm},

%We now have to compute the value of $\int_{-\infty}^{0} \frac{C_q \sqrt{|x|}}{(x - z_q)^2 (x- z_q')^2} \rd x$ explicitly.
%We see that when $z_q$ and $z_q'$ are positive real numbers $a$ and $b$ respectively, we have
\begin{align}\begin{split}
\int_{-\infty}^{0} \frac{C_q \sqrt{|x|}}{(x - z_q)^2 (x- z_q')^2}\rd x 
%&= \int_{0}^{\infty} \frac{2 C_q y^2}{(y^2+a)^2 (y^2 +b )^2} \rd y\\
%&= 2 C_t(\frac{-a}{(a-b)^2} \int_{0}^{\infty} \frac{1}{(y^2 + a)^2} \rd y + \frac{-b}{(b-a)^2} \int_{0}^{\infty} \frac{1}{(y^2 +b)^2} \rd y +\\
 %& \frac{a+b}{(-a+b)^3} \int_{0}^{\infty} \frac{1}{y^2 + a} \rd y +\frac{a+b}{(-b+a)^3} \frac{1}{y^2 +b} \rd y)\\
%&= 2 C_t( \frac{-a}{(a-b)^2} \frac{\sqrt{a}}{a^2} \int_{0}^{\pi/2} \cos(\theta)^2 \rd \theta + \frac{-b}{(b-a)^2} \frac{\sqrt{b}}{b^2} \int_{0}^{\pi/2} \cos(\theta)^2 \rd \theta +\\
%&\frac{a+b}{(-a+b)^3} \frac{\sqrt{a}}{a} \int_{0}^{\pi/2} 1 \rd \theta + \frac{a+b}{(-b+a)^3} \frac{\sqrt{b}}{b} \int_{0}^{\pi/2} 1 \rd \theta)\\
= (\pi/2) \frac{C_q}{\sqrt{z_q} \sqrt{z_q'} (\sqrt{z_q'} + \sqrt{z_q'})^3},
\end{split}\end{align}
where both $z_q$ and $z_q'$ are not on the negative real axis, and square root is the branch with positive real part.
%where we applied the change of variables $x\rightarrow y^2$ in the first line and evaluated the integral by partial fractions.
%By analytic continuation, we can extend the above computation to the region of $z_q, z_q'$ where neither of $z_q$ and $z_q'$ are on the negative real axis. 
\end{proof}

With this information in hand, we are able to compute the time integrals of these quantities.  However, we cannot do this yet because we do not know  exactly how the functions $\sqrt{\kappa_t + \ri \eta_t}$ behave in time. In the following lemma, we determine a relation that computes, up to highest order, the behavior of the functions $\sqrt{\kappa_t + \ri \eta_t}$ 

\begin{lemma}\label{l:sqrtgrowth}
For a point $z_t$ in $H_t$ whose characteristic $z_q$ can be written as $(\kappa_q + E_q) + i \eta_q$  we have
\begin{equation}
\left|2 \sqrt{\kappa_t + i \eta_t} - 2 \sqrt{\kappa_s + i \eta_s} + \pi \int_{s}^{t} C_q \rd q\right| = \OO\left( (\log N)^2\left((t-s)^2 + (t-s) \sqrt{|\kappa_t + i \eta_t|}\right)\right)
\end{equation}
\end{lemma}
\begin{proof}
We start by studying the differential equation determining  $\kappa_t + i \eta_t$
\begin{align}\begin{split}
\partial_t(\kappa_t + i \eta_t) &= -(m_t(z_t) - m_t(E_t)) - \frac{V'(z_t)}{2} + \frac{V'(E_t)}{2}\\
&= -\int_{-\infty}^{0} \frac{\hat{\rho}_t(x)}{x-(z_t-E_t)} \rd x + \int_{-\infty}^{0} \frac{\hat{\rho}_t(x)}{x} \rd x + O(|z_t -E_t|)\\
&= -(z_t-E_t) \int_{-\infty}^{0} \frac{\hat{\rho}_t(x)}{x(x-(z_t-E_t))} \rd x + O(|z_t -E_t|) \\
\end{split}
\end{align}

At this stage, we use our assumption on $\hat{\rho}_s$ to replace the integral appearing above with an expression that can be computed explicitly. As before, the main contribution is coming from the part close to the edge, and we can consider the integral from $-(\log N)^{-1}$ to $-\infty$ as error term. A Taylor expansion indicates that the leading order contribution will be from the $C_s \sqrt{s}$ component of the expansion of $\hat{\rho}_s(x)$ near the edge.

\begin{align} \label{e:sqrdiffeq}
\begin{split}
\partial_t(\kappa_t + i \eta_t) &= -(z_t-E_t) [ \int_{-(\log N)^{-1}}^{0} \frac{\hat{\rho}_t(x)}{x(x-(z_t-E_t))} \rd x  + \int_{-\infty}^{-(\log N)^{-1}} \frac{\hat{\rho}_t(x)}{x(x-(z_t-E_t))} \rd x ] + O(|z_t -E_t|) \\
&= -(z_t-E_t)[\int_{(\log N)^{-1}}^{0} \frac{C_t\sqrt{|x|} + O(|x|^{3/2})}{x(x-(z_t-E_t))} \rd x+ O((\log N)^2) ] + O(|z_t-E_t|)\\
&= -(z_t-E_t)[\int_{0}^{\infty} \frac{C_t \sqrt{|x|}}{x(x-(z_t -E_t))} \rd x +O((\log N)^2) ] + O(|z_t -E_t|) \\
&=- \pi C_t \sqrt{\kappa_t + i \eta_t} + O((\log N)^2|z_t -E_t|) .
\end{split}\end{align}.

{
In the second line, we bounded the integral 
$\int_{-(\log N)^{-1}}^{0} \frac{|x|^{3/2}}{|x||x- \kappa_t + i \eta_t|} \rd x$ by $(\log N)^2$. We will prove this by considering two cases: the first case is when $\kappa_t \le 0$, the second is when $\kappa_t \ge 0$. In the first case, we can split the integral as
\begin{equation}
    \int_{-(\log N)^{-1}}^{-2\max(-\kappa, \eta)} \frac{|x|^{3/2}}{|x||x- \kappa_t - i \eta_t|} \rd x +  \int_{-2\max(-\kappa, \eta)}^{0} \frac{|x|^{3/2}}{|x||x- \kappa_t - i\eta_t|} \rd x
\end{equation}
which can be bounded by $\OO(\sqrt{\max(-\kappa,\eta)}) + \OO(\frac{\max(-\kappa, \eta)^{3/2}}{\eta})$.
We can bound this quantity by $(\log N)^2$, since it is equivalent to $\sqrt{\eta} \le (\log N)^2$ and $-\kappa \le \eta^{2/3} (\log N)^2$. This is due to the fact that on $\cH_t$, we have that $-\kappa \le \eta^{4/5+ \epsilon}$ on $\cH_t$ and for the preservation property $z_t(u) \in \cH_t$ implies $z_s(u) \in \cH_s$ for $s \le t$.

In the case that $\kappa_t \ge 0$, we have the integral bound
\begin{equation*}
    \int_{-(\log N)^{-1}}^{0} \frac{|x|^{1/2}}{|x -\kappa_t|} \rd x = \sqrt{\kappa_t} \int_{-\kappa_t/(\log N)}^{0} \frac{|x|^{1/2}}{|x-1|}\rd x    
\end{equation*}    
The integral in the above equation is bounded by a constant since $\kappa_t/(\log N) $ will be less than 1. This gives us the bound of $\OO(\sqrt{\kappa_t})$ which will be less than $\OO(\log N^2).$

}

 The last line of \eqref{e:sqrdiffeq} comes from evaluating the integral at real numbers and extending by analytic continuation.
%Clearly, the term $\sqrt{\kappa_t + i \eta_t}^{-1}$ is the higher order term in the previous expression.

We have that
\begin{equation}\label{e:term1}
2\partial_q\sqrt{\kappa_q + i \eta_q}=- \pi C_q  + O\left((\log N)^2 \sqrt{|\kappa_q + \ri \eta_q|}\right) 
\end{equation}
We notice that the righthand is of order $\OO(1)$. By directly integrating both sides of the above equation, we get
\begin{align}\label{e:term2ab}
    \sqrt{|\kappa_q + \ri \eta_q|}=\OO\left(|t-q|+\sqrt{|\kappa_t+\ri\eta_t|}\right).
\end{align}
We can plug \eqref{e:term2ab} into \eqref{e:term1}, and integrate both side from $q=s$ to $q=t$, to get
\begin{align}\begin{split}
    &\phantom{{}={}}\sqrt{\kappa_t + \ri \eta_t}  - \sqrt{\kappa_s+ \ri \eta_s} + \pi/2 \int_s^t C_q \rd q \\
    &=\OO\left((\log N)^2\int_{s}^t |t-q|+\sqrt{|\kappa_t+\ri \eta_t|}\rd q\right)\\
    &=\OO\left((\log N)^2\left((t-s)\sqrt{|\kappa_t + \ri\eta_t|}+ (t-s)^2\right)\right).
\end{split}\end{align}
This finishes the proof of Lemma \ref{l:sqrtgrowth}.

\begin{comment}
One can apply Gronwall to study this lemmma. We will show here the upper bound.

One has here that
\begin{equation}
    \partial_t(\exp(-(t-s) \log(N)^2) \sqrt{\kappa_t + i \eta_t}  + \pi/2 \int_s^t C_q \exp(-(q-t)\log(N)^2) \le 0
\end{equation}
which implies that
\begin{align}
    \sqrt{\kappa_t + i \eta_t}  - \sqrt{\kappa_s+ i \eta_s} + \pi/2 \int_s^t C_q \rd q & \le [1 - \exp(-(t-s)\log(N)^2)] \sqrt{\kappa_t + i\eta_t} \\
    &+ \int_s^t C[1 - \exp(-(t-q)\log(N)^2)] \rd q\\
    & \le C(t-s) \log(N)^2 \sqrt{\kappa_t + i\eta_t} + C\log(N)^2 \int_s^{t} (t-q) \rd q\\
    &\le C(t-s) \log(N)^2 \sqrt{\kappa_t + i\eta_t}+ C\log(N)^2 (t-s)^2.
\end{align}

The lower bound can be proved in a similar way, which gives us the desired final estimate.

%\begin{equation}
%-(\pi/2 + o(1)) \frac{C_t}{\sqrt{\kappa_t + i \eta_t}}>\partial_t(\kappa_t + i \eta_t) >(\pi/2 - o(1)) \frac{C_t}{\sqrt{\kappa_t + i \eta_t}}
%\end{equation}

\end{comment}

\end{proof}

We can use the above expression when trying to compute the quadratic variance integrals
\begin{lemma} \label{Lem:VarEval}
 Let $z_t$ and $z_t'$ be two points in the region $H_t$. Let $z_s = (\kappa_s + E_s) + i \eta_s$ and $z_s'= (\kappa_s'+ E_s) + i \eta_s'$ represent the two characteristics that terminate at $z_t$ and $z_t'$ respectively at time $t$. We assume here that $\eta_t < \eta_t'$. Consider two times $s$ and $t$ with $\sqrt{\eta^*}\ll s,t \ll (\log N)^{-3}$ and $(s-t)^2 \ll \min\{\sqrt{|\kappa_t + i \eta_t|},\sqrt{|\kappa_t' + i \eta_t'|} \}$ . 
%If we additionally assume that \textbf{$|\sqrt{\kappa_0 + i \eta_0} - \sqrt{\kappa_0' + i \eta_0'}|\ll |\sqrt{\kappa_0 + i \eta_0} + \sqrt{\kappa_0' + i \eta_0'}|$ as well as
%$|\sqrt{\kappa_0 + i \eta_0} + \sqrt{\kappa_0' + i \eta_0'}|\gg |\sqrt{\kappa_t + i \eta_t} + \sqrt{\kappa_t' + i \eta_t'}|$}

We have  that 
\begin{align}\begin{split}\label{e:integralestimate}
&\phantom{{}={}}\int_{s}^{t}\frac{\sqrt{\kappa_q + i \eta_q}\sqrt{\kappa_q' + i \eta_q'}}{\sqrt{\kappa_t + i \eta_t}\sqrt{\kappa_t' + i \eta_t'}}\int_{-\infty}^{\infty} \frac{\hat{\rho}_q(x)}{(x -z_q)^2(x-z_q')^2} \rd x \rd q \\
&= 
\frac{1}{4 \sqrt{\kappa_t + i \eta_t}\sqrt{\kappa_t' + i \eta_t'} (\sqrt{\kappa_t + i \eta_t} + \sqrt{\kappa_t' + i \eta_t'})^2}[1+ \OO(N^{-\tilde \epsilon})] \\
&- \frac{1}{4 \sqrt{\kappa_t + i \eta_t}\sqrt{\kappa_t' + i \eta_t'} (\sqrt{\kappa_s + i \eta_s} + \sqrt{\kappa_s' + i \eta_s'})^2}[1+ \OO(N^{-\tilde \epsilon})]
\end{split}\end{align}
\end{lemma}

\begin{proof}

We want to compute
\begin{align}
\int_{s}^{t} \pi/2 \frac{C_q}{(\sqrt{(\kappa_q + i \eta_q)} + \sqrt{(\kappa_q' + i \eta_q')})^3} \rd q
\end{align}
which up to a factor of $(1+ \OO(N^{-\tilde \epsilon}))$ is letting $A= \sqrt{\kappa_s + \ri \eta_s}$ , $A' = \sqrt{\kappa_s' + \ri \eta_s'}$ and $B_t =  \pi/2 \int_{s}^{t} C_s \rd s$

\begin{align}
\int_{s}^{t} \pi/2 \frac{C_q}{(A+A'- 2 B_q)^3} \rd q   
\end{align}
We change variable from $q \rightarrow B_q$ the Jacobian of this transform is clearly $\pi/2 C_q$
\begin{equation}
\int_{B_s=0}^{B_t} \frac{1}{(A+A' -2 x)^3} \rd x 
\end{equation}
The above equation is up to a factor of $O(1+ N^{-\tilde \epsilon})$ is
\begin{equation}
    \frac{1}{4(\sqrt{\kappa_t + i \eta_t} + \sqrt{\kappa_t' + i \eta_t'})^2}
\end{equation}
\end{proof}
We remark here that if we shoose $t-s\gg \max \{\sqrt{|\kappa_t + \eta_t|}, \sqrt{|\kappa_t' + i \eta_t'|}\}$, then the last line of \eqref{e:integralestimate} can be subsumed as a $O(N^{- \tilde \epsilon})$ error of the second term.

\begin{proof}[Proof of Lemma 5.7]
 Through calculations similar to what we have done earlier, we can compute the value of $\mathcal{I}_t (\mathcal{I}_s)^{-1}$ when we have $(s-t)^2(\log N)^2 \ll \sqrt{\kappa_t + i \eta_t}^{1+\tilde \epsilon}$. 
\begin{align}\begin{split}
    \int_{s}^t \partial_z \hat{m}_s(z_s) \rd s &= \int_s^t\int_{0}^{\infty} \frac{C_q \sqrt{x}}{(x+ \kappa_q + i \eta_q)^2} \rd x+\OO((\log N)^2)  \rd t \\
    &= \int_{s}^{t} \frac{C_q \pi}{2 \sqrt{\kappa_q + i \eta_q}} \rd q +\OO((t-s) \log N^{1/2})\\
    &= \ln\left(\frac{\sqrt{\kappa_s + i \eta_s}}{\sqrt{\kappa_t + i \eta_t}} \right) +\OO(N^{-\tilde \epsilon})
\end{split}\end{align}
{ To get from the first line to the second line, we bounded the integral of $\int_{-(\log N)^{-1}}^{0}  \frac{x^{3/2}}{|x - \eta_q + i \eta_q|^2} \rd x$ by $\OO((\log N)^2)$. As before, the only interesting case is when $-\kappa_q \ge \eta_q$. The integral can be divided and bounded up to a constant independent of $N$ as
\begin{equation}
    \int_{-(\log N)^{-1}}^{2 \kappa_q} |x|^{-1/2} \rd x + \int_{2 \kappa_q}^{0} \frac{|x|^{3/2}}{|x-\kappa_q -i \eta_q|^2} \rd x \le (|\kappa_q|)^{1/2} + \frac{|\kappa_q|^{5/2}}{\eta_q^2}
\end{equation}
The bound $\frac{|\kappa_q|^{5/2}}{\eta_q^2} \le (\log N)^2$ is true because of the fact that terms in $\cH_t$ satisfy $|\kappa_q| \le \eta_q^{4/5 + \epsilon}$, the other inequality is true since $|\kappa_q| \ll (\log N)^2$.

The final line requires the following comparison estimate
\begin{align}
    \int_{s}^{t} \frac{C_q \pi/2}{\sqrt{\kappa_q + i \eta_q}} \rd q  - \int_{s}^{t} \frac{C_q \pi/2}{ (\sqrt{\kappa_s + i \eta_s} -\pi/2 \int_s^{q} C_l \rd l)} \rd q \le \int_{s}^{t} \frac{C_q \pi/2}{ (\sqrt{\kappa_s + i \eta_s} -\pi/2 \int_s^{q} C_l \rd l)^{1- \tilde \epsilon}} \rd q
\end{align}
where we used in the last line that for times $s<t$ as in the condition of Lemma \ref{lem:IntFac}, the difference in the denominator of the first integral and the second integral in the left hand side of the above equation is $(\sqrt{\kappa_q + i \eta_q})^{1+ \tilde \epsilon}$. This gives the $1- \epsilon$ in the denominator on the left hand side. We can integrate the right hand side to get that this is bounded by $O(|\sqrt{\kappa_s + i\eta_s}|^{\tilde \epsilon})$. Clearly, this can be bounded by $O(N^{-\tilde \epsilon})$ . This shows us that we have an additive error term of the order $O(N^{-\epsilon})$. Taking exponentials will give us a multiplicative error term as in \ref{lem:IntFac}}

\end{proof}

\section{Edge Universality}

We recall the $\beta$-Dyson Brownian motion from \eqref{DBM}
\beq 
\rd \la_i(t) = \sqrt{\frac{2}{\beta N}} \rd B_i(t) +\frac{1}{N}\sum_{j:j\neq i}\frac{\rd t}{\lambda_i(t)-\lambda_j(t)}-\frac{1}{2}V'(\lambda_i(t))\rd t,\quad i=1,2,\cdots, N,
\eeq
where the initial data $\bmla(0)=(\la_1,\la_2,\cdots,\la_N)$ satisfies assumption \ref{a:initial}, and the potential $V$ satisfies \ref{a:asumpV}.
\begin{comment}
We assume that the standard of comparison $\hat{m}$ satisfies the assumption \ref{Asmp:Reg} and, thus, we can write the measure associated to $\hat{m}_t$ as $\hat{\rho_t}(x) = C_t\sqrt{E_t-x} + O((E_t-x)^{3/2})$, where $\hat{m}_t$ is the solution to the McKean-Vlasov equation of potential $V$ with $\hat{m}_0$ as initial data.
\end{comment}

We have shown in previous sections that after applying the $\beta-$DBM with potential $V$ on $\bmla(0)$ would create a distribution that satisfies edge rigidity at an optimal scale after certain amount of time. Upon further applying the $\beta-$DBM with potential $V$ , we can compare the local eigenvalue fluctuations to that of the $\beta$-ensemble with a quadratic potential.  The main strategy is to perform a coupling of the $\beta$-DBM process with that of the $\beta$-DBM process with quadratic potential from its equilibrium measure. We will, as before, estimate the differences in the coupling via a continuous interpolation. A similar analysis has been performed in \cite{Landon2016}

 Without loss of generality, in the rest of this section we assume that the initial data $\bmla(0)$ satisfies the optimal rigidity. Otherwise, we can first apply the $\beta$-DBM  until we have the edge rigidity at an optimal scale.

We now define $\mu_i$ as the unique strong solution to the SDE,
\beq
\rd \mu_i(t) = \sqrt{\frac{2}{\beta N}} \rd B_i(t) +\frac{1}{N}\sum_{j:j\neq i}\frac{\rd t}{\mu_i(t)-\mu_j(t)}-\frac{1}{2}W'(\mu_i(t))\rd t,\quad i=1,2,\cdots, N,
\eeq
with initial data $\mu_i (0)$ being distributed like a $\beta$-ensemble 
\begin{align}\label{e:beta2}
    \frac{1}{Z_N}\prod_{i<j}|\mu_i-\mu_j|^\beta e^{-N\sum_{i}W(\mu_i)}\prod_{i}\rd \mu_i,
\end{align}
where the potential $W$ is quadratic, and the equilibrium measure behaves like
\begin{align}
    C_0 \sqrt{E_0-x},
\end{align}
as $x\rightarrow E_0$.

The main result of this section is the following.
\begin{theorem} \label{thm:maindbm}
 Let $t_1= \OO(\frac{N^{\omega_1}}{N^{1/3}})$ .  With overwhelming probability, we have
\beq
| (\lambda_{ i} (t_1) - E_{t_1} ) - ( \mu_i - E_0 ) | \leq \frac{1}{N^{2/3+\varepsilon} }
\eeq
for a small $\varepsilon>0$ and for any finite $1 \leq i \leq K$.
\end{theorem}

\subsection{Interpolation}
For clarity of presentation, we will write up a proof of \ref{thm:maindbm} in the case that we have a local law and rigidity along the entire spectrum. In the case that there is less control of the egienvalues above some scale $i^* \asymp N$, one can perform minor modifications to the construction of the interpolating process above the scale $i^*$ like in \cite{LandonEdge} section 3.1.

We define the following interpolating processes for $0 \leq \alpha \leq 1$.
\beq \label{e:defz}
\rd z_i (t,  \alpha ) = \sqrt{\frac{2}{\beta N}} \rd B_i(t) +\frac{1}{N}\sum_{j:j\neq i}\frac{\rd t}{z_i (t,  \alpha )-z_j (t,  \alpha )}-\frac{1}{2}V_\alpha'(z_i (t,  \alpha ))\rd t,\quad i=1,2,\cdots, N,
\eeq
with the potential
\begin{align}
    V_\alpha=\alpha V +(1-\alpha)W,
\end{align}
and the initial data
\beq
z_i (0, \alpha ):= \alpha \lambda_i (0) + (1 - \alpha ) \mu_i (0),
\eeq
for $i=1,2,\cdots,N$.

We define the Stieltjes transform of the empirical particle process $z_i(t,\alpha)$ as defined in \eqref{e:defz}
\beq
m_t (z, \alpha ) = \frac{1}{N} \sum_i \frac{1}{ z_i (t, \alpha ) - z }.
\eeq

We recall that $\hat{\rho}_{ t}$ is the solution of the McKean-Vlasov equation with initial data given by the Stieltjes transform $\hat m_0(z)$.  The edge of $\hat \rho_{t}$ will be designated by $E_t$. % Denote $\hat \rho_{t}$ to be the density of the $\beta$-ensemble  after running the DBM for time t with edge designated by $ E_{\mu} (t)$.

We recall that by our Assumption \ref{a:initial}, we have 
\begin{align}
    \hat\rho_0(x)=f_0(x,1)\sqrt{E_0-x}\eqd\hat \rho_0(x,1),
\end{align}
where $f_0(x,1)$ is analytic in a neighborhood of $x=E_0$. The equilibrium measure of the $\beta$-ensemble \eqref{e:beta2}, has the form
\begin{align}
    \hat\rho_0(x,0)=f_0(x,0)\sqrt{E_0-x}.
\end{align}
It turns out the empirical distribution of the interpolated initial data $z_i(z,\alpha)$ is close to the profile,
\begin{align}
    \hat\rho_0(x,\alpha)=f_0(x,\alpha)\sqrt{E_0-x}.
\end{align}
Let 
\begin{align}
    \hat F(y,\alpha)=\int_y^\infty \hat\rho_0(x,\alpha)\rd x,
\end{align}
the profiles $\hat \rho_0(x,\alpha)$ for $0<\alpha<1$ are determined by
\begin{align}
    F^{-1}(y,\alpha)=\alpha F^{-1}(y,1)+(1-\alpha)F^{-1}(y,0).
\end{align}

\begin{proposition}
{Under the assumptions of \ref{a:initial} on initial data to optimal scale $\eta^* = N^{-1/3}$, we have
\begin{align}
    \frac{1}{N}\sum_{i=1}^N \delta_{z_i(0,\alpha)}\sim \hat \rho_0(x,\alpha)\sim C_0\sqrt{E_0-x}.
\end{align}
where the $\sim$ relation holds in the following sense:
\begin{equation}
    |\lambda_i(0,\alpha) - \gamma_i(0,\alpha)| < \frac{N^{\epsilon}}{N^{2/3} i^{1/3}}
\end{equation}
and $\gamma_i(0,\alpha)$ is the value of $\hat F^{-1}(i/n, \alpha)$.
This rigidity of the eigenvalues would imply that the point process $\frac{1}{N}\sum_{i=1}^N \delta_{\lambda_i(0,\alpha)} $ would be close to the stable measure $\hat{\rho}_0(x,\alpha)$ in the sense of Assumption \ref{a:initial} to optimal scale.
}

\end{proposition}

One should note that the above lemma is quite simply a consequence of the definition of the $F$ inverse transform if deciphered correctly. The value of $\gamma_i(0,\alpha)$ should be the linearly interpolated value between the endpoints at $\alpha=0$ and 1. Since this optimal rigidity holds at these regions, we have it at $\alpha$.

{ \begin{proposition} \label{Prop:MeasureInterp}
Consider two measures $\rho_0(x)= f_0(x) \sqrt{x-E_0}$ and $\rho_1(x) =f_1(x)\sqrt{x-E_1}$ where $f_0(x)$ and $f_1(x)$ are analytic functions around their respective edges $E_0$ and $E_1$. The measure $\rho_{\alpha}$ whose eigenvalue counting function $F(y,\alpha)$ is determined by the relation 
\begin{equation*}
    F^{-1}(y,\alpha) = \alpha F^{-1}(y,1) + (1-\alpha)F^{-1}(y,0)
\end{equation*}
where $F(y,1)$ is the eigenvalue counting function of $\rho_1(x)$ and $F(y,0)$ is the eigenvalue counting function of $F(y,0)$ is of the following form:

\begin{equation*}
    \rho_{\alpha}(x) = f_{\alpha}(x) \sqrt{E_{\alpha} -x }
\end{equation*}
where $f_{\alpha}(x)$ is analytic around the new edge $E_{\alpha}= \alpha E_0 + (1-\alpha) E_1$
\end{proposition}}
\begin{proof}
This statement can be proved by expanding power series.
First expand $f_0= a_0 + a_1(E_0-z) + a_2(E_0-z)^2\cdots$ and $f_1= b_0 + b_1(z-E_1) + b_2(z-E_1)^2+\cdots$

Then we can write the function $\hat{F}(y,0)= \int_{y}^{E_0} [a_0 (E_0 -z_0)^{1/2} + a_1(E_0 -z_0)^{3/2} + a_2 (E_0 - z_0)^{5/2} +\cdots] \rd y$ so , by integrating, we get
\begin{equation}
    \hat{F}(z_0,0)= -2/3 a_0 (E_0 -z_0)^{3/2} -2/5 a_1 (E_0-z_0)^{5/2} -2/7 a_2 (E_0 - z_0)^{7/2} -\cdots
\end{equation}
Taking the inverse of this expression will give us fractional powers with lowest power $y^{2/3}$. Since we are only considering $y$ positive, there is no issue in defining our expansions.

\begin{equation}
    \hat{F}^{-1}(y,0) = E_0 + \hat{a}_0 y^{2/3} + \hat{a}_1 y^{4/3} + \hat{a_2} y^{6/3} 
\end{equation}
Upon taking the map $(E_0 - z_0) \rightarrow (E_0 - z_0)^2$, we see that we are actually inverting a formal power series. We can bound the coefficients exponentially by using Lagrange's inversion formula. After this procedure, we see that if we actually replace  
$y \rightarrow y^3$ in the above formula, we see that we get an analytic power series in a small neighborhood of 0.

One can get a similar power series expansion to $\hat{F}^{-1}(y,1) = E_1 + \hat{b}_0 y^{2/3} + \hat{b}_1 y^{4/3}+\cdots$, we then obtain
\begin{equation}
    \hat{F}^{-1}(y, \alpha) = \alpha E_0 + (1- \alpha) E_1 +(\alpha \hat{a}_0 + (1- \alpha) \hat{b}_0) y^{2/3} + (\alpha \hat{a}_1 + (1- \alpha) \hat{b}_1) y^{4/3} + (\alpha \hat{a}_2 + ( 1-\alpha) \hat{b}_2) y^{6/3}+\cdots
\end{equation}
We have to invert the function $\hat{F}^{-1}(y,\alpha)$ and, again, the inverse can be written in the form
\begin{equation}
    \hat{F}(z, \alpha) = c_0(z - \alpha E_0 -(1-\alpha) E_1)^{3/2} + c_1(z -\alpha E_0 -(1- \alpha) E_1)^{5/2}+\cdots.  
\end{equation}
Again, the coefficients appearing in the above expression would be bounded exponentially in a small interval by Lagrange's inversion formula, and the above expression, if we factor out $(z - \alpha E_0 -(1- \alpha) E_1)$ would represent an analytic power series in a small neighborhood around $\alpha E_0 + (1-\alpha) E_1$ .

We can take the derivative of this to find the functional form of the measure and, thus, we see that the functional form of the measure $\rho_{\alpha}(x) = f_{\alpha}(x) \sqrt{x - E_{\alpha}}$  where 
\end{proof}

We let now $\hat\rho_t (x, \alpha )$ be the solution of the McKean-Vlasov equation of potential $V_\alpha$ with initial data $\hat\rho_0(x,\alpha)$, and denote the Stieltjes transform by $\hat m_t (z, \alpha)$. It follows from Proposition \ref{Prop:MeasureInterp} that they have a square root density with an right edge which we denote by $E_t(\alpha)$. Let $\gamma_i (t, \alpha )$ be the classical eigenvalue locations with respect to $\hat\rho_t (E, \alpha)$.   To be more precise, they are defined by
\beq
\frac{i}{N} = \int_{ \gamma_i (t, \alpha ) }^\infty \hat\rho_t (x, \alpha ) \rd x.
\eeq
 As a consequence of Proposition \ref{prop:TimStab}, we have the following proposition for the solutions of McKean-Vlasov equation in time for the interpolated measures.

\begin{proposition}
{Under the assumption \ref{a:initial} on initial data to optimal scale $\eta* = N^{-1/3}$, we have
\begin{align}
    \sum_{i=1}^N \delta_{z_i(t,\alpha)}\sim \hat \rho_t(x,\alpha)\sim C_t(\alpha)\sqrt{E_t(\alpha)-x}.
\end{align}
In the following sense:
there exists a small constant $\fe>0$ so that the following estimates hold.  We have,
\beq \label{eqn:rigzi}
\sup_{ 0 \leq \alpha \leq 1 } \sup_{ 0 \leq t \leq T } |z_i (t, \alpha ) - \gamma_i (t, \alpha ) | \leq \frac{ M }{N^{2/3} i^{1/3}}
\eeq
for $1 \leq i \leq \fe N$ with overwhelming probability.
along with a corresponding local law.

In addition, we have the following estimates regarding the change of the measure and the edge in $\alpha$ and time $t$
\begin{align}
    |C_t(\alpha)-C_t(0)|=\OO(t),\\
    |E_t(\alpha)-E_t(0)|=\OO(t).
\end{align}

As a consequence of the scaling estimates, we have the rigidity results:
for $cN^{-2\epsilon} N^{-2/3} \le E \le 0$
\begin{equation} \label{e:interrig}
    |\Re[\hat{m}_t(E+ E_t(\alpha),\alpha)- \hat{m}_t(E_t(\alpha), \alpha)] - \Re[\hat{m}_t(E+ E_t(0),0)- \hat{m}_t(E_t(0), 0)]| \le C \frac{|E|N^{\epsilon}}{N^{-1/3}} 
\end{equation}

for $0 \le E \le cN^{-2 \epsilon} N^{-2/3}$
\begin{equation}
     |\Re[\hat{m}_t(E+ E_t(\alpha),\alpha)- \hat{m}_t(E_t(\alpha), \alpha)] - \Re[\hat{m}_t(E+ E_t(0),0)- \hat{m}_t(E_t(0), 0)]| \le C |E|^{1/2}N^{\epsilon} 
\end{equation}

For eigenvalues $i$ that are on the scale $N^{\omega_A} \ll N$, we would have the following estimates on the classical locations on the eigenvalues.
\begin{equation} \label{e:classloc}
    |(\gamma_i(t,\alpha) - E_t(\alpha)) - (\gamma_i(t,0) - E_t(0))| \le N^{\epsilon}\frac{i^{2/3}}{N^{2/3}} 
\end{equation}

}

\end{proposition}

The proof of the rigidity results are the same as in Lemma 7.11 and 7.12 of \cite{LandonEdge}. It only involves knowing that the two measures are close near the edge up to a small multiplicative factor and square root behavior around the edge.  As a remark, the corresponding result in \cite{LandonEdge} gives the bound on the right hand side of \eqref{e:classloc} as a function of $t$ and for a larger range of eigenvalues. However, if one investigates the proof of the coming Proposition \ref{prop:shortrangapprox} where these estimates are used, only the weaker version presented here is ever used.

\subsection{Short-range approximation}

We recall that the edge $E_t(\alpha)$ satisfies the differential equation:
\begin{align}
    \rd E_t(\alpha)=- m_t (E_t (\alpha ), \alpha )  \rd t- \frac{1}{2} V'(E_t(\alpha)) \rd t
\end{align}
Combining the SDE of $z_i (t, \alpha)$, we get
\begin{align}\begin{split}
&\rd \left(z_i (t, \alpha ) -E_t(\alpha)\right) = \sqrt{\frac{ 2}{ \beta N}} \rd B_i(t) + \frac{1}{N} \sum_{j:j\neq i} \frac{1}{ z_i (t, \alpha ) - z_j (t, \alpha ) } \rd t\\
&- \frac{1}{2} V'(z_i(t,\alpha)) \rd t
  +  m_t (E_t (\alpha ), \alpha )  \rd t + \frac{1}{2} V'(E_t(\alpha)) \rd t.
\end{split}\end{align}

{
The important effects of the edge behavior are due to short range interactions. To quantify this information, we also use the set of indices $\cal A \subseteq \qq{N} \times \qq{N}$.  We choose $\cal A$ to be symmetric, i.e., $(i, j) \in \cal A$ iff $(j, i) \in \cal A$.  The definition of $\cal A$ requires the choice of
\beq
\ell := N^{\om_\ell}.
\eeq
We let
\beq
\cal A := \left\{ (i, j) : |i-j| \leq \ell ( 10 \ell^2 + i^{2/3} + j^{2/3} ) \right\}.
\eeq

}

We denote the interval 
\begin{align*}
    I_i(t, \alpha)=[\gamma_{j-}(t,\alpha)-E_t(\alpha), \gamma_{j+}(t,\alpha)-E_t(\alpha)],
\end{align*}
where $j-=\min_j\{(i,j)\in \cA\}$ and $j+=\max_j\{(i,j)\in \cA\}$.

We introduce the short range approximation of $z_i(t,\alpha)$ by $\tilde z_i(t,\alpha)$, such that the difference $z_i(t,\alpha)-\tilde z_i(t,\alpha)$ is negligible for small indices $i$. The advantage for the new dynamics $\tilde z_i(t,\alpha)$ is that the derivative $\del_\alpha(\tilde z_i(t,\alpha)-E_t(\alpha))$ does not depend on particles far away for small indices $i$. The second benefit is the fact that there are fewer $\alpha$ dependencies near the edge, which will make the later analysis simpler.

For $1 \leq i \leq N^{\om_A}$,
\begin{align}\begin{split}\label{e:short}
&\rd (\tilz_i (t, \alpha )-E_t(\alpha)) = \sqrt{\frac{2}{\beta N}} \rd B_i(t) + \frac{1}{N}
\sum_{j:(i,j)\in \cal A} \frac{\rd t}{ \tilz_i (t, \alpha ) - \tilz_j (t, \alpha ) }
\\
&+ \int_{ I^c_i (0, t) } \frac{\hat \rho_t (E+E_t(0),0)}{ \tilz_i (t, \alpha ) -E_t(\alpha)- E }  \rd E \rd t
   + \Re[ m_t (E_t (0), 0) ]\rd t,
\end{split}\end{align}

for $N^{\om_A} \leq i $,
\begin{align}\begin{split}\label{e:long}
&\rd (\tilz_i (t, \alpha )-E_t(\alpha)) = \sqrt{\frac{2}{\beta N}} \rd B_i(t)  + \frac{1}{N} \sum_{j:(i,j)\in A} \frac{\rd t}{ \tilz_i (t, \alpha ) - \tilz_j (t, \alpha ) }\\
&+  \int_{ I^c_i (\alpha, t)  } \frac{\hat \rho_t (E+E_t(\alpha) , \alpha )}{ \tilz_i (t, \alpha )-E_t(\alpha) - E }  \rd E \rd t 
  -\frac{1}{2} V_\alpha'(\tilz(t,\alpha))\rd t \\
  &+ \Re[ m_t (E_t (\alpha), \alpha) ] \rd t + \frac{1}{2}V_\alpha'(E_{t}(\alpha))\rd t.
\end{split}\end{align}

{ One should notice that for particles near the edge, we have largely removed the dependence on $\alpha$. What one should realize is that the effects of the interpolation are very small near the edge, so one can approximate replacing terms like $\Re[m_t(E_t(\alpha),\alpha)$ by its counterpart at $\alpha =0$. Similarly, the effect of $V_{\alpha}'(E_t(\alpha)) - V_{\alpha}(\hat{z}(t,\alpha))$ is negligible near the edge. Since these error terms are small, and our differential kernel is a contraction in $\ell^p$ space, we are able to show the difference upon making the short range approximation is small. The details of the proof are similar to those found in \cite{LandonEdge} Lemma 3.7; we have the same rigidity estimates \eqref{eqn:rigzi} as well as the measure comparison estimates \eqref{e:interrig}- \eqref{e:classloc}, which allow us to show that the error made upon replacing the interpolating terms at $\alpha$ with terms at $0$ are negligible at scales $i \le N^{\omega_A}$.

\begin{proposition} \label{prop:shortrangapprox}
With the construction \eqref{e:short} and \eqref{e:long}, we have
\begin{align}
    \sup_{0\leq \alpha\leq 1}\sup_{0\leq t\leq T}\max_{1\leq i\leq N}|z_i(t,\alpha)-\tilde z_i(t,\alpha)|\leq \frac{1}{N^{2/3-\fc}},
    \end{align}
and especially,
\begin{align}
    \sup_{0\leq \alpha\leq 1}\sup_{0\leq t\leq T}\sup_{1\leq i\leq K}
    |z_i(t,\alpha)-\hat \gamma(t,\alpha)|\leq \frac{M}{N^{2/3}i^{1/3}}.
\end{align}
\end{proposition}

In the following we show that  for time $t\gg N^{1/3}$, $\max_i |\del_\alpha\tilde z_i(t,\alpha)|$ is negligible. 
Let $u_i (t, \alpha ) \deq \del_\alpha (\tilde z_i (t, \alpha )-E_t(\alpha))$.  We see that $u(t,\alpha)=(u_1(t,\alpha), u_2(t,\alpha),\cdots, u_N(t,\alpha))$ satisfies the equation,
\beq
\del_t u(t,\alpha) = \cal L u(t,\alpha) + \cal E,
\eeq
where the operator $\cal L$ and the force term $\cal E$ are given as follows. The operator $\cal L$ is 
\beq
\cal L = \cal B + \cal V,
\eeq
where
\beq
( \cal B u )_i = \frac{1}{N} \sum_{j: (i,j)\in \cal A} \frac{ u_j - u_i}{ ( \tilde z_i ( \alpha, t) - \tilde z_j ( \alpha, t) )^2},
\eeq
where for $1 \leq i \leq N^{\om_A}$,
\beq
\cal V_i = - \int_{ I_i (0, t ) } \frac{\rho_t (E, 0) }{ ( \tilde z_i ( \alpha, t) - E )^2 },\quad \cal E=0
\eeq
for $ N^{\om_A} \leq i$,
\beq
\cal V_i = - \int_{ I_i (\alpha, t ) } \frac{\rho_t (E, \alpha ) }{ ( \tilde z_i (\alpha, t) - E )^2 },\quad |\cal E|\leq N^C.
\eeq

The same as in \cite{LandonEdge}, the propagator of the operator $\cal L$ satisfies a finite speed estimate. It follows from an energy estimate the same as in \cite{LandonEdge}, we get
\begin{align}\begin{split}
    &\phantom{{}={}}N^{2/3}|(z_i(t,1)-E_t(1))-(z_i(t,0)-E_t(0))|\\
    &=N^{2/3}|(\tilde z_i(t,1)-E_t(1))-(\tilde z_i(t,0)-E_t(0))|+\OO\left(N^{-\fc}\right)\\
    &=N^{2/3}\left|\int_0^1\del_\alpha(\tilde z_i(t,\alpha)-E_t(\alpha))\right|+\OO\left(N^{-\fc}\right)\\
    &=\OO\left(N^{\varepsilon}/(N^{1/3}t)+N^{-\fc}\right)=\oo(1),
\end{split}\end{align}
provided $t\gg N^{-1/3}$, and Theorem \ref{thm:maindbm} follows.

\appendix
\section{Proof of Proposition \ref{prop:TimStab}}
\label{a:TimStab}

The proof of Theorem \ref{prop:TimStab} is based on performing a power series expansion of the analytic functions $A_0$ and $B_0$ in a neighborhood around $E_0$ and solving the McKean-Vlasov equation term by term. 
 $A_0(z),B_0(z)$ have power series representations
\begin{align}
A_0(z)&=a_0+a_1(z-E_0)+a_2(z-E_0)^2+a_3(z-E_0)^3+\cdots,\\
B_0(z)&=b_0+b_1(z-E_0)+b_2(z-E_0)^2+b_3(z-E_0)^3+\cdots,
\end{align}
such that
\begin{align}
|a_i|, |b_i|\leq \frac{C_0M^{i}}{(i+1)^2},\quad i=0,1,2,\cdots.
\end{align}

The following proposition states that the Stieltjes transform of $\hat\mu_t$ has the form $A_t+\sqrt{B_t}$, and $A_t, B_t$ have power series representation in a neighborhood of $E_0$. The Proposition \ref{prop:TimStab} is a natural consequence.
\begin{proposition}\label{p:AB}
We assume Assumption \eqref{e:holom0}. We fix small $T>0$ and $L= 1/T$. Then for $t\in [0,T]$, the solution of \eqref{eq:dm0} is given by
\begin{align}\label{e:mtAB}
\hat m_t(z)=A_t(z)+\sqrt{B_t(z)},
\end{align}
and in a small neighborhood of $E_0$, $A_t(z)$ and $B_t(z)$ have power series representations,
\begin{align}\label{e:ABpower}
\begin{split}
A_t(z)&=a_0(t)+a_1(t)(z-E_0)+a_2(t)(z-E_0)^2+a_3(t)(z-E_0)^3+\cdots,\\
B_t(z)&=b_0(t)+b_1(t)(z-E_0)+b_2(t)(z-E_0)^2+b_3(t)(z-E_0)^3+\cdots,
\end{split}
\end{align}
where the coefficients satisfy
\begin{align}\label{e:coebound}
|a_i(t)|, |b_i(t)|\leq \frac{CM^ie^{Lti}}{(i+1)^2}.
\end{align}
Moreover, in a small neighborhood of $E_0$, $B_t(z)$ has a unique simple root at $z=E_t$,
\begin{align}\label{e:AE}
\del_t E_t=-\hat m_t(E_t)-\frac{V'(E_t)}{2}.
\end{align}
\end{proposition}

\begin{proof}
We make the ansatz, such that the solution of \eqref{eq:dm0} is given by \eqref{e:mtAB}, and  $A_t(z)$, $B_t(z)$ have power series representations given by \eqref{e:ABpower}. We plug \eqref{e:mtAB} into \eqref{eq:dm0},
\begin{align}\begin{split}\label{e:ABeq}
\del_t A_t +\frac{\del_t B_t}{2\sqrt{B_t}}
&=D_t\del_zD_t+\frac{\del_z B_t}{2}+R_t(z)+\frac{1}{\sqrt{B_t}}\left(B_t\del_zD_t+\frac{1}{2}D_t\del_z B_t\right),
\end{split}\end{align}
where
\begin{align} \label{e:seriesExpansion}
D_t(z)=A_t+\frac{V'(z)}{2},\quad R_t(z)=\int_\bR g(z,x)\rd \hat{\mu}_t(x)-\frac{V'(z)V''(z)}{4},
\end{align}
are analytic in a neighborhood of $E_0$. One should note here that our goal in this section is not to show existence of solution of the original McKean-Vlasov equation, but only the existence of the analytic extension. Thus, the fact that we use $\hat{\mu}_t$ in our expression of $R_t(z)$ is not an issue.

For \eqref{e:ABeq} to hold, it is sufficient that
\begin{align}\begin{split}\label{e:BDReq}
\del_tD_t&=D_t\del_zD_t+\frac{\del_z B_t}{2}+R_t(z),\\
\frac{\del_t B_t}{2}&=B_t\del_zD_t+\frac{1}{2}D_t\del_z B_t.
\end{split}\end{align}
We solve \eqref{e:BDReq} using the power series representations. Let
\begin{align}\begin{split}
B_t(z)&=b_0(t)+b_1(t)(z-E_0)+b_2(t)(z-E_0)^2+b_3(t)(z-E_0)^3+\cdots\\
D_t(z)&=d_0(t)+d_1(t)(z-E_0)+d_2(t)(z-E_0)^2+d_3(t)(z-E_0)^3+\cdots\\
R_t(z)&=r_0(t)+r_1(t)(z-E_0)+r_2(t)(z-E_0)^2+r_3(t)(z-E_0)^3+\cdots,
\end{split}\end{align}
then \eqref{e:BDReq} is equivalent to the infinite system of ordinary differential equations for $i=0,1,2,\cdots$,
\begin{align}\begin{split}\label{e:ode}
\del_t d_i(t)&=\sum_{j=0}^i(j+1)d_{i-j}(t)d_{j+1}(t)+\frac{1}{2}(i+1)b_{i+1}(t)+r_i(t)\\
\del_t b_i(t)&=\sum_{j=0}^i(j+1)b_{i-j}(t)d_{j+1}(t)+\frac{1}{2}\sum_{j=0}^i(j+1)d_{i-j}(t)b_{j+1}(t).
\end{split}\end{align}
Although \eqref{e:ode} is not Lipschitz, we can still solve them by the Picard iteration. Let
\begin{align*}
d_i^{(0)}(t)=d_i(0), \quad b_i^{(0)}(t)=b_i(0),
\end{align*}
and recursively we define
\begin{align}\begin{split}\label{e:recur}
d_i^{(n+1)}(t)
&=d_i(0)+\int_0^t\left(\sum_{j=0}^i(j+1)d^{(n)}_{i-j}(t)d^{(n)}_{j+1}(t)+\frac{1}{2}(i+1)b^{(n)}_{i+1}(t)+r_i(t)\right)\rd t,\\
b_i^{(n+1)}(t)
&=b_i(0)+\int_0^t\left(\sum_{j=0}^i(j+1)b^{(n)}_{i-j}(t)d^{(n)}_{j+1}(t)+\frac{1}{2}\sum_{j=0}^i(j+1)d^{(n)}_{i-j}(t)b^{(n)}_{j+1}(t)\right)\rd t.
\end{split}\end{align}

As we have noted before, the existence of the measure $\hat{\mu}_t$ is not in question and, thus, we do not need to perform an iteration of the $r_i(t)$ terms in $n$. 
%
%Even though $A_t(z) + \sqrt{B_t(z)}$ is the Green's function of some probability measure, we do not want to maintain the identity $ \pi \hat{\mu}_t(x)\lim_{y \rightarrow 0} \Im[A_t(x+i y) + \sqrt{B_t(x+iy)}]$ at every time step. As the method is formulated here, we do not prove the existence of solution to the McKean Vlasov equation; the identity relating $\hat{\mu}$ to the Green's function $A_t + \sqrt{B_t}$ is only seen in the infinite end of the iteration which is sufficient to show that the solution of the Mckean-Vlasov equation has square root behavior.
%
We take large $C>0$, $L>0$ and small $T=1/L$. We first prove by induction that uniformly for $t\in[0,T]$,
\begin{align}\label{e:boundn}
|d_i^{(n)}(t)|, |b_i^{(n)}(t)|\leq  \frac{CM^i e^{Lti}}{(i+1)^2}.
\end{align}
Since $R_t(z)$ is analytic in a neighborhood of $E_0$, we can take $C_0, M>0$ large enough, such that
\begin{align}
|r_i(t)|\leq \frac{C_0M^i }{(i+1)^2}.
\end{align}

We assume that \eqref{e:boundn} holds for $n$, using \eqref{e:recur} we have
\begin{align}\begin{split}
|d_i^{(n+1)}(t)|
&\leq \frac{C_0M^i}{(i+1)^2}+\int_0^t \left(\sum_{j=0}^i
\frac{C^2M^{i+1} e^{Lt(i+1)}}{(i-j+1)^2(j+2)}
+\frac{CM^{i+1} e^{Lt(i+1)}}{2(i+2)}+\frac{C_0M^i }{(i+1)^2}\right)\rd t\\
&\leq \frac{(1+t)C_0M^i}{(i+1)^2}+\sum_{j=0}^i
\frac{C^2M^{i+1} e^{Lt(i+1)}}{L(i+1)(i+2)(j+1)^2}
+\frac{CM^{i+1} e^{Lt(i+1)}}{2L(i+1)(i+2)}\\
&\leq \frac{CM^i e^{Lti}}{(i+1)^2}\left(\frac{(1+t)C_0}{C}+\frac{2CMe^{LT}}{L}+\frac{Me^{TL}}{2L}\right)\leq \frac{CM^i e^{Lti}}{(i+1)^2},
\end{split}\end{align}
provided $C>4C_0$ and $L>8eCM$. Similarly for $b_i^{(n+1)}(t)$,
\begin{align}\begin{split}\label{e:bibound}
|b_i^{(n+1)}(t)|
&\leq \frac{C_0M^i}{(i+1)^2}+\int_0^t\frac{3}{2}\sum_{j=0}^i
\frac{C^2M^{i+1} e^{Lt(i+1)}}{(i-j+1)^2(j+2)}\rd t\\
&\leq \frac{CM^i e^{Lti}}{(i+1)^2}\left(\frac{C_0}{C}+\frac{3CMe^{LT}}{L}\right)\leq \frac{CM^i e^{Lti}}{(i+1)^2},
\end{split}\end{align}
provided $C>2C_0$ and $L>6eCM$. This finishes the proof of claim \eqref{e:boundn}.

In the following we prove that $d_i^{(n)}(t), b_i^{(n)}(t)$ converge uniformly as $n$ goes to infinity, which follows from 
\begin{align}\label{e:difbound}
|d_i^{(n)}(t)-d_i^{(n-1)}(t)|, |b_i^{(n)}(t)-b_i^{(n-1)}(t)|\leq \frac{1}{2^n}\frac{CM^ie^{Lti}}{(i+1)^2}.
\end{align}
In the following we prove \eqref{e:difbound} by induction. We assume that \eqref{e:difbound} holds for $n$, using \eqref{e:recur}, the difference $|d_i^{(n+1)}(t)-d_i^{(n)}(t)|$ is bounded by
\begin{align}\begin{split}
&\phantom{{}={}}
\int_0^t\sum_{j=0}^i(j+1)\left(\left|d^{(n)}_{i-j}(t)-d^{(n-1)}_{i-j}(t)\right|d^{(n)}_{j+1}(t)+d_{i-j}^{(n-1)}(t)\left|d^{(n)}_{j+1}(t)-d^{(n-1)}_{j+1}(t)\right|\right)\\
&+\frac{1}{2}(i+1)\left|b^{(n)}_{i+1}(t)-b^{(n-1)}_{i+1}(t)\right|\rd t\leq  \frac{1}{2^n}\int_0^t \left(\sum_{j=0}^i
\frac{2C^2M^{i+1} e^{Lt(i+1)}}{(i-j+1)^2(j+2)}
+\frac{CM^{i+1} e^{Lt(i+1)}}{2(i+2)}\right)\rd t\\
&\leq \frac{1}{2^n}\frac{CM^ie^{Lti}}{(i+1)^2}\left(\frac{4CMe^{LT}}{L}+\frac{Me^{TL}}{2L}\right)\leq \frac{1}{2^{n+1}}\frac{CM^ie^{Lti}}{(i+1)^2},
\end{split}\end{align}
provided that $L\geq 8eCM$. Similarly the difference $|b_i^{(n+1)}(t)-b_i^{(n)}(t)|$ is bounded by,
\begin{align}\begin{split} \label{e:Bestimates}
&\phantom{{}={}}
\int_0^t\sum_{j=0}^i(j+1)
\left(\left|b^{(n)}_{i-j}(t)-b^{(n-1)}_{i-j}(t)\right|d^{(n)}_{j+1}(t)+b^{(n-1)}_{i-j}(t)\left|d^{(n)}_{j+1}(t)-d_{j+1}^{(n-1)}(t)\right|\right)\\
&+\frac{1}{2}\sum_{j=0}^i(j+1)\left(\left|d^{(n)}_{i-j}(t)-d^{(n-1)}_{i-j}(t)\right|b^{(n)}_{j+1}(t)+d^{(n-1)}_{i-j}(t)\left|b^{(n)}_{j+1}(t)-b_{j+1}^{(n-1)}(t)\right|\right)\rd t\\
&\leq \frac{1}{2^n}\int_0^t \left(\sum_{j=0}^i
\frac{3C^2M^{i+1} e^{Lt(i+1)}}{(i-j+1)^2(j+2)}
\right)\rd t\leq  \frac{1}{2^n}\frac{CM^ie^{Lti}}{(i+1)^2}\frac{6CMe^{LT}}{L}\leq \frac{1}{2^{n+1}}\frac{CM^ie^{Lti}}{(i+1)^2},
\end{split}\end{align}
provided $L\geq 12eCM$.

We denote for $i=0,1,2,\cdots$,
\begin{align}
d_i(t)=\lim_{n\rightarrow \infty}d^{(n)}_i(t),\quad b_i(t)=\lim_{n\rightarrow \infty}b_i^{(n)}(t),
\end{align}
then they satisfy the system of differential equations \eqref{e:recur}, and \eqref{e:coebound} holds.

The same argument as for \eqref{e:bibound}, we get
\begin{align}
|b_i(t)-b_i(0)|\leq \frac{3C^2M^{i+1}(e^{Lt(i+1)}-1)}{(i+1)^2L}.
\end{align}
And thus for $T>0$ small enough, $0\leq t\leq T$, and $z$ on a small circle centered at $E_0$,
\begin{align}
|B_t(z)-B_0(z)|\leq \sum_{i\geq 0}\frac{3C^2(Mz)^{i+1}(e^{Lt(i+1)}-1)}{(i+1)^2L}< |B_0(z)|.
\end{align}
Thus by Rouch{\'e}'s theorem, in a small neighborhood of $E_0$, $B_t(z)$ has a unique simple root at $z=E_t$. Moreover, from our construction, $\bar{B}_t(z)=B_t(\bar z)$, $E_t$ is real. By taking the derivative of $B_t(E_t)=0$ with respect to $t$, and using \eqref{e:BDReq} we get
\begin{align}
\del_t E_t=-\frac{\del_t B_t(E_t)}{\del_z B_t(E_t)}
=-D_t(E_t)=-\hat m_t(E_t)-\frac{V'(E_t)}{2}.
\end{align}
This finishes the proof of Proposition \ref{p:AB}.
\end{proof}

\begin{comment}
We will now mention some consequences of the equation $\partial_t E_t = - m_t(E_t) - \frac{V'(E_t)}{2}$
\begin{lemma}\label{lem:EdgeMov}
 We have the following two relations. Recall the definition of $C_t$ from \eqref{e:ConstDef}.
 \begin{align} \label{e:EdgeChange}
     &|E_t - E_s| = O(|t-s|)\\
     &|C_t - C_s| = O(|t-s|)
 \end{align}
\end{lemma}
\end{comment}
\begin{proof}[Proof of Proposition \ref{prop:TimStab}]
It follows from Proposition \ref{p:AB} that $\hat\mu_t$ has square root behavior. And especially $\hat m_t(E_t) = A_t(E_t) + \sqrt{B_t(E_t)}$ is uniformly bounded for $0\leq t\leq T$. The claim that $E_t$ is Lipschitz in time simply follows from integrating \eqref{e:AE}.

%In the equation $\partial_t E_t = -m_t(E_t) - \frac{V'(E_t)}{2}$, the right hand side is thus bounded by some constant. Integrating this equation thus gives us the first equation in \eqref{e:EdgeChange}.

Next we prove that $C_t$ is Lipschitz in time. We notice that $C_t^2=B_t'(E_t)$. To prove $|C_t-C_s|=\OO(t-s)$, it suffices to prove this for $s=0$. Using the series expansion, as in the notation of \eqref{e:ABpower}, we see that
\begin{equation*}
C_t^2-C_0^2=B_t'(E_t) - B'_0(E_0) = b_1(t) -b_1(0) + (E_t-E_0)\sum_{i=2}^{\infty}  i b_i(t)(E_t-E_0)^{i-2}
\end{equation*}

Using the differential equation \eqref{e:ode}, we see that $b_1(t) - b_1(0)=O(t)$. While earlier we have shown that $(E_t - E_0)$ is of $O(t)$. The infinite sum converges provided we take $T$ sufficiently small.
Therefore, it follows that $C_t^2 - C_0^2 = O(t)$. Since $C_0$ is bounded away from $0$, this would imply that $C_t -C_0 = O(t)$ as desired.

\end{proof}

\begin{comment}
\begin{corollary}
If the probability measure $\hat \mu_0$ satisfies Assumption \ref{asup}, then it is stable in the sense of Definition \ref{d:stable}.
\end{corollary}

\begin{proof}
The claim \eqref{e:edgeEqn} follows directly from Proposition \ref{p:AB}. For \eqref{e:hatmt}, since $z=E_t$ is a zero of $B_t$ with multiplicity one, there exists a function $B_t'(z)$ is analytic in a neighborhood of $E_t$ such that
\begin{align}
\hat m_t(z)=A_t(z)+B'_t(z)\sqrt{z-E_t}.
\end{align}
Since $|\Im[A_t(z)]|, |\Im[B'_t(z)]|\asymp \Im[z]$, $\Im[\hat m_t(z)]\asymp \Im[\sqrt{z-E_t}]$, and the claim \eqref{e:hatmt} follows.
\end{proof}
\end{comment}

\section{Proof of Corollary \ref{c:mesoCLT}}
\begin{proof}[Proof of Corollary \ref{c:mesoCLT}]
The corollary follows from Theorem \ref{t:mesoCLT} and the rigidity estimate in Theorem \ref{Thm:EdgeRidgity} by the same argument as in \cite{MR3116567}.

\begin{comment}
We can approximate the test function $\psi_{\eta}(x)$ by its convolution with a Cauchy distribution on the scale $\eta$, as $\varepsilon$ goes to $0$,
\end{comment}
We will start considering by consider functions that are can be represented as a convolution of some function with the Cauchy Distribution.
\begin{align}
\psi_{\eta}^{\varepsilon}\deq\psi_{\eta}* \frac{1}{\pi}\frac{(\varepsilon \eta)}{x^2+(\varepsilon\eta)^2}.
\end{align}
We let 
\begin{align}
\hat{\cal L}(\psi_{\eta}^{\varepsilon})
&\deq \sum_{i=1}^N \psi_{\eta}^{\varepsilon}(\lambda_i(t))-N\int_{\bR} \psi^{\varepsilon}_{\eta}(x) \rd \hat\rho_t(x)\\
&=\frac{1}{\pi}\int (N \eta \Im\left[ m_t(E_t + x \eta+\varepsilon\eta \ri)-\hat m_t(E_t + x\eta+\varepsilon\eta\ri) - \frac{2-\beta}{4\beta N} \frac{1}{x\eta + \varepsilon \eta\ri }  \right]) \psi(x)\rd x \\
&- \frac{2-\beta}{4 \pi \beta}  \int \left[\frac{\varepsilon  }{ x ^2 + \varepsilon ^2}\right]\psi(x)\rd  x
\end{align}

For technical reasons, it would be beneficial to consider a modified version of $\cal L$ that will only consider the $c N $ closest eigenvalues near the edge where $c$  is chosen such that $|\gamma_{cN}(t) - E_t| \le r/2$ where $r$ ball of radius $r$ as in \eqref{e:defcDt} and $\gamma_{cN}(t)$is the classical location of the $cN$th eigenvalue with respect to the measure $\hat \rho_t$.  Intuitively, one should understand that for a function of compact support, the behavior of eigenvalues greater than a scale $\eta$ near the edge will not matter.

As such, we will define a modified $\hat{L}(\psi^{\varepsilon}_{\eta}):= \sum_{i=1}^{cN} \psi_{\eta}^{\varepsilon}(\lambda_i(t)) - N \int_{\gamma_{cN}(t)}^{\infty} \psi^{\varepsilon}_{\eta}(x) \rd \hat\rho_t(x)$.

Let $g_t(z) = \sum_{i=1}^{cN}\frac{1}{\lambda_i -z}$ and $\hat{g}_t(z)= \int_{\gamma_{cN}}^{\infty} \frac{1}{x-z} \hat{\rho}_t(z)$.
To simplify notation, we will define
\begin{equation}
    \tilde{h}_{\varepsilon}(x) = N\eta \Im[g_t((E_t + x \eta) + \varepsilon \eta \ri) - \hat{g}_t((E_t + x \eta)+ \varepsilon \eta \ri) - \frac{2- \beta}{4 \beta N} \frac{1}{x\eta + \varepsilon \eta \ri}]
\end{equation}

As before we have the identity
\begin{equation}
    \hat{L}(\psi^{\varepsilon}_{\eta}) = \frac{1}{\pi} \int_{-\infty}^{\infty} \tilde{h}(x) \psi(x) \rd x - \frac{2-\beta}{4 \pi \beta}  \int \left[\frac{\varepsilon  }{ x ^2 + \varepsilon ^2}\right]\psi(x)\rd  x
\end{equation}

We will define $L(\psi^{\varepsilon}_{\eta}):=\frac{1}{\pi} \int_{-\infty}^{\infty} \tilde{h}(x) \psi(x) \rd x$ and show that this has Gaussian Varaince for functions $\psi$ that have decay of the order $x^{-2}$.

\begin{comment}
Note that if we took the limit $\epsilon \rightarrow 0$ at this point, we would get the value of $\psi$ at the point $0$. This is due to the fact that the convolution with the Poisson Kernel behaves like a delta function for continuous , and therefore $H^{s}$ for $s>1$ functions.
\end{comment}

\begin{comment}
Thanks to Theorem \ref{t:mesoCLT}, $\cal L(\psi_{\eta}^{(\varepsilon)})$ is asymptotically gaussian, with variance given by
\begin{align}\begin{split}
%\cov \left[\cL(\psi_{\eta}^{(\varepsilon)})\right]
\sigma_\varepsilon^2=&-\frac{1}{4\pi^2}\int_\bR\int_\bR\left( K_{\rm edge}(x+\veps\eta\ri, x'+\veps\eta\ri)-K_{\rm edge}(x-\veps\eta\ri, x'+\veps\eta\ri)\right.\\
&\left.-K_{\rm edge}(x+\veps\eta\ri, x'-\veps\eta\ri)+K_{\rm edge}(x-\veps\eta\ri, x'-\veps\eta\ri)\right)\psi(x)\psi(x')\rd x \rd x'.
\end{split}
\end{align}
\end{comment}

{
\begin{comment}
We will now prove that $\cal L(\psi_{\eta}^{\epsilon})$ is asymptotically gaussian for any function $\psi$ that has the decay result $\psi(x) \le x^{-2}$.
\end{comment}
For convenience of notation, we define
\begin{equation}
    h_{\varepsilon}(x) = [N \eta \Im[m_t((E_t+ x \eta)+ \varepsilon \eta \ri) - \hat{m}_t((E_t +x \eta) + \varepsilon \eta \ri) - \frac{ 2- \beta}{4 \beta N} \frac{1}{x \eta + \varepsilon \eta \ri}] 
\end{equation}

\begin{comment}
We can write $\cal L(\psi_{\eta}^{\varepsilon})$ as
\begin{equation}
    \pi^{-1} \int_{-\infty}^{\infty}h(x) \psi(x) \rd x 
\end{equation}
\end{comment}

Recall that Theorem \ref{t:mesoCLT} allows us to show the asymptotic Gaussian behavior of a sum of the following form $\sum_{i=1}^{k} a_k h_{\epsilon}(x_k)$. 

Ideally, we would like to find for each $N$ a finite sum of the above form $F_N$ such that we have 
$\lim_{N \rightarrow \infty}  (\bE[e^{i \xi F_N}] - \bE[e^{i \xi  L(\psi^{\varepsilon}_{\eta})}]) =0$ and such that $\lim_{N \rightarrow \infty} \bE[e^{i \xi F_n}]$ has a limit. The existence of such a sequence comes from the uniformity of the approach in Theorem \ref{t:mesoCLT}.

For technical reasons, for each fixed constant $C$, we will find a sequence $F_N^{C}$  such that $\lim_{N \rightarrow \infty}(\bE[e^{i \xi F^{C}_N}] - \bE[e^{i\xi \cal L(\psi^{\varepsilon}_{\eta})}]) \le C^{-1}$ and $\lim_{C \rightarrow \infty} \lim_{N \rightarrow \infty} \bE[e^{i \xi F^{C}_n}]$ exists.

First fix $C$; we will start by  dividing $[-C,C]$ into sufficiently many intervals $[-C,C] = \cup_{k=1}^{M}[b_k,b_{k+1}]$ such that along each interval we have the following uniform continuity estimate.

If there exists $k$ such that $x,y \in [b_k,b_{k+1}]$ then
\begin{equation*}
    |h_{\varepsilon}(x) - h_{\varepsilon}(y)| \le N^{-2} 
\end{equation*}

Our approximation function will be as follows: we let $a_k = \int_{a_k}^{a_{k+1}} \psi(x) \rd x$ and $F_N =\frac{1}{\pi} \sum_{i=1}^{M}a_k h_{\varepsilon}(b_k)$.

We have the following general inequality that allows us to compare the characteristic functions for two random variables $a$ and $b$ at the point $\xi$
\begin{align}\label{e:approx2}
|\bE\left[e^{\ri \xi a}\right]-\bE\left[e^{\ri \xi b}\right]|
\leq |\xi|\bE\left[(a -b)^2\right]^{1/2}.
\end{align}
We will use the above inequality for $a$ is our appropriately chosen $F^{C}_N$ and $b= \cal L(\psi_{\eta}^{\varepsilon})$

One should notice that our continuity bound gives us that
\begin{align}
    |F^{C}_N - L(\psi^{\varepsilon}_{\eta})| &\le N^{-2} \int_{-C}^{C} |\psi| \rd x + \int_{-C}^{C} |h_{\varepsilon}(x)-\tilde{h}_{\varepsilon}(x)||\psi(x)| \rd(x)  +|\int_{- \eta^{-1/6}}^{-C} h_{\varepsilon}(x) \psi(x) \rd x| \\
    & +\int_{-\eta^{-1/6}}^{-C} |h_{\varepsilon}(x) - \tilde{h}_{\varepsilon}(x)||\psi(x)| \rd x +|\int_{(\gamma_{cN}(t) -E_t)\eta^{-1}-1}^{-\eta^{-1/6}} h_{\varepsilon}(x) \psi(x) \rd x|\\
   & +  + \int_{(\gamma_{cN}(t) -E_t)\eta^{-1}-1}^{-\eta^{-1/6}} |h_{\varepsilon} - \tilde{h}_{\varepsilon}(x)|\psi(x)| + |\int_{-\infty}^{(\gamma_{cN}(t)-E_t) \eta^{-1} -1} \tilde{h}_{\varepsilon}(x) \psi(x) \rd x| +\cdots
\end{align}
where the $\cdots$ indicate corresponding terms on the positive real axis.
On the complement of a set With exponentially small probability, we know that all of the eigenvalues $\lambda_1\cdots\lambda_{cN}$ should lie in $(-\gamma_{cN}(t) - \eta, \infty)$ and all of the eigenvalues $\lambda_{cN+1},\cdots\lambda_N$ lie outside $(-\gamma_{cN}(t) + \eta, \infty)$. In our further computations, we will be considering that we are in the set of such rigidity as the contribution of the complement set is exponentially vanishing.

We have the bound $|h_{\varepsilon}(x) - \tilde{h}_{\varepsilon}(x)| \le N \eta \frac{\varepsilon \eta}{(\varepsilon \eta)^2 +(E_t + x \eta- \gamma_{cN} - \eta )^2 }$ for $x \ge -\gamma_{cN}(t) + \eta$ since we assume the eigenvalues from $\lambda_{cN+1}(t)\cdots \lambda_N(t)$ lie in $(-\gamma_{cN}(t) + \eta ,\infty)$ with high probability. Similarly, we have the bound $\tilde{h}_{\varepsilon}(x) \le N \eta \frac{\varepsilon \eta}{(\varepsilon \eta)^2 +(E_t + x\eta -\gamma_{cN} +\eta)^2}$ as we assume the eigenvalues $\lambda_1\cdots\lambda_{cN}$ lie in $(-\gamma_{cN}(t) + \eta, \infty)$.

We can estimate the second integral $\int_{-C}^{C}|h_{\varepsilon}(x) - \tilde{h}_{\varepsilon}(x)| |\psi(x)| \rd x$ as $\OO(C N \eta^2)$ where the constant appearing does not depend on $N$ or $C$.  One should observe that as $N \rightarrow \infty$, the term $(E_t + x \eta - \gamma_{cN} - \eta)$ stays bounded below by a constant, say $2^{-1}(E_t - \gamma_{cN})$.  Thus, $\OO(C N \eta^2)$ is a natural bound which we see goes to 0 as $N \rightarrow \infty$ as $\eta \ll N^{-1/2 - \epsilon}$.

On $[-\eta^{-1/6},-C]$ we have the bound $h_{\varepsilon}(x) \le 1+\OO(N^{-\delta})$ essentially from the computations in \ref{t:mesoCLT}. By using the $x^{-2}$ decay of $\psi$, the integral of $h_{\varepsilon} \psi(x)$ over this region is bounded by $\frac{1}{C}$.

For the fourth integral, we can again use the bound $|h_{\varepsilon} - \tilde{h}_{\varepsilon}(x)| \le \OO(N \eta^2)$ as for $N$ very large, we can bound $(E_t + x \eta - \gamma_{cN} - \eta)$ from below by $2^{-1}(E_t - \gamma_{cN})$. We now also use the decay of $\psi(x) \le \frac{1}{x^2}$ to bound the integral of $\int_{-\eta^{-1/6}}^{-C}|h_{\varepsilon}-\tilde{h}_{\varepsilon}||\psi(x)| \rd x$ by $\OO(N \eta^2 C^{-1})$ where the constant does not depend on $C$ or on $N$. When C is fixed $N\eta^2 \rightarrow 0$, so this part goes to 0 as $N \rightarrow \infty$.

{
On the region $[(\gamma_{cN}(t) - E_t) \eta^{-1} -1, -\eta^{-1/6}] $ we will use the local law bound $h_{\varepsilon}(x) \le M=\OO((\log N)^k)$ for some integer $k$. Since $\psi$ has decay $x^{-2}$, we will be able to bound the integral of $h_{\varepsilon} \psi$ in this region by $M \eta^{1/6}$. We have chosen $\eta \ll N^{-\epsilon}$ for some $\epsilon$, so in the infinite limit the contribution of this integral is 0.

To bound the other integral on the region $[(\gamma_{cN}(t) - E_t) \eta^{-1} -1 , - \eta^{-1/6}]$, we use the bound $|h_{\epsilon} - \tilde{h}_{\epsilon}| \le \frac{\varepsilon \eta}{(\varepsilon \eta)^2 +(E_t + x \eta - \gamma_{cN}(t) +\eta)  }$ on the region $x \in [(\gamma_{cN} -E_t) \eta^{-1} +1, \infty]$. The integral we have to bound, up to constants that do not depend on $N$ or $C$ are
\begin{align}
    N \eta^3 \int_{(\gamma_{cN} -E_t) - \eta}^{(\gamma_{cN}-E_t) + 2 \eta} \frac{1}{\eta^2} \frac{1}{y^2} \rd y + N \eta^3 \int_{ 2 \eta}^{(E_t - \gamma_{cN})- \eta^{5/6}} \frac{1}{y^2} \frac{1}{ (E_t - \gamma_{cN}-y)^2} 
\end{align}
The second integral above can be bounded by using partial fractions. We get the bound $\OO(N \eta^2)$; the integrand in the first integral takes value $\OO(N \eta)$, but this is over a region of size $\eta$. Therefore the bound on this integral is $\OO(N \eta^2)$.

For the region $(-\infty, (\gamma_{cN}(t) - E_t) \eta^{-1} -1)$ we are outside of the spectrum corresponding to $\tilde{h}_{\varepsilon}$ and we use the bound $\tilde{h}_{\varepsilon}(x) \le N \eta \frac{ \varepsilon \eta}{(\varepsilon \eta)^2+(E_t + x \eta - \gamma_{cN} + \eta)}$. By a standard change of variable, we see that we have to bound
\begin{align}
    N \eta^3 \int_{0}^{\infty} \frac{\varepsilon}{y^2 + (\varepsilon \eta)^2} \frac{1}{(y + E_t - \gamma_{cN}(t) -\eta)^2 } \rd y &\le 2 N \eta^3 \int_{\varepsilon \eta}^{\infty} \frac{\varepsilon}{y^2} \frac{1}{(y + (E_t - \gamma_{cN}(t)))^2} \rd y \\
    &+ 2 N \eta^3 \int_{0}^{\varepsilon \eta} \frac{\varepsilon}{(\varepsilon \eta)^2} \frac{1}{(y +(E_t - \gamma{cN}(t)))^2} \rd y 
\end{align}
To evaluate the first integral, one can use a partial fraction decomposition. The order of growth is $\OO(N \eta^2)$. For the second, up to a constant that does not depend on $N$ the integrand is of order $\OO(N \eta)$, but we are integrating this over a region of size $2\eta$. Thus, the size of this is $\OO(N \eta^2)$ which goes to 0 as $N \rightarrow \infty$.

}

The analysis for the integrals along the positive real axis can be checked via the same line of reasoning. We have thus shown that $\lim_{N \rightarrow \infty} [\bE(e^{i \xi F_N^{C}}) - \bE(e^{i \xi \cal L(\psi_{\eta}^{\epsilon})}] \le \OO(C^{-1}) $ where the constant in $\OO(C^{-1})$ does not depend on $C$.

We now only need to show that the double limit $\lim_{C \rightarrow \infty} \lim{N\rightarrow \infty} \bE(e^{i \xi F_N^{C}})$ exists.

From Theorem $\ref{t:mesoCLT}$, we are able to show that 
\begin{align}
    \log \bE[e^{i \xi F^{C}_N}] = -\frac{1}{8\pi^2} (1+ \OO(N^{-\delta}))&\sum_{i,j=1}^{M} (\eta)^2 ( K_{\rm edge}(x_i \eta + \eta \varepsilon \ri, x_j \eta + \eta \varepsilon \ri)- K_{\rm edge}(x_i \eta - \eta \varepsilon \ri, x_j \eta + \eta \varepsilon \ri)\\
    &- K_{\rm edge}(x_i \eta + \eta \varepsilon \ri, x_j \eta - \eta \varepsilon \ri) +K_{\rm edge}(x_i \eta + \eta \varepsilon \ri, x_j \eta + \eta \varepsilon \ri) ) a_i a_j  \\
    &+ \OO(N^{-\delta}) \sum_{i=1}^{M} \frac{2- \beta}{4 \beta}  \frac{\eta}{x_i \eta + \varepsilon \eta \ri} a_k
\end{align}
One should note that for fixed $\varepsilon$, the terms $(\eta)^2 K_{\rm edge}$  and $\frac{\eta}{x_i \eta + \epsilon \eta \ri}$ are bounded above.  Thus  we see that the double sum and single sum are bounded by a constant factor independent of $N$ times $(\int_{-\infty}^{\infty} |\psi| \rd x)^2$ and $\int_{-\infty}^{\infty} |\psi| \rd x$ respectively. 

To show this has a limit, we now only need to show that 
$\sum_{i,j=1}^{M} (\eta)^2 ( K_{\rm edge}(x_i \eta + \eta \varepsilon \ri, x_j \eta + \eta \varepsilon \ri)\cdots)a_i a_j$ has a limit. However, we merely need to use continuity of $(\eta)^2K_{\rm edge}$ on a compact interval to see that as our division of intervals $\cup_{k=1}[b_k,b_{k+1}]$ gets increasingly fine this double sum manifestly approaches the integral $\int_{-C}^{C} \int_{-C}^{C} (\eta)^2 (K_{\rm edge}..) \psi(s) \psi(t) \rd s \rd t$

We have thus shown
\begin{equation}
    \lim_{N \rightarrow \infty} \bE[e^{i \xi F_N^{C}}] = \exp{\int_{-C}^{C} \int_{-C}^{C} (\eta)^2 (K_{\rm edge}..) \psi(s) \psi(t) \rd s \rd t}
\end{equation}
We can clearly take the limit as $C \rightarrow \infty$ to see that

    \begin{align}
\bE\left[e^{\ri \xi \cal L(\psi_{\eta})}\right]
=e^{-\xi^2 \sigma_\varepsilon^2/2}+\oo(1),
\end{align}
where 
\begin{align}
    \sigma_\varepsilon^2 =\int_{-\infty}^{\infty} \int_{-\infty}^{\infty} (\eta)^2 (&K_{\rm edge}(x \eta+  \varepsilon \eta \ri, x\eta+ \varepsilon \eta \ri)-K_{\rm edge}(x \eta+ \varepsilon \eta \ri, x\eta- \varepsilon \eta \ri)\\
    -&K_{\rm edge}(x \eta-  \varepsilon \eta \ri, x\eta+ \varepsilon \eta \ri)+K_{\rm edge}(x \eta+  \varepsilon \eta \ri, x\eta+ \varepsilon \eta \ri)) \psi(s) \psi(t) \rd s \rd t 
\end{align}

{ One can easily check that this expression is equivalent to
\begin{equation}
    \frac{1}{4\pi^2 \beta}\int_{\bR}\int_{\bR}\frac{(\psi^{\varepsilon}(x^2)-\psi^{\varepsilon}(x'^2))^2}{(x-x')^2}\rd x\rd x
\end{equation}
}

We now only need to extend our logic to a larger class of functions.
Notice that our argument above extends to finite sums of the form
$\sum_{k=1}^{n}  \psi^{y_k}_k$. Thus, if we are able to find an appropriate approximating sequence of this type in $H^s$, { then we will be finished as our covariance expression is manifestly a continuous functional in the Banach space $H^s$.}  The approximation involves a Littlewood-Paley decomposition of the same type as in \cite[Section 3]{MR3116567}.

}

We can write
\begin{equation}
    \psi = \sum_{k=-1}^{\infty} P_{2^{-k}}*g_k(x)
\end{equation}
Thus, we will have
\begin{equation}
    L(\psi)= \sum_{k=1}^{\infty} L(P_{2^{-k}} * g_k(x))
\end{equation}
Since $\psi$ is of compact support $L(\psi) = \mathcal{L}(\psi)$ in the limit as $N \rightarrow \infty$. We will also remark here that we get uniform bound $P_{2^{-k}} * g_k(x) \le \frac{C_l}{(2\pi x)^{l} {2^k}(l-1)}$ as in \eqref{e:DecayInt}. Thus, replacing $L(P_{2^{-k}}*g_k(x))$ with $\mathcal{L}(P_{2^{-k}}*g_k(x))$ will give a error that is summable in $k$ and vanish in $N$. Our later analysis will involve bounding estimates on $\mathcal{L}(P_{2^{-k}}*g_k(x))$

Our approximations will be the finite sums $\sum_{k=-1}^{M} P_{2^{-k}}*g_k(x)$. In order to apply $\eqref{e:approx2}$, we will need to bound $E[|\sum_{k=M+1}^{\infty} \mathcal{L}(P_{2^{-k}} * g_k(x))|^2]$. The latter expression can be explicitly written out as
\begin{align*}
     \sum_{k,l = M+1}^{\infty} \int_{-\infty}^{\infty}\int_{\infty}^{\infty}g_k(s) g_l(t) E[(N\eta)^2 &\Im\left[m_t(E_t + \eta s + \ri 2^{-k}\eta) - \hat m_t(E_t + \eta s + \ri 2^{-k}\eta) - \frac{2}{4 \beta N} \frac{1}{\eta s + \ri 2^{-k}}\right]\\
    &\Im\left[m_t(E_t + \eta t + \ri 2^{-l}\eta) - \hat m_t(E_t + \eta t + \ri 2^{-l}\eta) - \frac{2}{4 \beta N} \frac{1}{\eta t +\ri 2^{-l}}\right]] \rd s \rd t
\end{align*}

We can apply the Cauchy-Schwartz inequality on the inside of the expectation. This decouples the variables $s$ and $t$ and the resulting sum can be written as the square of the following term.

\begin{equation} \label{e:mainint}
\sum_{k=M+1}^{\infty} \int_{-\infty}^{\infty} g_k(s)E\left[(N \eta)^2 \left|\Im\left[m_t(E_t + \eta s + \ri 2^{-k}\eta) - \hat m_t(E_t + \eta s + \ri 2^{-k}\eta) - \frac{2}{4 \beta N} \frac{1}{\eta s +\ri 2^{-k}}\right]\right|^2\right]^{1/2} \rd s
\end{equation}
as in the equation below (21) in \cite[Section 4]{MR3116567}.

In order to justify the approximation, all we need to do is prove that the above sum is finite; then we would be able to show that the finite sums in the Littlewood-Paley decomposition provides a sufficiently good approximating sequence.
To do this at mesoscopic scales, we need additional decay estimates on $g_k(s)$. It is defined on page 11 of \cite{MR3116567} as
\begin{equation}
\hat{g_k}(\xi) = e^{2^{-k}|\xi|} \hat{\omega}(2^{-k} \xi) \hat{\psi}(\xi)
\end{equation}
which gives us that
\begin{equation}
    g_k(x) = \mathcal{F}^{-1}(e^{2^{-k} |\xi|} \hat{\omega}(2^{-k} \xi)) * \psi
\end{equation}
where $\mathcal{F}^{-1}$ is the inverse Fourier transform.

Let us get estimates on the value of $\mathcal{F}^{-1}(e^{2^{-k} |\xi|} \hat{\omega}(2^{-k} \xi))$ at large values.
We have
\begin{align} \begin{split} \label{e:DecayInt}
  \mathcal{F}^{-1}(e^{2^{-k} |\xi|} \hat{\omega}(2^{-k} \xi)) &= \int_{-\infty}^{0} e^{-2^{-k} \xi} \hat{\omega}(2^{-k} \xi) e^{2\pi i \xi x} \rd \xi + \int_{0}^{\infty} e^{2^{-k} \xi} \hat{\omega}(2^{-k} \xi) e^{2\pi i \xi x} \rd \xi\\
  &= 2^{k} \int_{-\infty}^{0} e^{- \xi} \hat{\omega}( \xi) e^{2\pi i 2^{k} \xi x} \rd \xi + 2^{k}\int_{0}^{\infty} e^{ \xi} \hat{\omega}( \xi) e^{2\pi i 2^k \xi x} \rd \xi\\
  &= \frac{1}{(2 \pi i x)^L (2^k)^{L-1}}  \left[(-1)^{L}\int_{-\infty}^{0} (e^{- \xi} \hat{\omega}( \xi))^{(L)} e^{2\pi i 2^{k} \xi x} \rd \xi + (-1)^{L} \int_{0}^{\infty} (e^{ \xi} \hat{\omega}( \xi))^{(L)} e^{2\pi i 2^k \xi x} \rd \xi \right]
\end{split}\end{align}
where in the second line we rescaled $\xi$ by $2^k$ and in the third line we performed an integration by parts L times, we can do this safely since $\hat{\omega}$ is smooth and supported away from 0.
At this point, we remark that for the function $g_{-1}$, we are able to integrate by parts twice even if the function $\hat{\omega}$ is not supported away from the origin; this implies that all functions in our decomposition have decay at least $x^{-2}$.

Using this, we can now get decay estimates on the convolution $g_k*\psi$. Without loss of generality, we may assume that the support of $\psi$ is $[-1,1]$ by scaling. We get
\begin{equation}
g_k* \psi(x) = \int_{-1}^{1} g_k(x-y) \psi(y) \rd y
\end{equation}
This can clearly be bounded by $\frac{C_L}{2^{(L-1)k}(x-1)^L}$ when $x\ge 2$, while we have  a bound of $\frac{C_L}{2^{(L-1)k}(x+1)^L}$ when $x \le 1$. 

We will now attempt to bound the value of \eqref{e:mainint}. We need to divide this into a couple cases based on whether we have $2^{5k} \eta \le N^{-\epsilon}$ or $2^{5k} \eta \ge N^{-\epsilon}$  and whether $2^{-k} \eta$ is greater or less than $N^{-2/3 + \epsilon}$.

Case 1: $2^{5k} \eta \le N^{-\epsilon/2}$ and $2^{-k} \eta \ge N^{-2/3+\epsilon/2}$

To simplify notation as before we denote $h^k = h_{2^{-k}}$

We divide the integral \eqref{e:mainint} into four parts.
\begin{align}
    &\int_{-\infty}^{-2^{-5/6 k}\eta^{-1/6}} g_k(s) [\bE (h^k(s))^2]^{1/2} \rd s \\
    &+ \int_{-2^{-5/6 k}\eta^{-1/6}}^{-2} g_k(s)[\bE (h^k(s))^2]^{1/2} \rd s
    + \int_{-1}^{1} g_k(s)  [\bE (h^k(s))^2]^{1/2} \rd s
\end{align}
as well as the analogue of the first three integrals on the positive real axis.
\begin{comment}
For the first integral over the region $(-\infty, \eta^{-1})$, we can use the trivial bound $h^k(s) \le N$ since we are outside of the spectrum . We also have the decay bound $g_k(s) \le C_L x^{-L} 2^{-k(L-1)} $. Integrating this, we get the bound $C_L N\eta^{(L-1)} 2^{-k(L-1)}$. We know that $\eta \ll N^{-\epsilon}$; therefore choosing $L$ such that $(L-1) \epsilon \ge 1$ we see that the contribution from this integral is summable in $N$.
\end{comment}

For the first integral over the region $(-\eta^{-1},2^{-5/6k} \eta^{-1}) $ , we use decay of $g_k(s)$ and the trivial bound $h^k \le N 2^{k} $.  We can bound this integral by the integral of  $\frac{C_L N }{2^{(L-2)k} x^L}$ from $2^{-5/6k \eta^{-1/6}}$ to $\infty$. One can see that the order of this term is $\frac{\hat C_L N 2^{5/6(L-1)k} \eta^{(L-1)/6}}{2^{k(L-2)}}$ uniformly in $k$. We can fix a constant $L$ such that $(L-2) - 5/6(L-1) >1$ to ensure the summability of this sequence in $k$. We can also choose $\epsilon (L-1)/6  \ge 1$ to get $\eta^{(L-1)/6} N \ll 1$. Here we used the fact that $\eta \ll N^{-\epsilon}$. Thus, this term is summable in $k$ and the resulting sum in $k$ will go to $0$ and $N$ goes to $\infty$.

    For the second integral, we also use decay of $g_k(s)$ but with a superior variance estimate. We are in the regime where we can apply the variance estimate of Theorem \ref{t:mesoCLT}. The expectation can be bounded by $2^{k}[1+\OO(N^{-\epsilon})]$. The integral will be upper bounded by $2^{-k} \int_{2}^{\infty} \frac{1}{(x-1)^3}[1+\OO(N^{-\epsilon})]$. This term is summable in $k$ and ,thus, will get smaller under better approximations.
    
The third integral will be bounded by using the Cauchy-Schwartz identity and the following bound on the summability of the $L^2$ norms of the $g_k$. 
\begin{equation}\label{e:summable}
    \sum_{k=-1}^{\infty} 2^{2ks}  ||g_k||_{L^2}^2 \le C ||\psi||_{H^s}
\end{equation}
Refer to Theorem 5 of \cite{MR3116567}.
Using our expectation bound of $2^{k}[1+ \OO(N^{-\epsilon}]$ the Cauchy-Schwartz inequality shows us that the integral is bounded by
\begin{equation}
    2 {\int_{-\infty}^{\infty} 2^{2k(1+\tilde \epsilon)}|g_k|^2 \rd x}^{1/2}
\end{equation}

$\tilde \epsilon$ can be chosen carefully enough to ensure that we can use AM-GM to bound the above by a term of the type $\eqref{e:summable}$ and $2^{-\hat{\epsilon}k}$, which is summable in k. Namely, if we know that $\phi$ is in the class $H^{1+ \gamma}$ one can choose $\tilde \epsilon$ to be $\gamma/2$ and $\hat{\epsilon}$ to be $\gamma/2$.

Case 2: $2^{5k} \eta \ge N^{-\epsilon/2}$ or $2^{-k} \eta \ge N^{-2/3+\epsilon/2}$

We will remark in this case that we have $2^k \ge N^{\hat\epsilon}$ for some positive $\hat\epsilon$. Notice that $\eta \ge N^{-\epsilon}$ implies that $2^{5k} \ge N^{\epsilon/2}$. In the other case,  the fact that $2^{k} \eta \ge N^{-2/3}$ and $\eta \ge N^{-2/3 + \epsilon}$ allows us to lower bound $2^k$ by $N^{\epsilon}$

The differences in this case relate to the fact that for low frequencies we must use another variance bound
\begin{equation}
\bE[(h^k(s))^2]^{1/2}  \le \frac{1}{2^{-k(1+\delta)}} 
\end{equation}
 Thus when $2^{-k} \eta \ge N^{-2/3+\epsilon/2}$, this is merely a consequence of the Local Law.

When $2^{-k}\eta \le N^{-2/3 + \epsilon/2}$ this comes from
\begin{equation}
    \Im[m_t(x+ i 2^{-k} \eta)] \le \frac{1}{N (2^{-k})^{1+ \delta} \eta}
\end{equation}
for sufficiently small $\delta$. Clearly the bound on the expectation is the square of the above quantity. We will prove the above identity via monotonicity.

We have 
\begin{equation}
    \Im[m_t(x+ i 2^{-k} \eta)] \le \frac{\Im[m_t(x+i N^{-2/3 + \hat{\delta}})] N^{-2/3+ \hat{\delta}}}{2^{-k} \eta} 
\end{equation}
Using the local law as well as the square root behavior of $\hat{m}_t$ allows us to bound $\Im[m_t(x + i N^{-2/3 + \hat{\delta}})]N^{-2/3+ \hat{\delta}} $ by $N^{-1+  \frac{3}{2}\hat{\delta} }$, where $\hat{\delta}$ can be chosen to be as small as possible. The quantity on the right hand side of the above equation can thus be written as
\begin{equation}
    \frac{N^{\frac{3}{2}\hat{\delta}}}{N 2^{-k} \eta}
\end{equation}

If we fix $\delta$, we see that $\hat{\delta}$ can be chosen sufficiently small to allow us to allow $(2^{k})^{\delta} \ge N^{\frac{3}{2}\hat{\delta}}$.

Now we write the integral \eqref{e:mainint} as two parts
\begin{align}
    &\int_{-\infty}^{-2}  g_k(s) \bE[(h_k(s))^2]^{1/2} \rd s 
     +\int_{-2}^{2} g_k(s) \bE[(h_k(s))^2]^{1/2} \rd s
\end{align}
and, again, a final integral would be the analogue of the first in the positive real axis.

The first integral can be bounded by using decay again. $g_k(s) \le \frac{C_L}{(x-1)^L 2^{k(L-1)}}$ while the expectation is bounded by $2^{k}N$. The resulting integral is bounded by $\frac{C_L N}{2^{k(L-1)}}$. Notice that since we had the lower bound $2^k \ge N^{\tilde \delta}$, we can choose $L$ such that $(L-1) \tilde \delta \ge 1$. This would imply that $\frac{N}{2^{k(L)}} \le 2^{-k}$. This sum would then be summable.

We apply Cauchy-Schwartz to bound the second integral in a manner similar to when it was done in the previous case. This integral is less than
\begin{equation}
    [\int_{-1}^{1} 2^{2k(1+\hat{\epsilon})}|g_{k}(s)|^2 \rd s]^{1/2} \int_{-1}^{1} 2^{-2k(1+ \hat{\epsilon})} 2^{2k(1+\delta)} \rd s
\end{equation}
Recall that we can choose $\delta$ as small as we want. After we fix $\hat{\epsilon}$, one merely needs to choose $\delta < \hat{\epsilon}$ and the second integral merely becomes a decaying multiplicative factor. At the very minimum, it is bounded by a constant factor uniform in $k$. As before, if we know that $\psi$ is in $H^(1+\gamma)$, we can choose $\hat{\epsilon}$ to be $\gamma/2$ and perform AM-GM inequality on the integral to get out a square as in \eqref{e:summable}.

This shows that the desired sum is finite and therefore we have a good approximating sequence in the form of finite sums in the Paley-Littlewood decomposition.
\end{proof}

\bibliography{References}{}

\begin{thebibliography}{10}

\bibitem{ArkaZiliang}
Arka Adhikari and Ziliang Che.
\newblock The edge universality of correlated matrices.
\newblock {\em preprint, arXiv:arXiv:1712.04889}, 2017.

\bibitem{AltEdge1}
Johannes Alt, L\'aszl\'o Erd\H~os, Torbin Kruger, and Dominik Schroder.
\newblock Correlated random matrices:band rigidity and edge universality.
\newblock {\em preprint, arXiv:arXiv:1804.07744}, 2018.

\bibitem{fixedBourgade}
Paul Bourgade, L\'aszl\'o Erd\H~os, Horng-Tzer Yau, and Jun Yin.
\newblock Fixed energy universality for generalized wigner matrices.
\newblock {\em Comm. Pure Appl. Math}, 69:1815--1881, 2016.

\bibitem{MR2905803}
Paul Bourgade, L{\'a}szl{\'o} Erd{\H{o}}s, and Horng-Tzer Yau.
\newblock Bulk universality of general {$\beta$}-ensembles with non-convex
  potential.
\newblock {\em J. Math. Phys.}, 53(9):095221, 19, 2012.

\bibitem{MR3253704}
Paul Bourgade, L{\'a}szl{\'o} Erd{\"o}s, and Horng-Tzer Yau.
\newblock Edge universality of beta ensembles.
\newblock {\em Comm. Math. Phys.}, 332(1):261--353, 2014.

\bibitem{MR1678012}
A.~Boutet~de Monvel and A.~Khorunzhy.
\newblock Asymptotic distribution of smoothed eigenvalue density. {I}.
  {G}aussian random matrices.
\newblock {\em Random Oper. Stochastic Equations}, 7(1):1--22, 1999.

\bibitem{MR1689027}
A.~Boutet~de Monvel and A.~Khorunzhy.
\newblock Asymptotic distribution of smoothed eigenvalue density. {II}.
  {W}igner random matrices.
\newblock {\em Random Oper. Stochastic Equations}, 7(2):149--168, 1999.

\bibitem{MR1176727}
Terence Chan.
\newblock The {W}igner semi-circle law and eigenvalues of matrix-valued
  diffusions.
\newblock {\em Probab. Theory Related Fields}, 93(2):249--272, 1992.

\bibitem{mesoCLTDBM}
Maurice Duits and Kurt Johansson.
\newblock On mesoscopic equilibrium for linear statistics in dyson's brownian
  motion.
\newblock {\em to appear in Mem. Amer. Math. Soc.}, 2013.

\bibitem{MR0148397}
Freeman~J. Dyson.
\newblock A {B}rownian-motion model for the eigenvalues of a random matrix.
\newblock {\em J. Mathematical Phys.}, 3:1191--1198, 1962.

\bibitem{ErdosEdge1}
L\'aszl\'o Erd\H~os, Torbin Kruger, and Dominik Schroder.
\newblock Random matrices with slow correlation decay.
\newblock {\em preprint, arXiv:1705.10661}, 2017.

\bibitem{YauErdosRMT}
Lazlo Erd\H~os, L\'aszl\'o and Horng-Tzer Yau.
\newblock {\em A Dynamical Approach to Random Matrix Theory}.
\newblock 2017.

\bibitem{kevin3}
L.~Erd\H{o}s and K.~Schnelli.
\newblock Universality for random matrix flows with time-dependent density.
\newblock {\em to appear in Ann. Inst. Henri Poincar� Probab. Stat.}, 2016.

\bibitem{MR3098073}
L{\'a}szl{\'o} Erd{\H{o}}s, Antti Knowles, Horng-Tzer Yau, and Jun Yin.
\newblock Spectral statistics of {E}rd{\H o}s-{R}\'enyi graphs {I}: {L}ocal
  semicircle law.
\newblock {\em Ann. Probab.}, 41(3B):2279--2375, 2013.

\bibitem{MR2662426}
L{\'a}szl{\'o} Erd{\H{o}}s, Sandrine P{\'e}ch{\'e}, Jos{\'e}~A. Ram{\'{\i}}rez,
  Benjamin Schlein, and Horng-Tzer Yau.
\newblock Bulk universality for {W}igner matrices.
\newblock {\em Comm. Pure Appl. Math.}, 63(7):895--925, 2010.

\bibitem{MR2661171}
L{\'a}szl{\'o} Erd{\H{o}}s, Jos{\'e} Ram{\'{\i}}rez, Benjamin Schlein, Terence
  Tao, Van Vu, and Horng-Tzer Yau.
\newblock Bulk universality for {W}igner {H}ermitian matrices with
  subexponential decay.
\newblock {\em Math. Res. Lett.}, 17(4):667--674, 2010.

\bibitem{MR2639734}
L{\'a}szl{\'o} Erd{\H{o}}s, Jos{\'e}~A. Ram{\'{\i}}rez, Benjamin Schlein, and
  Horng-Tzer Yau.
\newblock Universality of sine-kernel for {W}igner matrices with a small
  {G}aussian perturbation.
\newblock {\em Electron. J. Probab.}, 15:no. 18, 526--603, 2010.

\bibitem{MR2481753}
L{\'a}szl{\'o} Erd{\H{o}}s, Benjamin Schlein, and Horng-Tzer Yau.
\newblock Local semicircle law and complete delocalization for {W}igner random
  matrices.
\newblock {\em Comm. Math. Phys.}, 287(2):641--655, 2009.

\bibitem{MR2537522}
L{\'a}szl{\'o} Erd{\H{o}}s, Benjamin Schlein, and Horng-Tzer Yau.
\newblock Semicircle law on short scales and delocalization of eigenvectors for
  {W}igner random matrices.
\newblock {\em Ann. Probab.}, 37(3):815--852, 2009.

\bibitem{MR2810797}
L{\'a}szl{\'o} Erd{\H{o}}s, Benjamin Schlein, and Horng-Tzer Yau.
\newblock Universality of random matrices and local relaxation flow.
\newblock {\em Invent. Math.}, 185(1):75--119, 2011.

\bibitem{MR2919197}
L{\'a}szl{\'o} Erd{\H{o}}s, Benjamin Schlein, Horng-Tzer Yau, and Jun Yin.
\newblock The local relaxation flow approach to universality of the local
  statistics for random matrices.
\newblock {\em Ann. Inst. Henri Poincar\'e Probab. Stat.}, 48(1):1--46, 2012.

\bibitem{MR3372074}
L{\'a}szl{\'o} Erd{\H{o}}s and Horng-Tzer Yau.
\newblock Gap universality of generalized {W}igner and {$\beta$}-ensembles.
\newblock {\em J. Eur. Math. Soc. (JEMS)}, 17(8):1927--2036, 2015.

\bibitem{MR2981427}
L{\'a}szl{\'o} Erd{\H{o}}s, Horng-Tzer Yau, and Jun Yin.
\newblock Bulk universality for generalized {W}igner matrices.
\newblock {\em Probab. Theory Related Fields}, 154(1-2):341--407, 2012.

\bibitem{MR2871147}
L{\'a}szl{\'o} Erd{\H{o}}s, Horng-Tzer Yau, and Jun Yin.
\newblock Rigidity of eigenvalues of generalized {W}igner matrices.
\newblock {\em Adv. Math.}, 229(3):1435--1515, 2012.

\bibitem{mesoCLT3}
Yukun He and Antti Knowles.
\newblock Mesoscopic eigenvalue statistics of wigner matrices.
\newblock {\em preprint, arXiv: 1603.01499}, 2016.

\bibitem{HG}
Jiaoyang Huang and Alice Guionnet.
\newblock Rigidity and edge universality of discrete $\beta$-ensemble.
\newblock {\em preprint, arXiv:1705.05527}, 2017.

\bibitem{HL}
Jiaoyang Huang and Benjamin Landon.
\newblock Local law and mesoscopic fluctuations of dyson brownian motion for
  general $\beta$ and potential.
\newblock {\em preprint, arXiv:1612.06306}, 2016.

\bibitem{fix}
Benjamin Landon, Philippe Sosoe, and Horng-Tzer Yau.
\newblock Fixed energy universality of dyson brownian motion.
\newblock {\em preprint, arXiv: 1609.09011}, 2016.

\bibitem{Landon2016}
Benjamin Landon and Horng-Tzer Yau.
\newblock Convergence of local statistics of {D}yson {B}rownian motion.
\newblock {\em to appear in Comm. Math. Phys.}, 2014.

\bibitem{LandonEdge}
Benjamin Landon and Horng-Tzer Yau.
\newblock Edge statistics of {D}yson {B}rownian motion.
\newblock {\em preprint, arXiv:1712.03881}, 2017.

\bibitem{MR3502606}
Ji~Oon Lee, Kevin Schnelli, Ben Stetler, and Horng-Tzer Yau.
\newblock Bulk universality for deformed {W}igner matrices.
\newblock {\em Ann. Probab.}, 44(3):2349--2425, 2016.

\bibitem{MR3161313}
Ji~Oon Lee and Jun Yin.
\newblock A necessary and sufficient condition for edge universality of
  {W}igner matrices.
\newblock {\em Duke Math. J.}, 163(1):117--173, 2014.

\bibitem{GDBM1}
Songzi Li, Xiang-Dong Li, and Yong-Xiao Xie.
\newblock Generalized dyson brownian motion, mckean-vlasov equation and
  eigenvalues of random matrices.
\newblock {\em preprint, arXiv:1303.1240}, 2013.

\bibitem{GDBM2}
Songzi Li, Xiang-Dong Li, and Yong-Xiao Xie.
\newblock On the law of large numbers for the empirical measure process of
  generalized dyson brownian motion.
\newblock {\em preprint, arXiv:1407.7234}, 2015.

\bibitem{mesoCLT1}
Asad Lodhia and Nicholas~J. Simm.
\newblock Mesoscopic linear statistics of wigner matrices.
\newblock {\em preprint, arXiv: 1503.03533}, 2015.

\bibitem{MR1217451}
L.~C.~G. Rogers and Z.~Shi.
\newblock Interacting {B}rownian particles and the {W}igner law.
\newblock {\em Probab. Theory Related Fields}, 95(4):555--570, 1993.

\bibitem{MR3116567}
Philippe Sosoe and Percy Wong.
\newblock Regularity conditions in the {CLT} for linear eigenvalue statistics
  of {W}igner matrices.
\newblock {\em Adv. Math.}, 249:37--87, 2013.

\bibitem{MR2784665}
Terence Tao and Van Vu.
\newblock Random matrices: universality of local eigenvalue statistics.
\newblock {\em Acta Math.}, 206(1):127--204, 2011.

\bibitem{MR3109424}
Terence Tao and Van Vu.
\newblock Random matrices: sharp concentration of eigenvalues.
\newblock {\em Random Matrices Theory Appl.}, 2(3):1350007, 31, 2013.

\bibitem{MR0077805}
Eugene~P. Wigner.
\newblock Characteristic vectors of bordered matrices with infinite dimensions.
\newblock {\em Ann. of Math. (2)}, 62:548--564, 1955.

\bibitem{MR0083848}
Eugene~P. Wigner.
\newblock Characteristic vectors of bordered matrices with infinite dimensions.
  {II}.
\newblock {\em Ann. of Math. (2)}, 65:203--207, 1957.

\end{thebibliography}
\bibliographystyle{plain}

\end{document}